\pdfoutput=1
\documentclass[12pt,a4paper,reqno]{amsart}

\usepackage{amssymb}
\usepackage{amsmath}
\usepackage{amsthm}
\usepackage{thmtools}
\usepackage{mathtools}
\usepackage{mathabx}
\usepackage{mathrsfs}
\usepackage{soul}
\usepackage{bm}
\numberwithin{equation}{section}
\usepackage{framed}
\usepackage[bottom]{footmisc}
\usepackage[utf8]{inputenc}
\usepackage[colorlinks=false,hidelinks]{hyperref}
\usepackage[nameinlink,capitalize]{cleveref}
\usepackage{enumitem}
\usepackage{graphicx}
\usepackage[dvipsnames]{xcolor}
\usepackage{relsize}
\usepackage{exscale}
\usepackage{floatrow}
\usepackage{pdflscape}
\usepackage{afterpage}
\usepackage[justification=centering]{caption}
\usepackage{rotating}
\usepackage{pdflscape}
\usepackage{tabulary}
\usepackage{caption}
\usepackage{subcaption}

\theoremstyle{plain}

\newtheorem{theorem}[subsection]{Theorem}
\newtheorem{proposition}[subsection]{Proposition}
\newtheorem{lemma}[subsection]{Lemma}
\newtheorem{corollary}[subsection]{Corollary}

\newtheorem{remark}[subsection]{Remark}

\theoremstyle{definition}
\newtheorem{definition}[theorem]{Definition}
\newtheorem{example}[theorem]{Example}

\newcommand{\sign}{\text{sign}}

\makeatletter
\newcommand{\vast}{\bBigg@{4}}
\newcommand{\Vast}{\bBigg@{5}}
\makeatother

\makeatletter
\newcommand*\bigcdot{\mathpalette\bigcdot@{.5}}
\newcommand*\bigcdot@[2]{\mathbin{\vcenter{\hbox{\scalebox{#2}{$\m@th#1\bullet$}}}}}
\makeatother

\newcommand*{\medcap}{\mathbin{\scalebox{1.5}{\ensuremath{\cap}}}}

\usepackage[square,numbers]{natbib}
\usepackage{amsaddr}
\usepackage{etoolbox}

\makeatletter
\patchcmd{\@maketitle}
{\ifx\@empty\@dedicatory}
{\ifx\@empty\@date \else {\vskip3ex \centering\@date\par\vskip1ex}\fi
	\ifx\@empty\@dedicatory}
{}{}
\patchcmd{\@maketitle}
{\ifx\@empty\@date\else \@footnotetext{\@setdate}\fi}
{}{}{}
\makeatother

\renewcommand{\leq}{\leqslant}
\renewcommand{\geq}{\geqslant}

\def\vs{\vspace{0.2cm}}
\def\ni{\noindent}
\def\emph#1{{\it #1}}
\def\textbf#1{{\bf #1}}

\usepackage{subcaption}
\begin{document}
	
\title{Departure-based asymptotic stochastic order for random processes}

\author{Sugata Ghosh and Asok K. Nanda}
\address{Department of Mathematics and Statistics\\
	Indian Institute of Science Education and Research Kolkata\\
	Mohanpur 741246, India}
\email{sg18rs017@iiserkol.ac.in}
\email{asok@iiserkol.ac.in}



\allowdisplaybreaks
\raggedbottom

\begin{abstract}
	In \citet{GN_2021_A}, we have introduced the notion of asymptotic stochastic comparison of stochastic processes. In the present work, we propose and analyze a specific asymptotic stochastic order for random processes based on the measure of departure discussed in the literature. As applications, we stochastically compare mixtures of order statistics and record values coming from two different homogeneous samples, as the sample size becomes large.
\end{abstract}

\maketitle
\markleft{\uppercase{Departure-based asymptotic stochastic order}}
\markright{\uppercase{Sugata Ghosh and Asok K. Nanda}}

\section{Introduction}\label{2-section-introduction}

Stochastic orders are extensively studied in the literature in the context of comparing random variables. They have been extended to compare random vectors, but not much exploration has been done in the context of comparing stochastic processes. In the present work, we attempt to compare stochastic processes in an asymptotic sense. Let $X$ and $Y$ be two random variables. $X$ is said to be smaller than $Y$ in {\it usual stochastic order} (denoted by $X \leq_{\textnormal{st}} Y$) if, for every $x \in \mathbb{R}$, we have $P(X>x) \leq P(Y>x)$. We say that $X$ is equal to $Y$ in usual stochastic order (denoted by $X =_{st} Y$) if $X \leq_{\textnormal{st}} Y$ and $Y \leq_{\textnormal{st}} X$. If $F_X$ and $F_Y$ denote the respective cdfs of $X$ and $Y$, then the condition $P(X>x) \leq P(Y>x)$ can be written as $F_X(x) \geq F_Y(x)$ for every $x \in \mathbb{R}$. Let $\left\{x_1,x_2,\ldots,x_n\right\}$ and $\left\{y_1,y_2,\ldots,y_n\right\}$ be two samples taken from two populations. The empirical cdfs based on these two samples are defined respectively by $F_n(x)=n^{-1}\sum_{i=1}^n 1_{(x_i,\infty)}(x)$ and $G_n(x)=n^{-1}\sum_{i=1}^n 1_{(y_i,\infty)}(x)$, for every $x \in \mathbb{R}$, where for any $A \subseteq \mathbb{R}$, the indicator function $1_A:\mathbb{R} \to \{0,1\}$ is given by $1_A(x)=1$ if $x \in A$ and equal to $0$ otherwise. Note that if $\hat{X}$ and $\hat{Y}$ respectively have the empirical cdfs $F_n$ and $G_n$, then $\hat{X} \leq \hat{Y}$ if and only if $x_{i:n} \leq y_{i:n}$ for every $i \in \left\{1,2,\ldots,n\right\}$, where $x_{i:n}$ and $y_{i:n}$ respectively denote the $i$th largest element from the samples $\left\{x_1,x_2,\ldots,x_n\right\}$ and $\left\{y_1,y_2,\ldots,y_n\right\}$. In other words, $\hat{X} \leq_{\textnormal{st}} \hat{Y}$ if and only if $\#\left\{i \in \left\{1,2,\ldots,n\right\} : x_{i:n} \leq y_{i:n}\right\}/n=1$. The ratio $\#\left\{i \in \left\{1,2,\ldots,n\right\} : x_{i:n} \leq y_{i:n}\right\}/n$ is Galton's rank order statistic (see \citet{GH_1968}, \citet{H_1955}), which goes all the way back to a correspondence between Galton and Darwin in 1876, in the context of analyzing data from the latter's experiments on cross and self fertilization of plants (see \citet{D_1878}). It is well-known that usual stochastic order is a partial order.\vs

It is often observed that two cdfs $F_X$ and $F_Y$ cross each other, indicating that usual stochastic order cannot exist between $X$ and $Y$. However, $F_X(x)$ and $F_Y(x)$ may be extremely close in the region where $F_X$ dominates $F_Y$, but differ significantly where $F_Y$ dominates $F_X$ (see \cref{2-fig-t4N01-largest-order-statistics-distribution-functions-ggb} for a situation where such a phenomenon occurs). To address this situation, one possible approach is to consider certain measures based on the region where the defining condition, or a characterizing condition is violated, and the extent of that violation. Following this approach we had defined a number of asymptotic stochastic orders in \citet{GN_2021_A}. An alternative approach, which is the focus of the present paper, is to consider measures of departure from the usual stochastic order. \citet{LL_2002} introduced the notion of almost stochastic dominance in the context of risk aversion and considered the following measure, which can be considered as a measure of departure from the usual stochastic order $X \leq_{\textnormal{st}} Y$.
\begin{equation*}
	\Vert F_Y-F_X \Vert_1^{-1}\int_{B_0} (F_Y(x)-F_X(x)) dx,
\end{equation*}
where $B_0=\{x \in \mathbb{R} : F_X(x)<F_Y(x)\}$ and $\Vert \cdot \Vert_1$ denotes the $L_1$ norm with respect to the Lebegue measure on $\mathbb{R}$. Despite its simplicity, asymptotic properties of this measure are not available in the literature. Recently, \citet{B_2018} proposed and derived the asymptotic properties of a measure of departure from the usual stochastic order $X \leq_{\textnormal{st}} Y$, based on the $L_2$-Wasserstein distance between two cdfs. The measure is motivated by the following characterization of usual stochastic order in terms of quantile functions.\vs

\begin{proposition}[\citet{SS_2007}, p.$5$]
	\label{2-proposition-usual-stochastic-order-quantile}
	Let $X$ and $Y$ be two random variables with respective quantile functions $F_X^{-1}$ and $F_Y^{-1}$. Then, $X \leq_{\textnormal{st}} Y$ if and only if $F_X^{-1}(u) \leq F_Y^{-1}(u)$, for every $u \in (0,1)$. \hfill $\blacksquare$
\end{proposition}\vs

Here $F_X^{-1}$ and $F_Y^{-1}$ denote the respective left continuous inverses of $F_X$ and $F_Y$, i.e. $F_X^{-1}(u)=\inf{\left\{x \in \mathbb{R} : F_X(x) \geq u\right\}}$ and $F_Y^{-1}(u)=\inf{\left\{x \in \mathbb{R} : F_Y(x) \geq u\right\}}$. For more details on characterizations and closure properties of the usual stochastic order, we refer to \citet{SS_2007}. It is noted in \citet{L_1955} that the characterization of usual stochastic order in terms of the quantile functions, as stated in \cref{2-proposition-usual-stochastic-order-quantile}, is more intuitive than its definition. The measure of departure from $X \leq_{\textnormal{st}} Y$ proposed in \citet{B_2018} is given by
\begin{equation}
	\label{2-eq-del-barrio-measure}
	\varepsilon_{\mathcal{W}_2}(F_Y,F_X):=\frac{\int_{A_0} (F_X^{-1}(u)-F_Y^{-1}(u))^2 du}{\mathcal{W}_2^2(F_X,F_Y)},
\end{equation}
where $A_0=\{u \in (0,1) : F_X^{-1}(u)>F_Y^{-1}(u)\}$ and $\mathcal{W}_2(F_X,F_Y)$ is the $L_2$-Wesserstein distance between $F_X$ and $F_Y$, which has the following representation in terms of the corresponding quantile functions $F_X^{-1}$ and $F_Y^{-1}$.
\begin{equation}
	\label{2-eq-wasserstein-quantile}
	\mathcal{W}_2(F_X,F_Y)=\left(\int_0^1 (F_X^{-1}(u)-F_Y^{-1}(u))^2 du\right)^{1/2}.
\end{equation}
It is well-known that if $F_X$ and $F_Y$ both have finite second order moment, then $\mathcal{W}_2(F_X,F_Y)$ is finite. Note that the measure of departure from $X \leq_{\textnormal{st}} Y$ varies between $0$ and $1$, achieving the extremes in the cases $X \leq_{\textnormal{st}} Y$ and $Y \leq_{\textnormal{st}} X$, respectively. When the usual stochastic order does not hold, the measure gives a quantification of stochastic dominance of $X$ over $Y$. In this paper, we introduce the notion of asymptotic stochastic order for stochastic processes, based on this measure of departure.\vs

To motivate the idea, let us consider the following situation. Let $F_X$ and $F_Y$ be two cdfs with finite second order moments. Suppose that the right tail of $F_X$ is heavier than that of $F_Y$ in the sense that there exists $c \in \mathbb{R}$ such that $F_X(x)>F_Y(x)$, whenever $x>c$. Now, if one takes two samples (of same size) of random observations, the first one from $F_X$ and the second from $F_Y$, then one may expect that the largest observation from the first sample is smaller than that from the second sample. To be more precise, let $X_1,X_2,\ldots,X_n$ be a random sample from the standard normal distribution and $Y_1,Y_2,\ldots,Y_n$ be the same from $t$ distribution with $4$ degrees of freedom. Clearly, none of the two distributions dominate the other in usual stochastic order (see \cref{2-fig-t4N01-distribution-functions-ggb}). However, our intuition suggests that the largest order statistics $X_{n:n}$ and $Y_{n:n}$, from the respective samples should reflect the relative tail behaviour of the parent distributions. In particular, $Y_{n:n}$ should dominate $X_{n:n}$ in some stochastic sense. The respective cdfs of $X_{n:n}$ and $Y_{n:n}$ are given by $F_{X_{n:n}}(x)=(F_X(x))^n$ and $F_{Y_{n:n}}(x)=(F_Y(x))^n$ for every $x \in \mathbb{R}$. It follows that $X_{n:n} \leq_{\textnormal{st}} Y_{n:n}$ if and only if $X_1 \leq_{\textnormal{st}} Y_1$. Hence, the usual stochastic order between $X_{n:n}$ and $Y_{n:n}$ does not hold in the case explained. However, observation of the cdfs of $X_{n:n}$ and $Y_{n:n}$ suggests that stochastic order may exist in an approximate sense even when the sample size $n$ is as small as $10$ (see \cref{2-fig-t4N01-largest-order-statistics-distribution-functions-ggb}). 
\begin{figure}
	\centering
	\begin{subfigure}{0.45\textwidth}
		\centering
		\includegraphics[width=\linewidth]{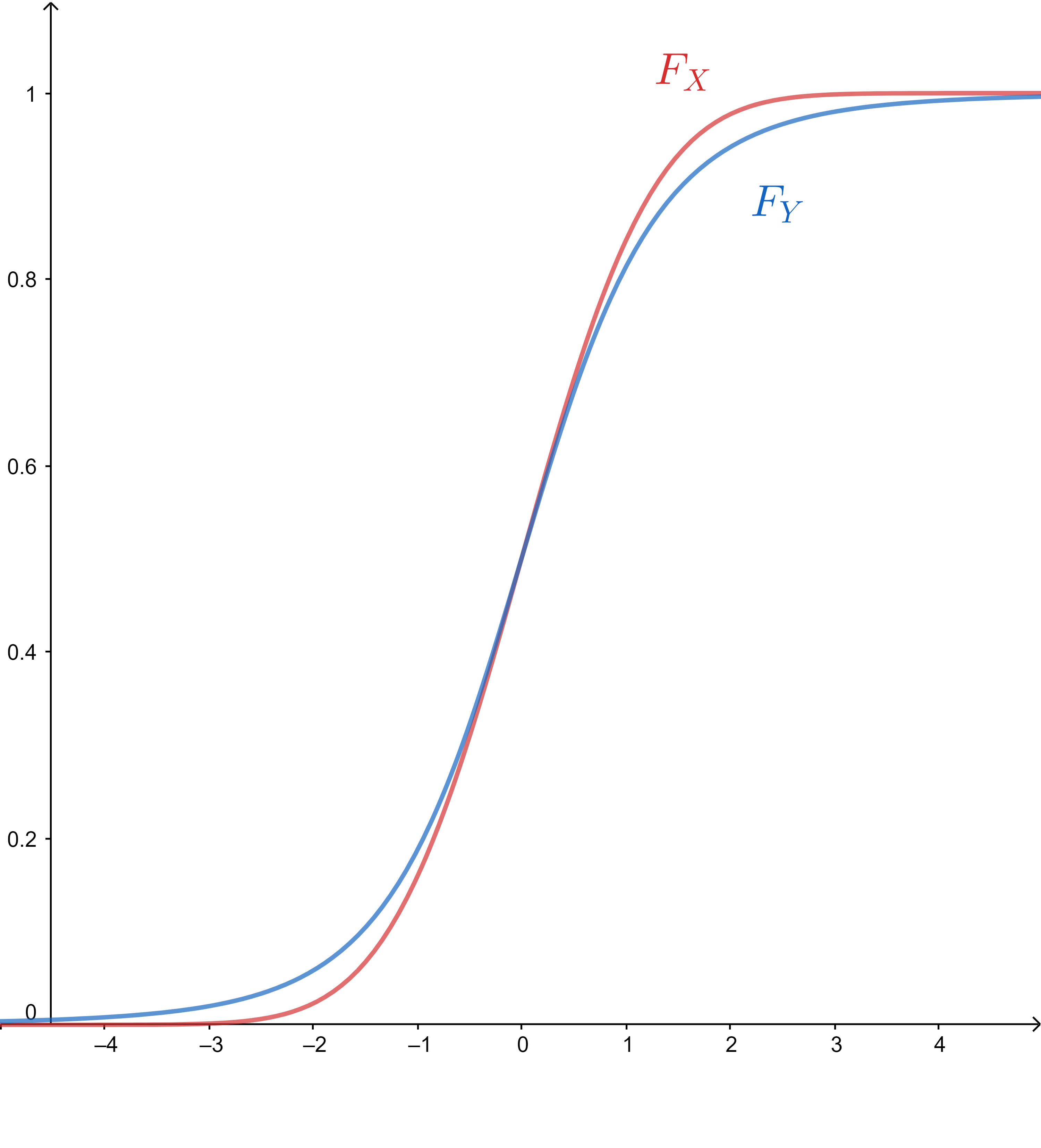}
		\caption{\centering cdfs of $N(0,1)$ denoted by $F_X$ and $t_4$ denoted by $F_Y$}
		\label{2-fig-t4N01-distribution-functions-ggb}
	\end{subfigure}
	\hfill
	\begin{subfigure}{0.45\textwidth}
		\centering
		\includegraphics[width=\linewidth]{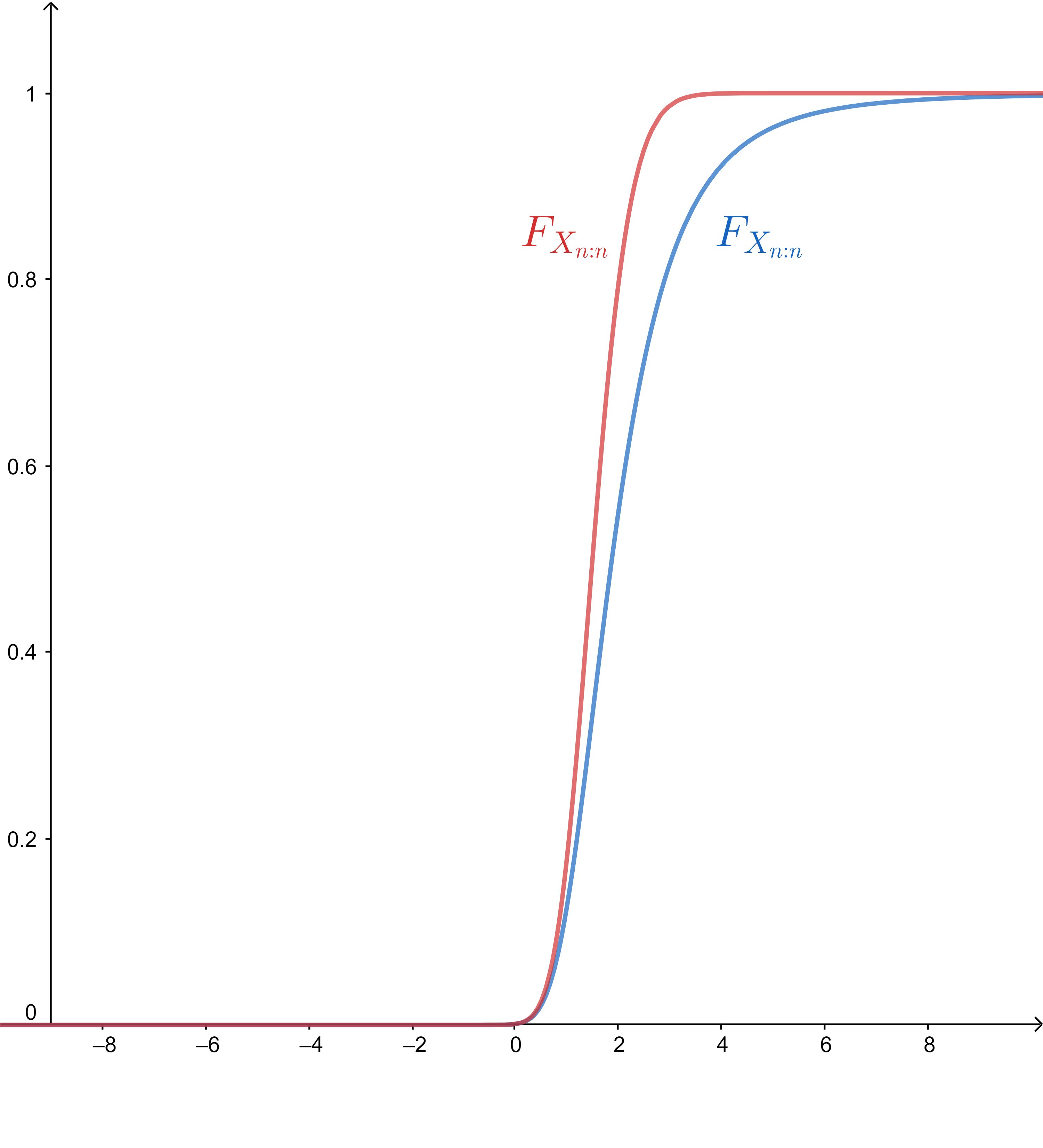}
		\caption{\centering cdfs of largest order statistics from $N(0,1)$ and $t_4$ with $n=10$}
		\label{2-fig-t4N01-largest-order-statistics-distribution-functions-ggb}
	\end{subfigure}
	\caption{}
	\label{2-fig-distributionplot}
\end{figure}
This motivate us to examine stochastic order that may exist between $X_{n:n}$ and $Y_{n:n}$ in some asymptotic sense. In this work, we find out that this is indeed the case, as indicated by the following proposition, which follows from a more general result, \cref{2-theorem-n-gamma-rate}, in \cref{2-section-order-statistics} (see \cref{2-remark}).\vs

\begin{proposition}\label{2-proposition-n-gamma-rate-special-case-largest-order}
	Let $X_1,X_2,\ldots,X_n$ and $Y_1,Y_2,\ldots,Y_n$ be two random samples respectively from distributions with continuous, strictly increasing cdfs $F_X$ and $F_Y$, both of which have finite second order moments. Suppose there exists $c \in \mathbb{R}$ such that $F_X(x)>F_Y(x)$, for every $x>c$.
	Then we have
	\begin{equation}\label{2-eq-n-gamma-rate-special-case-largest-order}
		\varepsilon_{\mathcal{W}_2}(F_{X_{n:n}},F_{Y_{n:n}}) \leq K_{\epsilon} \left(\frac{F_Y(c)}{1-\epsilon}\right)^{n-1},
	\end{equation}
	for arbitrarily small $\epsilon>0$, where $K_{\epsilon}$ is a nonnegative constant, independent of $n$. \hfill $\blacksquare$
\end{proposition}\vs

Note that, since $F_Y(c)<1$, we can choose $\epsilon>0$ such that $F_Y(c)/(1-\epsilon)<1$ and hence the right hand side of \eqref{2-eq-n-gamma-rate-special-case-largest-order} goes to $0$, as $n \to \infty$. Consequently, the measure of departure from $X_{n:n} \leq_{\textnormal{st}} Y_{n:n}$ shrinks to $0$, as $n \to \infty$. In the example with standard normal distribution and $t_4$ distribution, stated above, we have $c=0$ and consequently $F_Y(c)=1/2$. It is interesting to observe that the relative behavior of $F_X$ and $F_Y$ in the region $(-\infty,c)$ has no effect on the result. Building upon this particular observation, we consider the problem of comparing not only other order statistics, but also certain stochastic processes with common unbounded above index set $T \subseteq \mathbb{R}$ as well.\vs

Any function which is nondecreasing with $\phi(0)=0$ and $\phi(1)=1$ is called a distortion function or probability transformation function (see \citet{D_1994}). The general idea is that given a class of distortion functions $\left\{\phi_t : t \in T\right\}$, satisfying some asymptotic properties, if the baseline cdfs $F_X$ and $F_Y$ agree on the desired stochastic order around a finite set of points (dictated by the asymptotic behaviour of the map $t \mapsto \phi_t$), then the stochastic processes $\left\{\phi_t(F_X) : t \in T\right\}$ and $\left\{\phi_t(F_Y) : t \in T\right\}$ admit asymptotic stochastic order (see \cref{2-definition-asymptotic-usual-stochastic-order}). The way the baseline cdfs behave away from the specified finite number of points has no effect on the asymptotic stochastic order (see \cref{2-section-distorted-distributions}).\vs

Throughout the paper, we stick to the following notations and conventions. For any interval $I$, we denote its length by $\lvert I \rvert$ and, for any $A \subseteq \mathbb{R}$, we denote its Lebesgue measure by $l(A)$. For a function $f:B \to \mathbb{R}$, we define the $f$-image of a set $A \subseteq B$ as $f(A)=\{f(x) : x \in A\}$. Throughout the paper, the inverse of a cdf $F$ is defined as the left continuous inverse of $F$, which becomes the unique inverse if $F$ is strictly increasing. For any function $f:[0,1] \to \mathbb{R}$ which is continuously differentiable in $(0,1)$, we define the derivative of $f$ at the endpoints, i.e. $0$ and $1$, respectively by $f'(0)=\lim_{u \to 0+} f'(u)$ and $f'(1)=\lim_{u \to 1-} f'(u)$. For an arbitrary set $T \subseteq \mathbb{R}$ which is unbounded above, the limit of a real-valued function $t \mapsto \psi(t)$ defined on $T$, as $t \to \infty$, is considered in the usual sense, i.e. $\lim_{t \to \infty} \psi(t) \equiv \lim_{t \to \infty;\,t \,\in\, T} \psi(t)=a \in \mathbb{R} \cup \{-\infty,\infty\}$ if, for every sequence $\left\{t_n \in T : n \in \mathbb{N}\right\}$ with $t_n \to \infty$, as $n \to \infty$, we have $\lim_{n \to \infty} \psi(t_n)=a$. We follow the notations listed below throughout. Let $Z$ be a random variable. We denote the cdf of $Z$ by $F_Z$ and the quantile function of $Z$ by $F_Z^{-1}$, which is the left-continuous inverse of $F_Z$. Let $X$ and $Y$ be two random variables. Then we denote
\begin{align*}
	&A_0=\{u \in (0,1) : F_X^{-1}(u)>F_Y^{-1}(u)\},\\
	&A_1=\{u \in (0,1) : F_X^{-1}(u)<F_Y^{-1}(u)\},\\
	&A_2=\{u \in (0,1) : F_X^{-1}(u) \neq F_Y^{-1}(u)\},\\
	&B_0=\{x \in \mathbb{R} : F_X(x)<F_Y(x)\},\\
	&B_1=\{x \in \mathbb{R} : F_X(x)>F_Y(x)\},\\
	&B_2=\{x \in \mathbb{R} : F_X(x) \neq F_Y(x)\}.
\end{align*}

The rest of the paper is structured as follows. In \cref{2-section-asymptotic}, we define the departure-based asymptotic stochastic order and discuss its properties. In \cref{2-section-order-statistics}, we derive sufficient conditions for departure-based asymptotic stochastic order between $1+\left[(n-1)\gamma_n\right]$th order statistics from two different homogeneous samples, where one may choose the sequence $\{\gamma_n\}$ which converges to $\gamma \in [0,1]$ appropriately to take any extreme or central order statistic into account. In \cref{2-section-distorted-distributions}, we have extended the idea to stochastically compare certain stochastic processes in an asymptotic sense. In \cref{2-section-applications}, we apply the concept to make asymptotic stochastic comparison of mixtures of order statistics as well as record values from two different homogeneous samples.\vs

\section{Departure-based asymptotic stochastic order}\label{2-section-asymptotic}

Let $\left\{X_t : t \in T\right\}$ and $\left\{Y_t : t \in T\right\}$ be two stochastic processes, where $T$ ($\subseteq \mathbb{R}$) is unbounded above and the state space is $\mathbb{R}$. Note that when $T=\mathbb{N}$ (resp. $T=[0,\infty)$), the processes become discrete-time (resp. continuous-time) stochastic processes. Now it may happen that there does not exist any $t \in T$ such that $X_t \leq_{\textnormal{st}} Y_t$ but, as $t \to \infty$, it increasingly gets closer to $X_t \leq_{\textnormal{st}} Y_t$, in some sense. The measure in \eqref{2-eq-del-barrio-measure} gives us a way to mathematically address such a situation. Observe that if $\mathcal{W}_2(F_X,F_Y)=0$, then $\varepsilon_{\mathcal{W}_2}(F_X,F_Y)$ has $0/0$ form, and hence is undefined. To reflect the fact that $X \leq_{\textnormal{st}} X$ for any random variable $X$, we use the convention that $\varepsilon_{\mathcal{W}_2}(F_X,F_Y)=0$ if $\mathcal{W}_2(F_X,F_Y)=0$. Again, observe that if $\int_{A_0}(F_X^{-1}(u)-F_Y^{-1}(u))^2\,du=\infty$, then $\varepsilon_{\mathcal{W}_2}(F_X,F_Y)$ has $\infty/\infty$ form, and hence is undefined. In this case, we use the convention that $\varepsilon_{\mathcal{W}_2}(F_X,F_Y)=1$.

\begin{definition}
	\label{2-definition-asymptotic-usual-stochastic-order}
	Let $\left\{X_t : t \in T\right\}$ and $\left\{Y_t : t \in T\right\}$ be two stochastic processes, with respective classes of cdfs $\left\{F_{X_t} : t \in T\right\}$ and $\left\{F_{Y_t} : t \in T\right\}$. We say that $X_t$ is smaller than $Y_t$ in departure-based asymptotic stochastic order, denoted by $X_t \leq_{\textnormal{d-ast}} Y_t$, as $t \to \infty$, if
	\begin{equation}\label{2-eq-definition-asymptotic-usual-stochastic-order}
		\lim_{t \to \infty} \varepsilon_{\mathcal{W}_2}(F_{X_t},F_{Y_t})=0.
	\end{equation}
	Also, we say that $X_t$ is asymptotically equal to $Y_t$ in usual stochastic order, denoted by $X_t =_{\text{ast}} Y_t$, as $t \to \infty$, if both $X_t \leq_{\textnormal{d-ast}} Y_t$ and $Y_t \leq_{\textnormal{d-ast}} X_t$ hold true as $t \to \infty$. \hfill $\blacksquare$
\end{definition}\vs

Now we discuss some properties of the departure-based asymptotic stochastic order.
\begin{proposition}\label{2-proposition-equality-asymptotic-usual-stochastic-order}
	If $X_t =_{\textnormal{ast}} Y_t$, as $t \to \infty$, then $\lim_{t \to \infty} \mathcal{W}_2(F_{X_t},F_{Y_t})=0$.
\end{proposition}

\begin{proof}
	Suppose, for the sake of contradiction, that $\mathcal{W}_2\left(F_{X_t},F_{Y_t}\right) \nrightarrow 0$ as $t \to \infty$. So, there must exist a sequence $\left\{t_n \in T : n \in \mathbb{N}\right\}$ with $t_n \to \infty$, as $n \to \infty$ such that $\mathcal{W}_2\left(F_{X_{t_n}},F_{Y_{t_n}}\right)$ does not go to $0$, as $n \to \infty$. Thus, there exists $\epsilon>0$ such that, for every $N \in \mathbb{N}$, there exists $n \geq N$ satisfying $\mathcal{W}_2\left(F_{X_{t_n}},F_{Y_{t_n}}\right)>\epsilon$. It follows that there exists a subsequence $\left\{n_m : m \in \mathbb{N}\right\}$ such that $\mathcal{W}_2\left(F_{X_{t_{n_m}}},F_{Y_{t_{n_m}}}\right)>\epsilon$ for every $m \in \mathbb{N}$, i.e. $\left\{\mathcal{W}_2\left(F_{X_{t_{n_m}}},F_{Y_{t_{n_m}}}\right) : m \in \mathbb{N}\right\}$ is bounded away from $0$. Then, we have $\varepsilon_{\mathcal{W}_2}\left(F_{X_{t_{n_m}}},F_{Y_{t_{n_m}}}\right)+\varepsilon_{\mathcal{W}_2}\left(F_{Y_{t_{n_m}}},F_{X_{t_{n_m}}}\right)=1$, for every $m \in \mathbb{N}$, and hence
	\begin{equation*}
		\label{2-eq-equality-asymptotic-usual-stochastic-order}
		\lim_{m \to \infty} \varepsilon_{\mathcal{W}_2}\left(F_{X_{t_{n_m}}},F_{Y_{t_{n_m}}}\right)+\lim_{m \to \infty} \varepsilon_{\mathcal{W}_2}\left(F_{Y_{t_{n_m}}},F_{X_{t_{n_m}}}\right)=1,
	\end{equation*}
	which contradicts the hypothesis. Hence, we must have $\lim_{t \to \infty} \mathcal{W}_2\left(F_{X_t},F_{Y_t}\right)=0$.
\end{proof}\vs

Let $\left\{X_t : t \in T\right\}$ be a stochastic process. Then $X_t \leq_{\textnormal{d-ast}} X_t$, as $t \to \infty$, i.e. departure-based asymptotic stochastic order is {\it reflexive}. To see this, note that, for every $t \in T$, we have $\mathcal{W}_2(F_{X_t},F_{X_t})=0$ and hence $\lim_{t \to \infty} \varepsilon_{\mathcal{W}_2}(F_{X_t},F_{X_t})=0$, giving $X_t \leq_{\textnormal{d-ast}} X_t$, as $t \to \infty$. Again, departure-based asymptotic stochastic order is {\it antisymmetric} by definition, i.e., for any two stochastic processes $\left\{X_t : t \in T\right\}$ and $\left\{Y_t : t \in T\right\}$, $X_t \leq_{\textnormal{d-ast}} Y_t$, as $t \to \infty$ and $Y_t \leq_{\textnormal{d-ast}} X_t$, as $t \to \infty$ imply $X_t =_{\textnormal{ast}} Y_t$, as $t \to \infty$. The {\it transitivity} property, however, is not so straightforward and requires additional conditions, as shown in the next result.
\begin{theorem}\label{2-theorem-transitivity-asymptotic-usual-stochastic-order}
	Let $\left\{X_t : t \in T\right\}$, $\left\{Y_t : t \in T\right\}$ and $\left\{Z_t : t \in T\right\}$ be three stochastic processes such that $X_t \leq_{\textnormal{d-ast}} Y_t$ and $Y_t \leq_{\textnormal{d-ast}} Z_t$. Then, we have $X_t \leq_{\textnormal{d-ast}} Z_t$, provided
	\[
	\mathcal{W}_2^2\left(F_{X_t},F_{Y_t}\right)=O\left(\mathcal{W}_2^2\left(F_{X_t},F_{Z_t}\right)\right) \text{and } \mathcal{W}_2^2\left(F_{Y_t},F_{Z_t}\right)=O\left(\mathcal{W}_2^2\left(F_{X_t},F_{Z_t}\right)\right).
	\]
\end{theorem}

\begin{proof}
	Let $A_{0,t}=\{u \in (0,1) : F_{X_t}^{-1}(u)>F_{Y_t}^{-1}(u)\}$, $B_{0,t}=\{u \in (0,1) : F_{Y_t}^{-1}(u)>F_{Z_t}^{-1}(u)\}$ and $C_{0,t}=\{u \in (0,1) : F_{X_t}^{-1}(u)>F_{Z_t}^{-1}(u)\}$. Then, it is easy to see that $C_{0,t} \subseteq A_{0,t} \cup B_{0,t}$. Note that $C_{0,t}=C_{0,t} \medcap (A_{0,t} \cup B_{0,t})=(A_{0,t} \medcap B_{0,t}^c \medcap C_{0,t}) \cup (A_{0,t}^c \medcap B_{0,t} \medcap C_{0,t}) \cup (A_{0,t} \medcap B_{0,t} \medcap C_{0,t})$. Hence
	\begin{align}\label{2-eq-C-0-t-decomposition}
		\int_{C_{0,t}} (F_{Z_t}^{-1}(u)-F_{X_t}^{-1}(u))^2du &= \int_{A_{0,t} \cap B_{0,t}^c \cap C_{0,t}} (F_{Z_t}^{-1}(u)-F_{X_t}^{-1}(u))^2du\nonumber\\
		&\,\,\,\,\,+\int_{A_{0,t}^c \cap B_{0,t} \cap C_{0,t}} (F_{Z_t}^{-1}(u)-F_{X_t}^{-1}(u))^2du\nonumber\\
		&\,\,\,\,\,+\int_{A_{0,t} \cap B_{0,t} \cap C_{0,t}} (F_{Z_t}^{-1}(u)-F_{X_t}^{-1}(u))^2du.
	\end{align}
	If $u \in A_{0,t} \medcap B_{0,t}^c \medcap C_{0,t}$, then $F_{X_t}^{-1}(u) > F_{Z_t}^{-1}(u) \geq F_{Y_t}^{-1}(u)$ and hence
	\begin{equation*}
		\int\limits_{A_{0,t} \cap B_{0,t}^c \cap C_{0,t}} (F_{X_t}^{-1}(u)-F_{Z_t}^{-1}(u))^2du \leq \int\limits_{A_{0,t}} (F_{X_t}^{-1}(u)-F_{Y_t}^{-1}(u))^2du.
	\end{equation*}
	Similarly, if $u \in A_{0,t}^c \medcap B_{0,t} \medcap C_{0,t}$, then
	\begin{equation*}
		\int\limits_{A_{0,t}^c \cap B_{0,t} \cap C_{0,t}} (F_{X_t}^{-1}(u)-F_{Z_t}^{-1}(u))^2du \leq \int\limits_{B_{0,t}} (F_{Y_t}^{-1}(u)-F_{Z_t}^{-1}(u))^2du
	\end{equation*}
	and if $u \in A_{0,t} \medcap B_{0,t} \medcap C_{0,t}$, then
	\begin{equation*}
		\int_{A_{0,t} \cap B_{0,t} \cap C_{0,t}} (F_{X_t}^{-1}(u)-F_{Z_t}^{-1}(u))^2du \leq 2\left\{\int_{A_{0,t}} (F_{X_t}^{-1}(u)-F_{Y_t}^{-1}(u))^2du+\int_{B_{0,t}} (F_{Y_t}^{-1}(u)-F_{Z_t}^{-1}(u))^2du\right\}.
	\end{equation*}
	Then, using \eqref{2-eq-C-0-t-decomposition} and the above upper bounds, we have
	\begin{equation}
		\label{2-eq-int-C_n-3-int-B_n-int-A_n}
		\int_{C_{0,t}} (F_{X_t}^{-1}(u)-F_{Z_t}^{-1}(u))^2du \leq 3\left\{\int_{B_{0,t}} (F_{Y_t}^{-1}(u)-F_{Z_t}^{-1}(u))^2du+\int_{A_{0,t}} (F_{X_t}^{-1}(u)-F_{Y_t}^{-1}(u))^2du\right\}.
	\end{equation}
	Since, by the hypothesis, there exists $C_1>0$ and $C_2>0$ such that $\mathcal{W}_2^2(F_{X_t},F_{Y_t}) \leq C_1\,\mathcal{W}_2^2(F_{X_t},F_{Z_t})$ and $\mathcal{W}_2^2(F_{Y_t},F_{Z_t}) \leq C_2\,\mathcal{W}_2^2(F_{X_t},F_{Z_t})$, for every $t \geq t_0$, for some $t_0 \in T$. Let us fix $t \geq t_0$. We consider the following four cases.\vs
	
	\ni{\bf Case 1.} $\mathcal{W}_2(F_{X_t},F_{Y_t})=0$ and $\mathcal{W}_2(F_{Y_t},F_{Z_t})=0$. Then, by nonnegativity and triangle inequality, $0 \leq \mathcal{W}_2(F_{X_t},F_{Z_t}) \leq \mathcal{W}_2(F_{X_t},F_{Y_t})+\mathcal{W}_2(F_{Y_t},F_{Z_t})=0$, i.e. $\mathcal{W}_2(F_{X_t},F_{Z_t})=0$. Thus by convention, $\varepsilon_{\mathcal{W}_2}(F_{X_t},F_{Z_t})=0$.\vs
	
	\ni{\bf Case 2.} $\mathcal{W}_2(F_{X_t},F_{Y_t})=0$ and $\mathcal{W}_2(F_{Y_t},F_{Z_t})>0$. In this case,
	\[
	\mathcal{W}_2(F_{X_t},F_{Z_t}) \geq \left\vert \mathcal{W}_2(F_{X_t},F_{Y_t})-\mathcal{W}_2(F_{Y_t},F_{Z_t}) \right\vert=\mathcal{W}_2(F_{Y_t},F_{Z_t})>0.
	\]
	Using \eqref{2-eq-int-C_n-3-int-B_n-int-A_n}, we get
	\begin{align*}
	\int_{C_{0,t}} (F_{X_t}^{-1}(u)-F_{Z_t}^{-1}(u))^2du &\leq 3\int_{B_{0,t}} (F_{Y_t}^{-1}(u)-F_{Z_t}^{-1}(u))^2du\\
	&=3\varepsilon_{\mathcal{W}_2}(F_{Y_t},F_{Z_t})\mathcal{W}_2^2(F_{Y_t},F_{Z_t})\\
	&\leq 3C_2\,\varepsilon_{\mathcal{W}_2}(F_{Y_t},F_{Z_t})\mathcal{W}_2^2(F_{X_t},F_{Z_t}).
	\end{align*}
	Dividing the extreme sides by $\mathcal{W}_2^2(F_{X_t},F_{Z_t})$, we obtain $\varepsilon_{\mathcal{W}_2}(F_{X_t},F_{Z_t}) \leq 3C_2\,\varepsilon_{\mathcal{W}_2}(F_{Y_t},F_{Z_t})$.\vs
	
	\ni {\bf Case 3.} $\mathcal{W}_2(F_{X_t},F_{Y_t})>0$ and $\mathcal{W}_2(F_{Y_t},F_{Z_t})=0$. In this case,
	\[
	\mathcal{W}_2(F_{X_t},F_{Z_t}) \geq \left\vert \mathcal{W}_2(F_{X_t},F_{Y_t})-\mathcal{W}_2(F_{Y_t},F_{Z_t}) \right\vert=\mathcal{W}_2(F_{X_t},F_{Y_t})>0.
	\]
	Using \eqref{2-eq-int-C_n-3-int-B_n-int-A_n}, we get
	\begin{align*}
	\int_{C_{0,t}} (F_{X_t}^{-1}(u)-F_{Z_t}^{-1}(u))^2du &\leq 3\int_{A_{0,t}} (F_{X_t}^{-1}(u)-F_{Y_t}^{-1}(u))^2du\\
	&=3\varepsilon_{\mathcal{W}_2}(F_{X_t},F_{Y_t})\mathcal{W}_2^2(F_{X_t},F_{Y_t})\\
	&\leq 3C_1\,\varepsilon_{\mathcal{W}_2}(F_{X_t},F_{Y_t})\mathcal{W}_2^2(F_{X_t},F_{Z_t}).
	\end{align*}
	Dividing the extreme sides by $\mathcal{W}_2^2(F_{X_t},F_{Z_t})$, we obtain $\varepsilon_{\mathcal{W}_2}(F_{X_t},F_{Z_t}) \leq 3C_1\,\varepsilon_{\mathcal{W}_2}(F_{X_t},F_{Y_t})$.\vs
	
	\ni{\bf Case 4.} $\mathcal{W}_2(F_{X_t},F_{Y_t})>0$ and $\mathcal{W}_2(F_{Y_t},F_{Z_t})>0$. In this case, $\mathcal{W}_2(F_{X_t},F_{Z_t})$ may or may not be $0$. If $\mathcal{W}_2(F_{X_t},F_{Z_t})=0$, then $\varepsilon_{\mathcal{W}_2}(F_{X_t},F_{Z_t})=0$. Now, let us consider the case where $\mathcal{W}_2(F_{X_t},F_{Z_t})>0$. Here, \eqref{2-eq-int-C_n-3-int-B_n-int-A_n} gives
	\begin{align*}
		\int_{C_{0,t}} (F_{X_t}^{-1}(u)-F_{Z_t}^{-1}(u))^2du &\leq 3\left\{\varepsilon_{\mathcal{W}_2}(F_{Y_t},F_{Z_t})\mathcal{W}_2^2(F_{Y_t},F_{Z_t})+\varepsilon_{\mathcal{W}_2}(F_{X_t},F_{Y_t})\mathcal{W}_2^2(F_{X_t},F_{Y_t})\right\}\\
		&\leq 3\,\mathcal{W}_2^2(F_{X_t},F_{Z_t})\left\{C_2\,\varepsilon_{\mathcal{W}_2}(F_{Y_t},F_{Z_t})+C_1\,\varepsilon_{\mathcal{W}_2}(F_{X_t},F_{Y_t})\right\}.
	\end{align*}
	Dividing both sides by $\mathcal{W}_2^2(F_{X_t},F_{Z_t})$, we obtain
	\[
	\varepsilon_{\mathcal{W}_2}(F_{X_t},F_{Z_t}) \leq 3\,\left\{C_2\,\varepsilon_{\mathcal{W}_2}(F_{Y_t},F_{Z_t})+C_1\,\varepsilon_{\mathcal{W}_2}(F_{X_t},F_{Y_t})\right\}.
	\]
	Since $X_t \leq_{\textnormal{d-ast}} Y_t$ and $Y_t \leq_{\textnormal{d-ast}} Z_t$, we respectively have $\lim_{t \to \infty} \varepsilon_{\mathcal{W}_2}(F_{X_t},F_{Y_t})=0$ and $\lim_{t \to \infty} \varepsilon_{\mathcal{W}_2}(F_{Y_t},F_{Z_t})=0$. Thus, given $\epsilon>0$, there exists $t_1,t_2 \in T$ such that $t \geq t_1 \Rightarrow \varepsilon_{\mathcal{W}_2}(F_{X_t},F_{Y_t})<\epsilon/(6C_1)$ and $t \geq t_2 \Rightarrow \varepsilon_{\mathcal{W}_2}(F_{Y_t},F_{Z_t})<\epsilon/(6C_2)$. Combining all the cases, we see that $t \geq \max{\left\{t_0,t_1,t_2\right\}} \Rightarrow \varepsilon_{\mathcal{W}_2}(F_{X_t},F_{Z_t})<\epsilon$. Since $\epsilon>0$ is arbitrarily chosen, we have $\lim_{t \to \infty} \varepsilon_{\mathcal{W}_2}(F_{X_t},F_{Z_t}) \leq 0$. The reverse inequality follows from the fact that $\varepsilon_{\mathcal{W}_2}(F_{X_t},F_{Z_t}) \geq 0$, for every $t \in T$. Hence the proof is established.
\end{proof}\vs

\begin{remark}
The two conditions $\mathcal{W}_2^2(F_{X_t},F_{Y_t})=O(\mathcal{W}_2^2(F_{X_t},F_{Z_t}))$ and $\mathcal{W}_2^2(F_{Y_t},F_{Z_t})=O(\mathcal{W}_2^2(F_{X_t},F_{Z_t}))$ prevent the $\mathcal{W}_2$-distance between $X_t$ and $Z_t$ to diminish rapidly compared to the same between $X_t$ and $Y_t$, and also the same between $Y_t$ and $Z_t$. For instance, these conditions may get violated if $X_t$ and $Z_t$ both converge to the same random variable $X$ in $\mathcal{W}_2$ sense (i.e. $\lim_{t \to \infty} \mathcal{W}_2(F_{X_t},F_X)=\lim_{t \to \infty} \mathcal{W}_2(F_{Z_t},F_X)=0$), whereas $Y_t$ converges to a different random variable $Y$ in $\mathcal{W}_2$ sense. \hfill $\blacksquare$
\end{remark}\vs

The next result establishes that the measure defined in \eqref{2-eq-del-barrio-measure} is location invariant, but depends on the sign of the scaling parameter. The proof follows from straightforward calculations and hence omitted.
\begin{proposition}
	\label{proposition-wasserstein-2-location-scale}
	Let $X$ and $Y$ be two continuous random variables with respective distribution functions $F_X$ and $F_Y$. Assume that both $X$ and $Y$ have finite second order moment. Let us denote, for $a \in \mathbb{R}$ and $b \neq 0$, the respective distribution functions of $a+bX$ and $a+bY$ by $F_{a+bX}$ and $F_{a+bY}$. Then we have
	\[
	\varepsilon_{\mathcal{W}_2}\left(F_{a+bX},F_{a+bY}\right)=\begin{cases*}
		\varepsilon_{\mathcal{W}_2}\left(F_X,F_Y\right) & if $b>0$,\\
		\varepsilon_{\mathcal{W}_2}\left(F_Y,F_X\right) & if $b<0$.
	\end{cases*} \tag*{$\blacksquare$}
	\]
\end{proposition}\vs

Choosing $a=0$ and $b=-1$, we immediately have the following corollary.
\begin{corollary}\label{corollary-wasserstein-2-location-scale}
	We have $\varepsilon_{\mathcal{W}_2}\left(F_{-Y},F_{-X}\right)=\varepsilon_{\mathcal{W}_2}\left(F_X,F_Y\right)$.
\end{corollary}\vs

\begin{theorem}
	\label{2-theorem-increasing-function-asymptotic-usual-stochastic-order}
	Let $\left\{X_t : t \in T\right\}$ and $\left\{Y_t : t \in T\right\}$ be two stochastic processes such that $X_t \leq_{\textnormal{d-ast}} Y_t$, as $t \to \infty$. Also let $\psi:\mathbb{R} \to \mathbb{R}$ be a strictly increasing, Lipschitz continuous function such that both $\psi(X_t)$ and $\psi(Y_t)$ have finite second order moments for every $t \in T$. Then, $\psi(X_t) \leq_{\textnormal{d-ast}} \psi(Y_t)$, provided $\mathcal{W}_2^2(F_{X_t},F_{Y_t})=O(\mathcal{W}_2^2(F_{\psi(X_t)},F_{\psi(Y_t)}))$.
\end{theorem}

\begin{proof}
	Using strict increasingness of $\psi$, we find the distribution function of $\psi(X_t)$ to be $F_{\psi(X_t)}(x)=P\left\{\psi(X_t) \leq x\right\}=P\left\{X_t \leq \psi^{-1}(x)\right\}=F_{X_t} \circ \psi^{-1}(x)$ for every $x \in \mathbb{R}$. Thus, $F_{\psi(X_t)}^{-1}(u)=\psi \circ F_{X_t}^{-1}(u)$, for every $u \in [0,1]$. Similarly, $F_{\psi(Y_t)}^{-1}(u)=\psi \circ F_{Y_t}^{-1}(u)$, for every $u \in [0,1]$. Hence, by strict increasingness of $\psi$, we have
	\[
	\big\{u \in (0,1) : F_{\psi(X_t)}^{-1}(u)>F_{\psi(Y_t)}^{-1}(u)\big\}=\big\{u \in (0,1) : F_{X_t}^{-1}(u)>F_{Y_t}^{-1}(u)\big\}.
	\]
	Since $\psi$ is Lipshcitz continuous, there exists $K>0$, such that $\left\vert \psi(x)-\psi(y) \right\vert \leq K \left\vert x-y \right\vert$, for every $x,y \in \mathbb{R}$. Now
	\begin{align}\label{2-eq-theorem-increasing-function-numerator}
		&\phantom{\,\,\,\,\,\,\,\,}\int_{\left\{u \in (0,1) : F_{\psi(X_t)}^{-1}(u)>F_{\psi(Y_t)}^{-1}(u)\right\}} (F_{\psi(X_t)}^{-1}(u)-F_{\psi(Y_t)}^{-1}(u))^2du\nonumber\\
		&=\int_{\left\{u \in (0,1) : F_{X_t}^{-1}(u)>F_{Y_t}^{-1}(u)\right\}} (\psi(F_{X_t}^{-1}(u))-\psi(F_{Y_t}^{-1}(u)))^2du\nonumber\\
		&\leq K^2 \int_{\left\{u \in (0,1) : F_{X_t}^{-1}(u)>F_{Y_t}^{-1}(u)\right\}} (F_{X_t}^{-1}(u)-F_{Y_t}^{-1}(u))^2du\nonumber\\
		&\leq K^2 \mathcal{W}_2^2(F_{X_t},F_{Y_t})=0,
	\end{align}
	if $\mathcal{W}_2^2(F_{X_t},F_{Y_t})=0$. Now, if $\mathcal{W}_2(F_{\psi(X_t)},F_{\psi(Y_t)})=0$, then by convention $\varepsilon_{\mathcal{W}_2}(F_{\psi(X_t)},F_{\psi(Y_t)})=0$. Again, if $\mathcal{W}_2(F_{\psi(X_t)},F_{\psi(Y_t)})>0$, then from \eqref{2-eq-theorem-increasing-function-numerator}, we have $\varepsilon_{\mathcal{W}_2}(F_{\psi(X_t)},F_{\psi(Y_t)}) \leq 0$. Since $\varepsilon_{\mathcal{W}_2}(F_{\psi(X_t)},F_{\psi(Y_t)})$ is nonnegative for every $t \in T$, we obtain $\varepsilon_{\mathcal{W}_2}(F_{\psi(X_t)},F_{\psi(Y_t)})=0$. On the other hand, if $\mathcal{W}_2^2(F_{X_t},F_{Y_t})>0$, then
	\[
	\int_{\left\{u \in (0,1) : F_{\psi(X_t)}^{-1}(u)>F_{\psi(Y_t)}^{-1}(u)\right\}} (F_{\psi(X_t)}^{-1}(u)-F_{\psi(Y_t)}^{-1}(u))^2du \leq K^2 \varepsilon_{\mathcal{W}_2}(F_{X_t},F_{Y_t}) \mathcal{W}_2^2(F_{X_t},F_{Y_t}).
	\]
	Since $\mathcal{W}_2^2(F_{X_t},F_{Y_t})=O(\mathcal{W}_2^2(F_{\psi(X_t)},F_{\psi(Y_t)}))$, there exists $C>0$ such that $\mathcal{W}_2^2(F_{X_t},F_{Y_t}) \leq C\mathcal{W}_2^2(F_{\psi(X_t)},F_{\psi(Y_t)})$, for every $t \geq t_0$, for some $t_0 \in T$. Thus, for every $t \geq t_0$,
	\[
	\int_{\left\{u \in (0,1) : F_{\psi(X_t)}^{-1}(u)>F_{\psi(Y_t)}^{-1}(u)\right\}} (F_{\psi(X_t)}^{-1}(u)-F_{\psi(Y_t)}^{-1}(u))^2du \leq CK^2 \varepsilon_{\mathcal{W}_2}(F_{X_t},F_{Y_t}) \mathcal{W}_2^2(F_{\psi(X_t)},F_{\psi(Y_t)}).
	\]
	Now, if $\mathcal{W}_2(F_{\psi(X_t)},F_{\psi(Y_t)})=0$, then by convention $\varepsilon_{\mathcal{W}_2}(F_{\psi(X_t)},F_{\psi(Y_t)})=0$. Again, if $\mathcal{W}_2(F_{\psi(X_t)},F_{\psi(Y_t)})>0$, then we have $\varepsilon_{\mathcal{W}_2}(F_{\psi(X_t)},F_{\psi(Y_t)}) \leq CK^2 \varepsilon_{\mathcal{W}_2}(F_{Y_t},F_{X_t})$. Since $X_t \leq_{\textnormal{d-ast}} Y_t$, we have $\lim_{t \to \infty} \varepsilon_{\mathcal{W}_2}(F_{X_t},F_{Y_t})=0$. Thus, $\lim_{t \to \infty} \varepsilon_{\mathcal{W}_2}(F_{\psi(X_t)},F_{\psi(Y_t)}) \leq 0$. Again, since $\varepsilon_{\mathcal{W}_2}(F_{\psi(X_t)},F_{\psi(Y_t)})$ is nonnegative for every $t \in T$, we have the reverse inequality, which completes the proof.
\end{proof}\vs

Noting that if $\psi$ is strictly decreasing, then $-\psi$ is strictly increasing and using \cref{corollary-wasserstein-2-location-scale}, we have the following corollary.

\begin{corollary}
	\label{2-corollary-increasing-function-asymptotic-usual-stochastic-order}
	Let $\left\{X_t : t \in T\right\}$ and $\left\{Y_t : t \in T\right\}$ be two stochastic processes such that $X_t \leq_{\textnormal{d-ast}} Y_t$, as $t \to \infty$. Also let $\psi:\mathbb{R} \to \mathbb{R}$ be a strictly decreasing, Lipschitz continuous function such that both $\psi(X_t)$ and $\psi(Y_t)$ have finite second order moments for every $t \in T$. Then, $\psi(Y_t) \leq_{\textnormal{d-ast}} \psi(X_t)$, provided $\mathcal{W}_2^2(F_{X_t},F_{Y_t})=O(\mathcal{W}_2^2(F_{\psi(X_t)},F_{\psi(Y_t)}))$.
\end{corollary}\vs

\begin{remark}
	The condition $\mathcal{W}_2^2(F_{X_t},F_{Y_t})=O(\mathcal{W}_2^2(\psi(F_{X_t}),\psi(F_{Y_t})))$ essentially prevents the $\mathcal{W}_2$-distance between $\psi(F_{X_t})$ and $\psi(F_{Y_t})$ to diminish rapidly compared to the same between $X_t$ and $Y_t$, as $t \to \infty$. In particular, the condition holds if $\psi$ induces a location-scale transformation.
\end{remark}\vs

\section{Departure-based asymptotic stochastic ordering of order statistics}\label{2-section-order-statistics}

The main result regarding asymptotic stochastic comparison of certain stochastic processes, stated in \cref{2-section-distorted-distributions}, involves a number of conditions which may seem abstract at a first glance. To motivate these conditions, we first study departure-based asymptotic stochastic ordering of order statistics from two independent homogeneous samples, as the sample size $n$ becomes large. Order statistics have been a widely discussed topic in the literature of various fields of study. Given a sample $\left\{X_1,X_2,\ldots,X_n\right\}$ of random observations, let us denote the order statistics by $X_{1:n} \leq X_{2:n} \leq \ldots \leq X_{n:n}$, where $X_{k:n}$ is the $k$th order statistic. Observe that $\left\{X_{k:n} : n \in \mathbb{N}\right\}$ may be considered as a valid discrete-time stochastic process. In \cref{2-proposition-n-gamma-rate-special-case-largest-order}, we have essentially compared the sequences $\left\{X_{n:n} : n \in \mathbb{N}\right\}$ and $\left\{Y_{n:n} : n \in \mathbb{N}\right\}$ of largest order statistics arising from the respective parent distributions $F_X$ and $F_Y$, as $n \to \infty$. Let $\left\{X_1,X_2,\ldots,X_n\right\}$ be a random sample from a distribution $F_X$ and let $\left\{Y_1,Y_2,\ldots,Y_n\right\}$ be that from a distribution $F_Y$. Assume that $F_X$ and $F_Y$ are continuous, strictly increasing and have finite second order moments. The main goal of this section is to make asymptotic stochastic comparison between $X_{1+\left[(n-1)\gamma_n\right]:n}$ and $Y_{1+\left[(n-1)\gamma_n\right]:n}$, where $\left\{\gamma_n : n \in \mathbb{N}\right\}$ is a $[0,1]$-valued sequence that converges to some $\gamma \in [0,1]$, as $n \to \infty$. The reason for choosing the $(1+\left[(n-1)\gamma_n\right])$th order statistic for comparison purpose is that it gives a coverage of various types of order statistics encountered in the literature. In particular,
\begin{enumerate}[label=(\roman*)]
	\item $\gamma_n=0$, for every $n \in \mathbb{N}$ : smallest order statistic.
	\item $\gamma_n=1$, for every $n \in \mathbb{N}$ : largest order statistic.
	\item $\gamma_n=\frac{k-1}{n-1}$, for every $n \in \mathbb{N}$ : $k$th order statistic.
	\item $\gamma_n=\frac{n-k}{n-1}$, for every $n \in \mathbb{N}$ : $(n-k+1)$th order statistic.
	\item $\gamma_n=\gamma \in (0,1)$, for every $n \in \mathbb{N}$ : central order statistics.
\end{enumerate}

Note that, for the constant sequences $\gamma_n=\gamma \in [0,1]$, for every $n \in \mathbb{N}$, $1+\left[(n-1)\gamma\right]$ goes from $1$ to $n$, as $\gamma$ traverses from $0$ to $1$. Hence it covers all the central order statistics as well as the smallest and the largest order statistics. We have to consider nonconstant sequences to accommodate the extreme order statistics, which are characterized by $\gamma_n=O(1/n)$ or $1-\gamma_n=O(1/n)$, apart from the smallest and the largest order statistics. All the cases described above fall under the common umbrella given by the following rate of convergence $\lvert \gamma_n-\gamma \rvert=O(1/n)$. The respective cdfs of $X_{1+\left[(n-1)\gamma_n\right]:n}$ and $Y_{1+\left[(n-1)\gamma_n\right]:n}$ are given by $F_{X_{1+\left[(n-1)\gamma_n\right]:n}}(x)=\phi_{n,\gamma_n}(F_X(x))$ and $F_{Y_{1+\left[(n-1)\gamma_n\right]:n}}(x)=\phi_{n,\gamma_n}(F_Y(x))$ for every $x \in \mathbb{R}$, where
\begin{equation}
\label{2-eq-phi-n-alpha}
\phi_{n,\alpha}(t)=\sum_{j=1+\left[(n-1)\alpha\right]}^n \binom{n}{j} t^j (1-t)^{n-j},
\end{equation}
for every $n \in \mathbb{N}$, $\alpha \in [0,1]$ and $t \in [0,1]$. The next sequence of results are useful in proving the main result of this section. Some of the proofs are given in the appendix.

\begin{lemma}\label{2-lemma-A_0-B_0}
	Let $F_X$ and $F_Y$ be two continuous and strictly increasing cdfs. We have
	$A_0=F_X(B_0)=F_Y(B_0)$, $A_1=F_X(B_1)=F_Y(B_1)$ and $A_2=F_X(B_2)=F_Y(B_2)$.
\end{lemma}\vs

\begin{remark}
Roughly, \cref{2-lemma-A_0-B_0} conveys that if $F_X$ and $F_Y$ are well-behaved, i.e. they are continuous and strictly increasing, then the comparative behavior of $F_X$ and $F_Y$ reflects in the same of $F_X^{-1}$ and $F_Y^{-1}$. To be precise, if $F_X>F_Y$ around a point $x_0 \in \mathbb{R}$, then $F_X^{-1}<F_Y^{-1}$ in an open interval containing both $F_X^{-1}(x_0)$ and $F_Y^{-1}(x_0)$. \hfill $\blacksquare$
\end{remark}\vs

The next proposition outlines some key properties of $\phi_{n,\alpha}$, defined in \eqref{2-eq-phi-n-alpha}.
\begin{proposition}
\label{2-proposition-phi_n-r}
	For every $n \in \mathbb{N}$, we have
	\begin{enumerate}[label=\textnormal{(\arabic*)}]
		\item\label{2-item-r-0-1} $\phi_{n,\alpha}$ is a distortion function.
		
		\item\label{2-item-r-continuous-differentiable} $\phi_{n,\alpha}$ is continuous in $[0,1]$ and differentiable in $(0,1)$.
		
		\item\label{2-item-r-increasing-in-t} $\phi_{n,\alpha}(t)$ is strictly increasing in $(0,1)$.
		
		\item\label{2-item-r-derivative-continuous-differentiable} $\phi_{n,\alpha}'(t)$ is continuous in $[0,1]$ and differentiable in $(0,1)$.
		
		\item\label{2-item-r-increasing-decreasing-derivative} $\phi_{n,\alpha}'(t)$ is increasing in $\left(0,\frac{\left[(n-1)\alpha\right]}{n-1}\right)$ and decreasing in $\left(\frac{\left[(n-1)\alpha\right]}{n-1},1\right)$. \hfill $\blacksquare$
	\end{enumerate}
\end{proposition}\vs

\cref{2-fig-orderstat1} and \cref{2-fig-orderstat2} illustrate the properties of $\phi_{n,\gamma_n}'$, given in \cref{2-proposition-phi_n-r}. We shall need the following computation later in the paper.
\begin{equation}\label{2-eq-n-gamma-phi-n-first-derivative}
	\phi_{n,\alpha}'(t)=n\binom{n-1}{\left[(n-1)\alpha\right]} t^{\left[(n-1)\alpha\right]}(1-t)^{n-\left[(n-1)\alpha\right]-1}.
\end{equation}

\begin{figure}
	\centering
	\begin{subfigure}{0.45\textwidth}
		\centering
		\includegraphics[width=\linewidth]{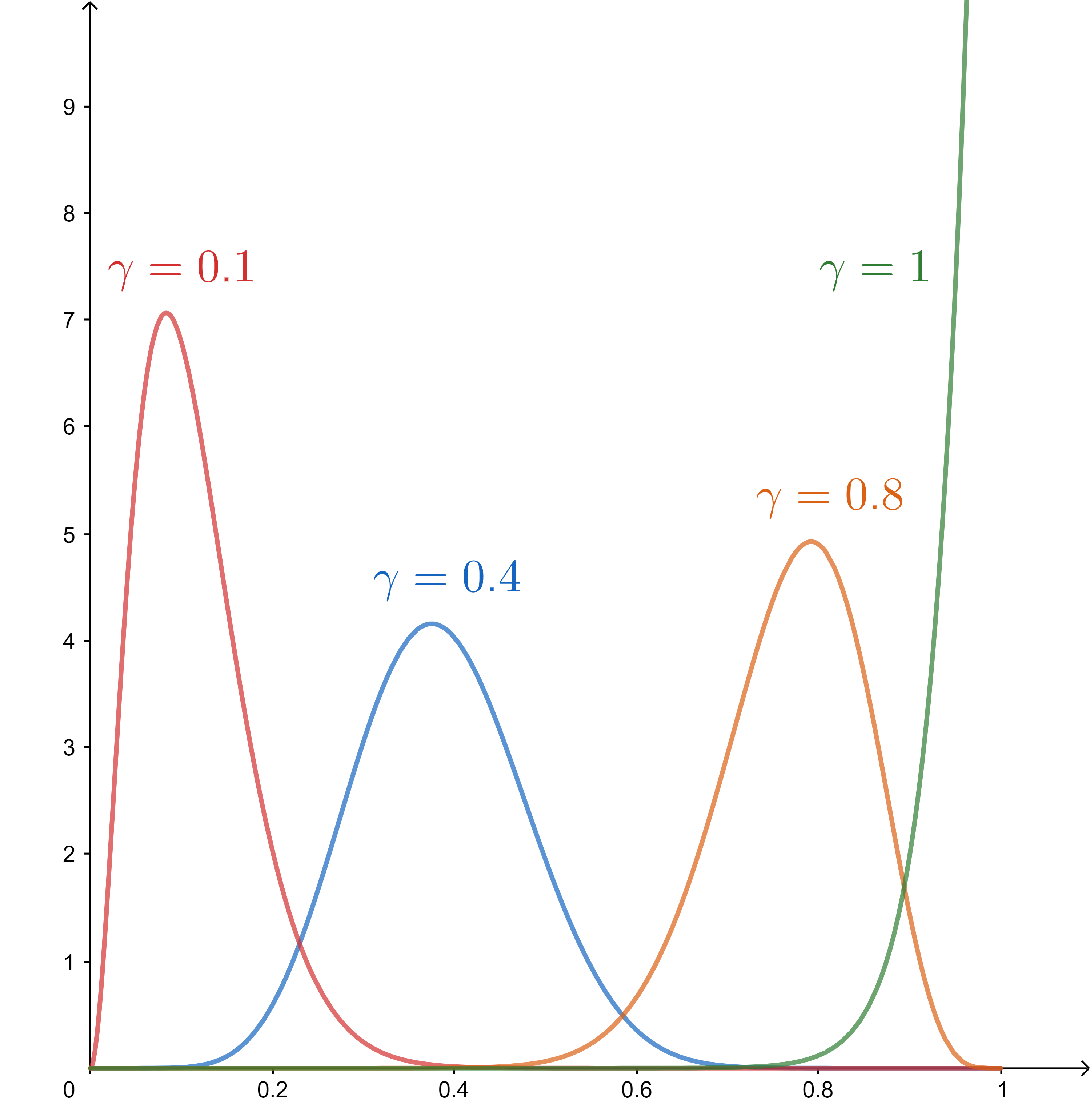}
		\caption{\centering Plots of $\phi_{n,\gamma_n}'$ for $n=10$ and constant sequences $\gamma_n=\gamma=0.1,0.4,0.8$ and $1$}
		\label{2-fig-orderstat1}
	\end{subfigure}
	\hfill
	\begin{subfigure}{0.45\textwidth}
		\centering
		\includegraphics[width=1.03\linewidth]{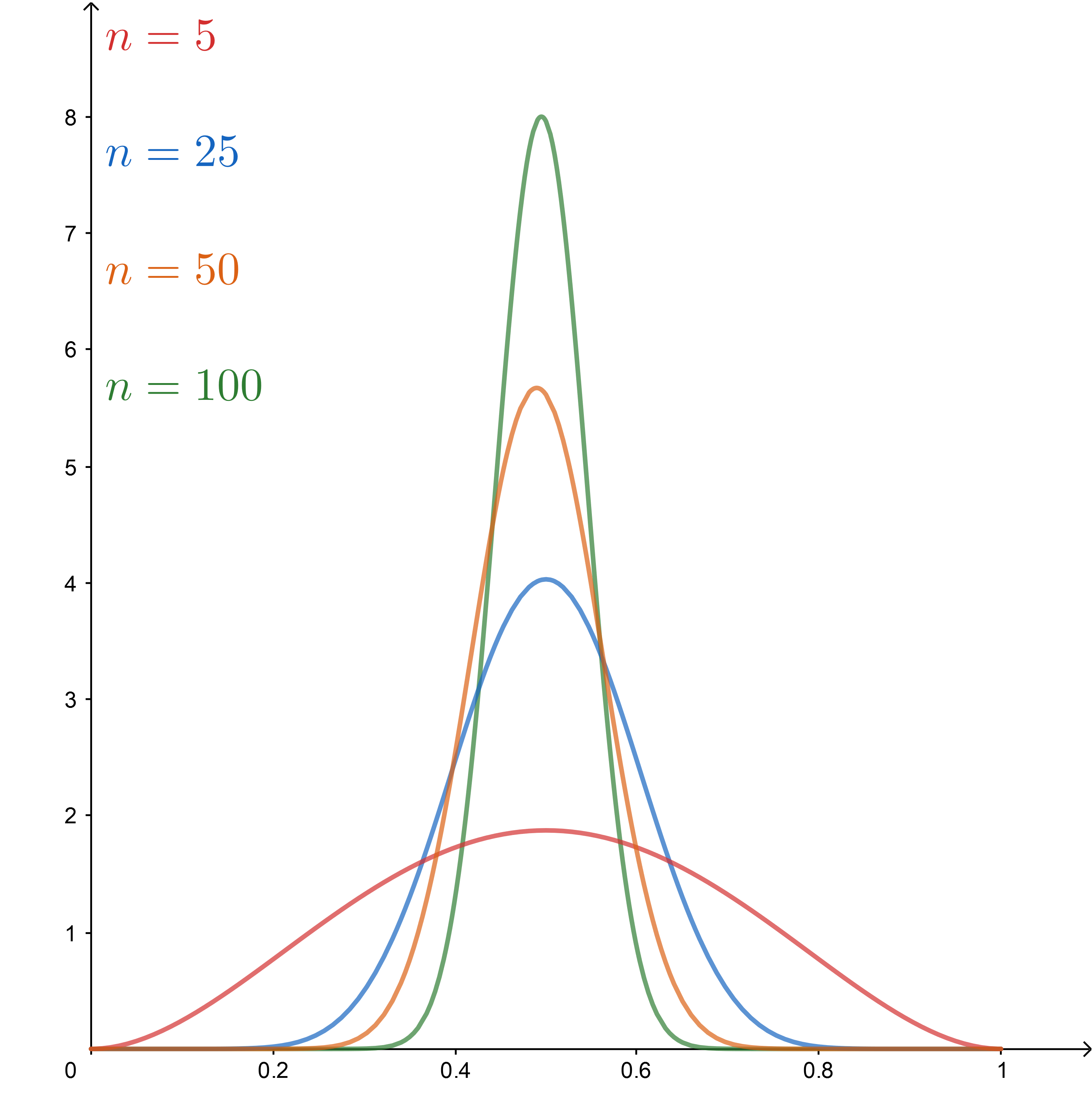}
		\caption{\centering Plots of $\phi_{n,\gamma_n}'$ for constant sequence $\gamma_n=\gamma=1/2$ and $n=5,25,50$ and $100$}
		\label{2-fig-orderstat2}
	\end{subfigure}
	\caption{}
	\label{2-fig-orderstat}
\end{figure}

\begin{lemma}
	\label{2-lemma-xi}
	Let $\gamma \in [0,1]$. Also, let $\xi_{\gamma}:(0,1) \times (0,1) \to \mathbb{R}$ be given by
	\[
	\xi_{\gamma}(t,s)=\left(\frac{t}{s}\right)^{\gamma} \left(\frac{1-t}{1-s}\right)^{1-\gamma}.
	\]
	Then we have $\xi_{\gamma}(t,s)<1$ if $0<t<s \leq \gamma \text{ or } 1>t>s \geq \gamma$.
\end{lemma}\vs

\begin{proposition}\label{2-proposition-phi_n-r-2}
	Let $\gamma_n \to \gamma \in [0,1]$, as $n \to \infty$ at the rate $\lvert \gamma_n-\gamma \rvert=O(1/n)$. Then, for $(t,s) \in (0,1) \times (0,1)$, we have
	\[
	\frac{\phi_{n,\gamma_n}'(t)}{\phi_{n,\gamma_n}'(s)} \leq C(t,s) \left\{\left(\frac{t}{s}\right)^{\gamma} \left(\frac{1-t}{1-s}\right)^{1-\gamma}\right\}^{n-1},
	\]
	where $C(t,s)$ is a positive constant, independent of $n$. \hfill $\blacksquare$
\end{proposition}\vs

The next corollary, which follows from \cref{2-lemma-xi} and \cref{2-proposition-phi_n-r-2}, gives an asymptotic property of the sequence of distortion functions $\left\{\phi_{n,\gamma_n} : n \in \mathbb{N}\right\}$.
\begin{corollary}
	\label{2-corollary-phi_n-r-2-limit}
	Let $\gamma_n \to \gamma \in [0,1]$, as $n \to \infty$ at the rate $\lvert \gamma_n-\gamma \rvert=O(1/n)$. Also assume that $0<t<s<\gamma$ or $1>t>s>\gamma$. Then
	\begin{equation*}
		\label{2-eq-phi_{n,gamma}-derivative-ratio-limit}
		\lim_{n \to \infty} \frac{\phi_{n,\gamma_n}'(t)}{\phi_{n,\gamma_n}'(s)}=0.
	\end{equation*}
\end{corollary}\vs

\begin{lemma}
\label{2-lemma-cauchy-normal-2.r}
	Let $n \in \mathbb{N}$ and $\alpha \in [0,1]$. Then, we have $F_{X_{1+\left[(n-1)\alpha\right]:n}}^{-1}(t)=F_X^{-1}(\phi_{n,\alpha}^{-1}(t))$, for every $t \in (0,1)$.
\end{lemma}\vs

\begin{lemma}
\label{2-lemma-A_0-A_0-n}
	Let us define
	\[
	A_{0,n}=\left\{v \in (0,1) : F_{X_{1+\left[(n-1)\gamma_n\right]:n}}^{-1}(v)>F_{Y_{1+\left[(n-1)\gamma_n\right]:n}}^{-1}(v)\right\}.
	\]
	Then $A_{0,n}=\phi_{n,\gamma_n}(A_0)$, where $\phi_{n,\gamma_n}(A_0)=\left\{\phi_{n,\gamma_n}(u) : u \in A_0\right\}$.
\end{lemma}\vs

Now we are in a position to state and prove the first main result of this section. Here, we consider random samples from two different distributions $F_X$ and $F_Y$, and derive sufficient conditions for asymptotic stochastic order for extreme and central order statistics from the said random samples.
\begin{theorem}\label{2-theorem-n-gamma-rate}
	Let $X_1,X_2,\ldots,X_n$ be a random sample from a distribution $F_X$ and let $Y_1,Y_2,\ldots,Y_n$ be that from a distribution $F_Y$. Assume that $F_X$ and $F_Y$ are continuous, strictly increasing and have finite second order moments. Let us define, for each $\gamma \in [0,1]$,
	\begin{align*}
	&c_{\gamma}=\sup{\left\{B_0 \medcap (-\infty,F_X^{-1}(\gamma))\right\}},\\
	&d_{\gamma}=\inf{\left\{B_0 \medcap (F_X^{-1}(\gamma),\infty)\right\}},\\
	&a_{\gamma}=\sup{\left\{B_2 \medcap (-\infty,F_X^{-1}(\gamma))\right\}},\\
	&b_{\gamma}=\inf{\left\{B_2 \medcap (F_X^{-1}(\gamma),\infty)\right\}}.
	\end{align*}
	Assume that $c_{\gamma}<a_{\gamma}$, whenever $B_0 \medcap (-\infty,F_X^{-1}(\gamma)) \neq \emptyset$ and $b_{\gamma}<d_{\gamma}$, whenever $B_0 \medcap (F_X^{-1}(\gamma),\infty) \neq \emptyset$.\footnote{In particular, the conditions $c_{\gamma}<a_{\gamma}$ and $b_{\gamma}<d_{\gamma}$ imply that $\gamma \notin A_0 \cup \partial A_0$.} Suppose that $\left\{\gamma_n : n \in \mathbb{N}\right\}$ is a $[0,1]$-valued sequence that converges to $\gamma \in [0,1]$, at the rate $\lvert \gamma_n-\gamma \rvert=O(1/n)$. Then, for any given $\epsilon>0$,
	\begin{equation}\label{2-eq-n-gamma-rate}
	\varepsilon_{\mathcal{W}_2}(F_{X_{1+[(n-1)\gamma_n]:n}},F_{Y_{1+[(n-1)\gamma_n]:n}}) \leq C_{\epsilon,\gamma} z_{\epsilon,\gamma}^{n-1},
	\end{equation}
	where $C_{\epsilon,\gamma}$ is a nonnegative constant which is independent of $n$ and
	\[
	z_{\epsilon,\gamma}=\begin{cases*}
	\frac{1-F_X(d_{\gamma})}{1-F_X(b_{\gamma})-\epsilon}, &\textnormal{ if } $\gamma=0$,\\
	\max{\left\{\left(\frac{F_X(c_{\gamma})}{F_X(a_{\gamma})-\epsilon}\right)^{\gamma} \left(\frac{1-F_X(c_{\gamma})}{1-F_X(a_{\gamma})+\epsilon}\right)^{1-\gamma},\left(\frac{F_X(d_{\gamma})}{F_X(b_{\gamma})+\epsilon}\right)^{\gamma} \left(\frac{1-F_X(d_{\gamma})}{1-F_X(b_{\gamma})-\epsilon}\right)^{1-\gamma}\right\}}, &\textnormal{ if } $0<\gamma<1$,\\
	\frac{F_X(c_{\gamma})}{F_X(a_{\gamma})-\epsilon}, &\textnormal{ if } $\gamma=1$.
	\end{cases*}
	\]
	Furthermore, $X_{1+[(n-1)\gamma_n]:n} \leq_{\textnormal{d-ast}} Y_{1+[(n-1)\gamma_n]:n}$, as $n \to \infty$.
\end{theorem}

\begin{proof}
	If $B_0 \medcap (-\infty,F_X^{-1}(\gamma))=\emptyset$, then $c_{\gamma}=\sup{\left\{B_0 \medcap (-\infty,F_X^{-1}(\gamma))\right\}}=-\infty$, and hence $F_X(c_{\gamma})=0$. Otherwise, we have
	\begin{align}
	\label{2-eq-theorem-F_X-F_Y-A_0-1-r-left-F}
	F_X(c_{\gamma})&=F_X(\sup{\left\{B_0 \medcap (-\infty,F_X^{-1}(\gamma))\right\}})\nonumber\\
	&=\sup{F_X(\left\{B_0 \medcap (-\infty,F_X^{-1}(\gamma))\right\})}\nonumber\\
	&=\sup{\left\{F_X(B_0) \medcap F_X((-\infty),F_X^{-1}(\gamma))\right\}}\nonumber\\
	&=\sup{\left\{A_0 \medcap (0,\gamma)\right\}},
	\end{align}
	where the second equality holds since $F_X$ is continuous and nondecreasing in $B_0 \medcap (-\infty,F_X^{-1}(\gamma))$, the third inequality follows from the fact that $F_X$ is strictly increasing and hence injective, the last equality follows from \cref{2-lemma-A_0-B_0}. Note that $B_0 \medcap (-\infty,F_X^{-1}(\gamma))=\emptyset$ if and only if $A_0 \medcap (0,\gamma)=\emptyset$. Hence, in the case where $A_0 \medcap (0,\gamma) \neq \emptyset$, we have
	\begin{equation}
	\label{2-eq-theorem-F_X-F_Y-A_0-1-r-left}
	A_0 \medcap (0,\gamma) \subseteq (0,F_X(c_{\gamma})).
	\end{equation}
	Note that \eqref{2-eq-theorem-F_X-F_Y-A_0-1-r-left} holds even when $A_0 \medcap (0,\gamma)=\emptyset$, since the empty set is a subset of itself. On the other hand, if $B_0 \medcap (F_X^{-1}(\gamma),\infty)=\emptyset$, then we have $d_{\gamma}=\inf{\left\{B_0 \medcap (F_X^{-1}(\gamma),\infty)\right\}}=\infty$, and hence $F_X(d_{\gamma})=1$. Otherwise, in the same line of argument as in deriving \eqref{2-eq-theorem-F_X-F_Y-A_0-1-r-left-F}, we have
	\begin{equation}\label{2-eq-theorem-F_X-F_Y-A_0-1-r-right-F}
	F_X(d_{\gamma})=\inf{\left\{A_0 \medcap (\gamma,1)\right\}}
	\end{equation}
	and hence
	\begin{equation}\label{2-eq-theorem-F_X-F_Y-A_0-1-r-right}
	A_0 \medcap (\gamma,1) \subseteq (F_X(d_{\gamma}),1).
	\end{equation}
	Note that \eqref{2-eq-theorem-F_X-F_Y-A_0-1-r-right} holds even when $A_0 \medcap (\gamma,1)=\emptyset$. It follows from the assumption $c_{\gamma}<a_{\gamma} \leq b_{\gamma}<d_{\gamma}$ that $c_{\gamma}<F_X^{-1}(\gamma)<d_{\gamma}$. By strict increasingness of $F_X$, $F_X(c_{\gamma})<\gamma<F_X(d_{\gamma})$. By \eqref{2-eq-theorem-F_X-F_Y-A_0-1-r-left-F} and \eqref{2-eq-theorem-F_X-F_Y-A_0-1-r-right-F}, we have $\sup{\left\{A_0 \medcap (0,\gamma)\right\}}<\gamma<\inf{\left\{A_0 \medcap (\gamma,1)\right\}}$. Thus, $\gamma \notin A_0$. Note that
	\begin{align*}
		&\gamma=0 \Rightarrow A_0=A_0 \medcap (\gamma,1),\\
		&0<\gamma<1 \Rightarrow A_0=(A_0 \medcap (0,\gamma)) \cup (A_0 \medcap (\gamma,1)),\\
		&\gamma=1 \Rightarrow A_0=A_0 \medcap (0,\gamma).
	\end{align*}
	By \eqref{2-eq-theorem-F_X-F_Y-A_0-1-r-left} and \eqref{2-eq-theorem-F_X-F_Y-A_0-1-r-right}, we obtain
	\begin{equation}
	\label{2-eq-gamma-A_0}
	A_0 \subseteq (0,F_X(c_{\gamma})) \cup (F_X(d_{\gamma}),1),
	\end{equation}
	for every $\gamma \in [0,1]$. Let $A_{0,n}$ be as defined in \cref{2-lemma-A_0-A_0-n} and $\phi_{n,\gamma_n}$ be as defined in \eqref{2-eq-phi-n-alpha}. Then, by \cref{2-lemma-A_0-A_0-n},
	\begin{equation}
	\label{2-eq-gamma-A_{0,n}}
	A_{0,n} \subseteq (0,\phi_{n,\gamma_n}(F_X(c_{\gamma}))) \cup (\phi_{n,\gamma_n}(F_X(d_{\gamma})),1),
	\end{equation}
	for every $\gamma \in [0,1]$. Let $\epsilon \in (0,\min{\left\{F_X(a_{\gamma})-F_X(c_{\gamma}),F_X(d_{\gamma})-F_X(b_{\gamma})\right\}})$ be arbitrarily small. Note that
	\[
	\left\vert\frac{\left[(n-1)\gamma_n\right]}{n-1}-\gamma\right\vert=\left\vert \gamma_n-\frac{\left\{(n-1)\gamma_n\right\}}{n-1}-\gamma \right\vert < \left\vert\gamma_n-\gamma\right\vert+\frac{1}{n-1}.
	\]
	Thus, we can choose $N$ large enough so that $\left\vert\frac{\left[(n-1)\gamma_n\right]}{n-1}-\gamma\right\vert<\frac{\epsilon}{2}$, for every $n \geq N$. Now,
	\begin{equation}
	\label{2-eq-gamma-varepsilon-A_{0,n}}
	\varepsilon_{\mathcal{W}_2}(F_{X_{1+\left[(n-1)\gamma_n\right]:n}},F_{Y_{1+\left[(n-1)\gamma_n\right]:n}})=\frac{\int_{A_{0,n}} (F_{X_{1+\left[(n-1)\gamma_n\right]:n}}^{-1}(t)-F_{Y_{1+\left[(n-1)\gamma_n\right]:n}}^{-1}(t))^2 dt}{\int_0^1 (F_{X_{1+\left[(n-1)\gamma_n\right]:n}}^{-1}(t)-F_{Y_{1+\left[(n-1)\gamma_n\right]:n}}^{-1}(t))^2 dt}.
	\end{equation}
	\ni{\bf Case 1: $0<\gamma<1$}. Assume that $\epsilon<\min{\left\{2\gamma,2(1-\gamma)\right\}}$. Using \eqref{2-eq-gamma-A_{0,n}} and \cref{2-lemma-cauchy-normal-2.r}, we see that the right hand side of \eqref{2-eq-gamma-varepsilon-A_{0,n}} is bounded above by
	\[
	\frac{\int_0^{\phi_{n,\gamma_n}(F_X(c_{\gamma}))} (F_X^{-1}(\phi_{n,\gamma_n}^{-1}(t))-F_Y^{-1}(\phi_{n,\gamma_n}^{-1}(t)))^2 dt}{\int_0^1 (F_X^{-1}(\phi_{n,\gamma_n}^{-1}(t))-F_Y^{-1}(\phi_{n,\gamma_n}^{-1}(t)))^2 dt}+\frac{\int_{\phi_{n,\gamma_n}(F_X(d_{\gamma}))}^1 (F_X^{-1}(\phi_{n,\gamma_n}^{-1}(t))-F_Y^{-1}(\phi_{n,\gamma_n}^{-1}(t)))^2 dt}{\int_0^1 (F_X^{-1}(\phi_{n,\gamma_n}^{-1}(t))-F_Y^{-1}(\phi_{n,\gamma_n}^{-1}(t)))^2 dt}.
	\]
	Then, by taking the transformation $t=\phi_{n,\gamma_n}(u)$, we have
	\begin{align*}
	&\phantom{\,\,\,\,\,\,\,\,}\varepsilon_{\mathcal{W}_2}(F_{X_{1+\left[(n-1)\gamma_n\right]:n}},F_{Y_{1+\left[(n-1)\gamma_n\right]:n}})\\
	&\leq \frac{\int_0^{F_X(c_{\gamma})} (F_X^{-1}(u)-F_Y^{-1}(u))^2 \phi_{n,\gamma_n}'(u) du}{\int_0^1 (F_X^{-1}(u)-F_Y^{-1}(u))^2 \phi_{n,\gamma_n}'(u) du}+\frac{\int_{F_X(d_{\gamma})}^1 (F_X^{-1}(u)-F_Y^{-1}(u))^2 \phi_{n,\gamma_n}'(u) du}{\int_0^1 (F_X^{-1}(u)-F_Y^{-1}(u))^2 \phi_{n,\gamma_n}'(u) du}\\
	&\leq \frac{\int_0^{F_X(c_{\gamma})} (F_X^{-1}(u)-F_Y^{-1}(u))^2 \phi_{n,\gamma_n}'(u) du}{\int_{F_X(a_{\gamma})-\epsilon}^{\frac{\left[(n-1)\gamma_n\right]}{n-1}} (F_X^{-1}(u)-F_Y^{-1}(u))^2 \phi_{n,\gamma_n}'(u) du}+\frac{\int_{F_X(d_{\gamma})}^1 (F_X^{-1}(u)-F_Y^{-1}(u))^2 \phi_{n,\gamma_n}'(u) du}{\int_{\frac{\left[(n-1)\gamma_n\right]}{n-1}}^{F_X(b_{\gamma})+\epsilon} (F_X^{-1}(u)-F_Y^{-1}(u))^2 \phi_{n,\gamma_n}'(u) du}\\
	&\leq \frac{\int_0^{F_X(c_{\gamma})} (F_X^{-1}(u)-F_Y^{-1}(u))^2 \phi_{n,\gamma_n}'(u) du}{\int_{F_X(a_{\gamma})-\epsilon}^{\gamma-\frac{\epsilon}{2}} (F_X^{-1}(u)-F_Y^{-1}(u))^2 \phi_{n,\gamma_n}'(u) du}+\frac{\int_{F_X(d_{\gamma})}^1 (F_X^{-1}(u)-F_Y^{-1}(u))^2 \phi_{n,\gamma_n}'(u) du}{\int_{\gamma+\frac{\epsilon}{2}}^{F_X(b_{\gamma})+\epsilon} (F_X^{-1}(u)-F_Y^{-1}(u))^2 \phi_{n,\gamma_n}'(u) du}.
	\end{align*}
	Since $F_X(a_{\gamma})-\epsilon<\gamma$ and $\left[(n-1)\gamma_n\right]/(n-1) \to \gamma$, as $n \to \infty$, we have, for large enough $n$, $F_X(a_{\gamma})-\epsilon<\left[(n-1)\gamma_n\right]/(n-1)$ and consequently $F_X(c_{\gamma})<\left[(n-1)\gamma_n\right]/(n-1)$. Similarly, $F_X(d_{\gamma})>F_X(b_{\gamma})+\epsilon>\left[(n-1)\gamma_n\right]/(n-1)$, for large enough $n$. Now, from \cref{2-proposition-phi_n-r} \ref{2-item-r-increasing-decreasing-derivative}, we obtain
	\begin{equation}\label{2-eq-varepsilon-estimate-n-gamma-rate}
	\varepsilon_{\mathcal{W}_2}(F_{X_{1+\left[(n-1)\gamma_n\right]:n}},F_{Y_{1+\left[(n-1)\gamma_n\right]:n}}) \leq \frac{k_{11}\phi_{n,\gamma_n}'(F_X(c_{\gamma}))}{k_{21}\phi_{n,\gamma_n}'(F_X(a_{\gamma})-\epsilon)}+\frac{k_{12}\phi_{n,\gamma_n}'(F_X(d_{\gamma}))}{k_{22}\phi_{n,\gamma_n}'(F_X(b_{\gamma})+\epsilon)},
	\end{equation}
	where
	\begin{align*}
	&k_{11}=\int_0^{F_X(c_{\gamma})} (F_X^{-1}(u)-F_Y^{-1}(u))^2 du \geq 0,\\
	&k_{12}=\int_{F_X(d_{\gamma})}^1 (F_X^{-1}(u)-F_Y^{-1}(u))^2 du \geq 0,\\
	&k_{21}=\int_{F_X(a_{\gamma})-\epsilon}^{\gamma-\frac{\epsilon}{2}} (F_X^{-1}(u)-F_Y^{-1}(u))^2 du>0,\\
	&k_{22}=\int_{\gamma+\frac{\epsilon}{2}}^{F_X(b_{\gamma})+\epsilon} (F_X^{-1}(u)-F_Y^{-1}(u))^2 du>0.
	\end{align*}
	The strict positivity of $k_{21}$ and $k_{22}$ follows from the observations that one can choose $\epsilon$ $(>0)$ small enough such that $\lvert F_X^{-1}(u)-F_Y^{-1}(u) \rvert>0$ in the regions $(F_X(a_{\gamma})-\epsilon,F_X(a_{\gamma})-\epsilon/2)$ and $(F_X(b_{\gamma})+\epsilon/2,F_X(b_{\gamma})+\epsilon)$, due to continuity of $F_X^{-1}$ and $F_Y^{-1}$ (resulting from strict increasingness of $F_X$ and $F_Y$). Now, if $F_X(c_{\gamma})=0$, then $\phi_{n,\gamma_n}'(F_X(c_{\gamma}))/\phi_{n,\gamma_n}'(F_X(a_{\gamma})-\epsilon)=0$. On the other hand, if $F_X(c_{\gamma})>0$, then choosing $t=F_X(c_{\gamma})$ and $s=F_X(a_{\gamma})-\epsilon$ in \cref{2-proposition-phi_n-r-2}, we have
	\begin{equation}
	\label{2-eq-ratio-estimate-c}
	\frac{\phi_{n,\gamma_n}'(F_X(c_{\gamma}))}{\phi_{n,\gamma_n}'(F_X(a_{\gamma})-\epsilon)} \leq c_1 \left\{\left(\frac{F_X(c_{\gamma})}{F_X(a_{\gamma})-\epsilon}\right)^{\gamma} \left(\frac{1-F_X(c_{\gamma})}{1-F_X(a_{\gamma})+\epsilon}\right)^{1-\gamma}\right\}^{n-1}.
	\end{equation}
	where
	\[
	c_1=\left\{\frac{(F_X(a_{\gamma})-\epsilon)(1-F_X(c_{\gamma}))}{F_X(c_{\gamma})(1-F_X(a_{\gamma})+\epsilon)}\right\}^{K+1}.
	\]
	Again, if $F_X(d_{\gamma})=1$, then $\phi_{n,\gamma_n}'(F_X(d_{\gamma}))/\phi_{n,\gamma_n}'(F_X(b_{\gamma})+\epsilon)=0$. On the other hand, if $F_X(d_{\gamma})<1$, then choosing $t=F_X(d_{\gamma})$ and $s=F_X(b_{\gamma})+\epsilon$ in \cref{2-proposition-phi_n-r-2}, we have
	\begin{equation}
	\label{2-eq-ratio-estimate-d}
	\frac{\phi_{n,\gamma_n}'(F_X(d_{\gamma}))}{\phi_{n,\gamma_n}'(F_X(b_{\gamma})+\epsilon)} \leq c_2\left\{\left(\frac{F_X(d_{\gamma})}{F_X(b_{\gamma})+\epsilon}\right)^{\gamma} \left(\frac{1-F_X(d_{\gamma})}{1-F_X(b_{\gamma})-\epsilon}\right)^{1-\gamma}\right\}^{n-1},
	\end{equation}
	where
	\[
	c_2=\left\{\frac{F_X(d_{\gamma})(1-F_X(b_{\gamma})-\epsilon)}{(F_X(b_{\gamma})+\epsilon)(1-F_X(d_{\gamma}))}\right\}^K.
	\]
	Putting \eqref{2-eq-ratio-estimate-d} and \eqref{2-eq-ratio-estimate-c} in \eqref{2-eq-varepsilon-estimate-n-gamma-rate}, we have
	\begin{align}
	\label{2-eq-n-gamma-rate-proof}
	\varepsilon_{\mathcal{W}_2}(F_{X_{1+\left[(n-1)\gamma_n\right]:n}},F_{Y_{1+\left[(n-1)\gamma_n\right]:n}}) &\leq \frac{c_1\,k_{11}}{k_{21}} \left\{\left(\frac{F_X(c_{\gamma})}{F_X(a_{\gamma})-\epsilon}\right)^{\gamma} \left(\frac{1-F_X(c_{\gamma})}{1-F_X(a_{\gamma})+\epsilon}\right)^{1-\gamma}\right\}^{n-1}\nonumber\\
	&\phantom{\leq}\,\,+\frac{c_2\,k_{12}}{k_{22}} \left\{\left(\frac{F_X(d_{\gamma})}{F_X(b_{\gamma})+\epsilon}\right)^{\gamma} \left(\frac{1-F_X(d_{\gamma})}{1-F_X(b_{\gamma})-\epsilon}\right)^{1-\gamma}\right\}^{n-1}\nonumber\\
	&\leq C_{\epsilon,\gamma} z_{\epsilon,\gamma}^{n-1},
	\end{align}
	where $C_{\epsilon,\gamma}=\frac{c_1\,k_{11}}{k_{21}}+\frac{c_2\,k_{12}}{k_{22}}$ and $z_{\epsilon,\gamma}$ is as defined in the statement of the result. Since $c_1$, $c_2$, $k_{11}$, $k_{12}$, $k_{21}$, $k_{22}$ are all nonnegative real numbers, which do not depend on $n$, we see that $C_{\epsilon,\gamma}$ is a nonnegative real number, which is independent of $n$. Observe that, if $a \geq 0$ and $b \geq 0$, then $z=\max{\left\{a,b\right\}} \Rightarrow z^{n-1}=\max{\left\{a^{n-1},b^{n-1}\right\}}$. Now choosing $t=F_X(c_{\gamma})$ and $s=F_X(a_{\gamma})-\epsilon$ in \cref{2-lemma-xi}, we have
	\[
	\left(\frac{F_X(c_{\gamma})}{F_X(a_{\gamma})-\epsilon}\right)^{\gamma} \left(\frac{1-F_X(c_{\gamma})}{1-F_X(a_{\gamma})+\epsilon}\right)^{1-\gamma}<1.
	\]
	Again, choosing $t=F_X(d_{\gamma})$ and $s=F_X(b_{\gamma})+\epsilon$ in \cref{2-lemma-xi}, we have
	\[
	\left(\frac{F_X(d_{\gamma})}{F_X(b_{\gamma})+\epsilon}\right)^{\gamma} \left(\frac{1-F_X(d_{\gamma})}{1-F_X(b_{\gamma})-\epsilon}\right)^{1-\gamma}<1.
	\]
	Hence $z_{\epsilon,\gamma}<1$. Then, using \eqref{2-eq-n-gamma-rate-proof}, we get $X_{1+[(n-1)\gamma_n]:n} \leq_{\textnormal{d-ast}} Y_{1+[(n-1)\gamma_n]:n}$, as $n \to \infty$.\vs
	
	\ni{\bf Case 2: $\gamma=0$.} We have from \eqref{2-eq-gamma-A_{0,n}},
	\begin{align*}
	\varepsilon_{\mathcal{W}_2}(F_{X_{1+\left[(n-1)\gamma_n\right]:n}},F_{Y_{1+\left[(n-1)\gamma_n\right]:n}})&\leq \frac{\int_{\phi_{n,\gamma_n}(F_X(d_{\gamma}))}^1 (F_{X_{1+\left[(n-1)\gamma_n\right]:n}}^{-1}(t)-F_{Y_{1+\left[(n-1)\gamma_n\right]:n}}^{-1}(t))^2 dt}{\int_0^1 (F_{X_{1+\left[(n-1)\gamma_n\right]:n}}^{-1}(t)-F_{Y_{1+\left[(n-1)\gamma_n\right]:n}}^{-1}(t))^2 dt}\\
	&\leq \frac{c_2 k_{12}}{k_{22}} \left(\frac{1-F_X(d_{\gamma})}{1-F_X(b_{\gamma})-\epsilon}\right)^{n-1}\\
	&=C_{\epsilon,\gamma} z_{\epsilon,\gamma}^{n-1},
	\end{align*}
	where the first and the second inequalities follow from the same chain of arguments as in the case of $0<\gamma<1$, the equality follows by choosing $C_{\epsilon,\gamma}=\frac{c_2 k_{12}}{k_{22}}$ and $z_{\epsilon,\gamma}=\frac{1-F_X(d_{\gamma})}{1-F_X(b_{\gamma})-\epsilon}$, which is strictly smaller than unity since $\epsilon<F_X(d_{\gamma})-F_X(b_{\gamma})$. Thus, $X_{1+[(n-1)\gamma_n]:n} \leq_{\textnormal{d-ast}} Y_{1+[(n-1)\gamma_n]:n}$, as $n \to \infty$.\vs
	
	\ni{\bf Case 3: $\gamma=1$}. On using the same chain of arguments as in the case $0<\gamma<1$ and choosing $C_{\epsilon,\gamma}=\frac{c_1 k_{11}}{k_{21}}$, $z_{\epsilon,\gamma}=\frac{F_X(c_{\gamma})}{F_X(a_{\gamma})-\epsilon}<1$, we have
	\begin{align*}
	\varepsilon_{\mathcal{W}_2}(F_{X_{1+\left[(n-1)\gamma_n\right]:n}},F_{Y_{1+\left[(n-1)\gamma_n\right]:n}})&\leq \frac{\int_0^{\phi_{n,\gamma_n}(F_X(c_{\gamma}))} (F_{X_{1+\left[(n-1)\gamma_n\right]:n}}^{-1}(t)-F_{Y_{1+\left[(n-1)\gamma_n\right]:n}}^{-1}(t))^2 dt}{\int_0^1 (F_{X_{1+\left[(n-1)\gamma_n\right]:n}}^{-1}(t)-F_{Y_{1+\left[(n-1)\gamma_n\right]:n}}^{-1}(t))^2 dt}\\
	&\leq \frac{c_1 k_{11}}{k_{21}} \left(\frac{F_X(c_{\gamma})}{F_X(a_{\gamma})-\epsilon}\right)^{n-1}\\
	&\leq C_{\epsilon,\gamma} z_{\epsilon,\gamma}^{n-1}.
	\end{align*}
	Since $\epsilon<F_X(a_{\gamma})-F_X(c_{\gamma})$, we have $z_{\epsilon,\gamma}<1$. Hence $X_{1+[(n-1)\gamma_n]:n} \leq_{\textnormal{d-ast}} Y_{1+[(n-1)\gamma_n]:n}$, as $n \to \infty$.
\end{proof}\vs

\begin{remark}
Note that, as $\epsilon \to 0$, the multiplicative constant $C_{\epsilon,\gamma}$ becomes larger. Thus, for a very small $\epsilon>0$, we may get a large value of $C_{\epsilon,\gamma}$. However, it pays off asymptotically, since we get a smaller value of $z_{\epsilon,\gamma}$, so that $z_{\epsilon,\gamma}^{n-1}$ diminishes faster as $n \to \infty$. To see this, take $\gamma \in (0,1)$ and $\epsilon \in (0,\min{\left\{s,1-s\right\}})$, and define, for $(t,s) \in (0,1) \times (0,1)$,
\[
A_{\epsilon,\gamma}(t,s)=\left(\frac{t}{s-\epsilon}\right)^{\gamma} \left(\frac{1-t}{1-s+\epsilon}\right)^{1-\gamma} \text{ and }\,\, B_{\epsilon,\gamma}(t,s)=\left(\frac{t}{s+\epsilon}\right)^{\gamma} \left(\frac{1-t}{1-s-\epsilon}\right)^{1-\gamma}.
\]
Observe that $\sign\left(\frac{\partial}{\partial \epsilon} A_{\epsilon,\gamma}(t,s)\right)=\sign(\gamma-s+\epsilon)$ and $\sign\left(\frac{\partial}{\partial \epsilon} B_{\epsilon,\gamma}(t,s)\right)=\sign(-\gamma+s+\epsilon)$. So, $A_{\epsilon,\gamma}(t,s)$ is increasing in $\epsilon$ if $s \leq \gamma$ and  $B_{\epsilon,\gamma}(t,s)$ is increasing in $\epsilon$ if
$s \geq \gamma$. Then, for $0<\gamma<1$,
\[
z_{\epsilon,\gamma}=\max{\left\{A_{\epsilon,\gamma}(F_X(c_{\gamma}),F_X(a_{\gamma})),B_{\epsilon,\gamma}(F_X(d_{\gamma}),F_X(b_{\gamma}))\right\}}
\]
decreases as $\epsilon$ decreases. However, for any fixed $n$, the optimal choice of $\epsilon$ has to be computed numerically.
\end{remark}\vs

\begin{remark}
An upper bound of $\varepsilon_{\mathcal{W}_2}(F_{Y_{1+\left[(n-1)\gamma_n\right]:n}},F_{X_{1+\left[(n-1)\gamma_n\right]:n}})$ in \cref{2-theorem-n-gamma-rate} is given in terms of $F_X(c_{\gamma})$, $F_X(a_{\gamma})$, $F_X(b_{\gamma})$ and $F_X(d_{\gamma})$, where $c_{\gamma},a_{\gamma},b_{\gamma}$ and $d_{\gamma}$ are defined in terms of $F_X^{-1}(\gamma)$. One can write the same in terms of $F_Y(c_{\gamma})$, $F_Y(a_{\gamma})$, $F_Y(b_{\gamma})$ and $F_Y(d_{\gamma})$, where $c_{\gamma},a_{\gamma},b_{\gamma}$ and $d_{\gamma}$ are defined in terms of $F_Y^{-1}(\gamma)$, without creating any inconsistency. To see this, let us define
\begin{align*}
&c_{\gamma,F_X}=\sup{\left\{B_0 \medcap (-\infty,F_X^{-1}(\gamma))\right\}},\,\, c_{\gamma,F_Y}=\sup{\left\{B_0 \medcap (-\infty,F_Y^{-1}(\gamma))\right\}},\\
&d_{\gamma,F_X}=\inf{\left\{B_0 \medcap (-\infty,F_X^{-1}(\gamma))\right\}},\,\, d_{\gamma,F_Y}=\inf{\left\{B_0 \medcap (-\infty,F_Y^{-1}(\gamma))\right\}},\\
&a_{\gamma,F_X}=\sup{\left\{B_2 \medcap (-\infty,F_X^{-1}(\gamma))\right\}},\,\, a_{\gamma,F_Y}=\sup{\left\{B_2 \medcap (-\infty,F_Y^{-1}(\gamma))\right\}},\\
&b_{\gamma,F_X}=\inf{\left\{B_2 \medcap (-\infty,F_X^{-1}(\gamma))\right\}},\,\, b_{\gamma,F_Y}=\inf{\left\{B_2 \medcap (-\infty,F_Y^{-1}(\gamma))\right\}}.
\end{align*}

There will be no inconsistency if $F_X(c_{\gamma,F_X})=F_Y(c_{\gamma,F_Y})$, $F_X(a_{\gamma,F_X})=F_Y(a_{\gamma,F_Y})$, $F_X(b_{\gamma,F_X})=F_Y(b_{\gamma,F_Y})$ and $F_X(d_{\gamma,F_X})=F_Y(d_{\gamma,F_Y})$. We shall show the first one and the proofs of the rest are similar.
\begin{align*}
F_X(c_{\gamma,F_X})&=F_X(\sup{\left\{B_0 \medcap (-\infty,F_X^{-1}(\gamma))\right\}})\\
&=\sup{F_X(B_0 \medcap (-\infty,F_X^{-1}(\gamma)))}\\
&=\sup{\left\{F_X(B_0) \medcap F_X(-\infty,F_X^{-1}(\gamma))\right\}}\\
&=\sup{\left\{F_X(B_0) \medcap (0,\gamma)\right\}}\\
&=\sup{\left\{A_0 \medcap (0,\gamma)\right\}},
\end{align*}
where the second equality holds due to continuity of $F_X$, the third equality follows from injectivity of $F_X$ (which in turn follows from strict increasingness of $F_X$) and the last equality is a consequence of \cref{2-lemma-A_0-B_0}. By exactly the same line of arguments, $F_Y(c_{\gamma,F_Y})=\sup{\left\{A_0 \medcap (0,\gamma)\right\}}$, and hence it is equal to $F_X(c_{\gamma,F_X})$.
\end{remark}\vs

\begin{remark}
The sets $B_0 \medcap (-\infty,F_X^{-1}(\gamma))$ and $B_0 \medcap (F_X^{-1}(\gamma),\infty)$ are not required to be nonempty. However, when both of these sets are empty, then the upper bound given in \eqref{2-eq-n-gamma-rate} reduces to $0$, which is expected, since in this case, we have $B_0$ to be empty, and hence follows the usual stochastic order between $F_X$ and $F_Y$, which directly translates to the same between $F_{X_{(1+\left[(n-1)\gamma\right])}}$ and $F_{Y_{(1+\left[(n-1)\gamma\right])}}$. Consider the case where $0<\gamma<1$ and assume that $B_0 \medcap (-\infty,F_X^{-1}(\gamma))=\emptyset$. In this case, $F_X(c_{\gamma})=0$, and hence the term
\[
\left(\frac{F_X(c_{\gamma})}{\gamma-\epsilon}\right)^{\gamma} \left(\frac{1-F_X(c_{\gamma})}{1-\gamma+\epsilon}\right)^{1-\gamma}
\]
does not contribute to $z_{\epsilon,\gamma}$. Likewise, if $B_0 \medcap (F_X^{-1}(\gamma),\infty)=\emptyset$, then $F_X(d_{\gamma})=1$, and hence the term 
\[
\left(\frac{F_X(d_{\gamma})}{F_X(b_{\gamma})+\epsilon}\right)^{\gamma} \left(\frac{1-F_X(d_{\gamma})}{1-F_X(b_{\gamma})-\epsilon}\right)^{1-\gamma}
\]
does not contribute to $z_{\epsilon,\gamma}$. Consequently, when $B_0 \medcap (-\infty,F_X^{-1}(\gamma))=\emptyset$ and $B_0 \medcap (F_X^{-1}(\gamma),\infty)=\emptyset$, then $z_{\epsilon,\gamma}=0$. In a similar manner, if $\gamma=0$ (resp. $\gamma=1$), then the emptiness of $B_0 \medcap (F_X^{-1}(\gamma),\infty)$ (resp. $B_0 \medcap (-\infty,F_X^{-1}(\gamma))$) results into $z_{\epsilon,\gamma}=0$.
\end{remark}\vs

\begin{remark}\label{2-remark}
Consider the case of largest order statistics, i.e. $\gamma_n=1$ for every $n \in \mathbb{N}$. Observe that, under the assumptions of \cref{2-proposition-n-gamma-rate-special-case-largest-order}, we have $c_{\gamma}=c<\infty$, $a_{\gamma}=\infty$ and hence $z_{\epsilon,\gamma}=F_X(c)/(1-\epsilon)$. \cref{2-theorem-n-gamma-rate} then immediately boils down to \cref{2-proposition-n-gamma-rate-special-case-largest-order}. \hfill $\blacksquare$
\end{remark}\vs

\cref{2-table-order-statistics} demonstrates the behaviour of the measure of departure from usual stochastic dominance of $Y_{1+[(n-1)\gamma]:n}$ over $X_{1+[(n-1)\gamma]:n}$ as $n$ increases for different values of $\gamma$. The computed values are rounded off to seven decimal places, except when these are less than $0.0000001$ or more than $0.9999999$, in which cases we report $<0.0000001$ or $>0.9999999$, respectively. It is clear from the table that, in general, the measure is not monotone in $n$ when $\gamma$ is fixed. However, as $n$ increases, there is a general tendency of the measure to increase towards $1$ if $\gamma<0.5$ and decrease towards $0$ if $\gamma>0.5$. These observations agree to \cref{2-theorem-n-gamma-rate} (with the roles of $F_X$ and $F_Y$ reversed when $\gamma<0.5$). The case when $\gamma=0.5$ falls outside the scope of \cref{2-theorem-n-gamma-rate} as in this case, for the example of $t_4$ and $N(0,1)$, we have $c_{\gamma}=a_{\gamma}=\gamma=b_{\gamma}=d_{\gamma}$ and the computed values of the measure do not converge to either $0$ or $1$, but rather fluctuate in between. Also note that the further away $\gamma$ is from $0.5$, the faster is the convergence of the measure to $0$ (if $\gamma<0.5$) or $1$ (if $\gamma>0.5$).\vs

\begin{table}
	\centering
	\setlength{\arrayrulewidth}{.05em}
	\setlength{\tabcolsep}{12pt}
	\resizebox{\textwidth}{!}{
	\begin{tabular}{|c|c|c|c|c|c|c|c|c|c|}
		\hline
		$\gamma$ $\rightarrow$ & 0 & 0.25 & 0.4 & 0.49 & 0.5 & 0.51 & 0.6 & 0.75 & 1 \\
		$n$ $\downarrow$ & & & & & & & & & \\
		\hline
		2 & 0.9865254 & 0.9865254 & 0.9865254 & 0.9865254 & 0.9865254 & 0.9865254 & 0.9865254 & 0.9865254 & 0.0134944 \\
		3 & 0.9988934 & 0.9988934 & 0.9988934 & 0.9988934 & 0.5000000 & 0.5000000 & 0.5000000 & 0.5000000 & 0.0011076 \\
		4 & 0.9998307 & 0.9998307 & 0.9262060 & 0.9262060 & 0.9262060 & 0.9262060 & 0.926206 & 0.0737940 & 0.0001703 \\
		5 & 0.9999649 & 0.9880588 & 0.9880588 & 0.9880588 & 0.5000075 & 0.5000075 & 0.5000075 & 0.0119412 & 0.0000363 \\
		10 & $>$0.9999999 & 0.9991746 & 0.9827515 & 0.7887339 & 0.7887339 & 0.7887339 & 0.2112661 & 0.0172481 & 0.0000001 \\
		15 & $>$0.9999999 & 0.9998625 & 0.9839465 & 0.8827090 & 0.5000000 & 0.5000000 & 0.117291 & 0.0016991 & $<$0.0000001 \\
		20 & $>$0.9999999 & 0.9999717 & 0.9864589 & 0.6974228 & 0.6974228 & 0.6974228 & 0.07393575 & 0.0002738 & $<$0.0000001 \\
		25 & $>$0.9999999 & 0.9999444 & 0.9889356 & 0.8112051 & 0.5000000 & 0.5000000 & 0.0499461 & 0.0000557 & $<$0.0000001 \\
		30 & $>$0.9999999 & 0.9999870 & 0.9910559 & 0.6572712 & 0.6572712 & 0.6572712 & 0.03525113 & 0.0000801 & $<$0.0000001 \\
		50 & $>$0.9999999 & 0.9999996 & 0.9961734 & 0.6186971 & 0.6186971 & 0.6186971 & 0.01113497 & 0.0000018 & $<$0.0000001 \\
		100 & $>$0.9999999 & $>$0.9999999 & 0.9994253 & 0.729899 & 0.5819812 & 0.4180188 & 0.001268601 & $<$0.0000001 & $<$0.0000001 \\[1ex]
		\hline
	\end{tabular}
}
	\vspace*{0.4cm}
	\caption{Computation of $\varepsilon_{\mathcal{W}_2}(F_{X_{1+[(n-1)\gamma]:n}},F_{Y_{1+[(n-1)\gamma]:n}})$ for varying $n$ and $\gamma$, where $F_X$ and $F_Y$ are the respective cdfs of standard normal distribution and $t$-distribution with $4$ degrees of freedom}
	\label{2-table-order-statistics}
\end{table}

\begin{remark}
	Consider the situation of smallest order statistic ($\gamma=0$) in \cref{2-table-order-statistics}, it is easy to see from the measure of departure from $X_{1:n} \leq_{\textnormal{st}} Y_{1:n}$ that the situation is in fact very close to $Y_{1:n} \leq_{\textnormal{st}} X_{1:n}$, even when $n$ is as small as $2$. On the other hand, considering the situation of largest order statistic ($\gamma=1$), the measure of departure from $X_{n:n} \leq_{\textnormal{st}} Y_{n:n}$ suggests that the stochastic order holds in an approximate sense for $n$ as small as $2$. In practice, the number of components of a system is usually not very large. But the above observations (as well as the computed values in \cref{2-table-order-statistics}) suggest that in many situations involving series and parallel systems (and in general, $k$-out-of-$n$ systems) with iid components, where usual stochastic order does not hold, it may hold in an approximate sense. \hfill $\blacksquare$
\end{remark}\vs

Next, we discuss some asymptotic properties of $\phi_{n,\gamma}$ and $\phi_{n,\gamma}'$, as $n \to \infty$, which prove to be crucial in proving \cref{2-lemma-item-discrete-limit-p-mixture-delta-i}. First, we need the following proposition, where we approximate certain binomial probabilities by means of the central limit theorem (CLT) with continuity corrections.
\begin{proposition}\label{2-proposition-asymptotics-binomial}
	Let $\gamma_n \to \gamma \in [0,1]$, as $n \to \infty$ at the rate $\lvert \gamma_n-\gamma \rvert=O(1/n)$ and let $B_{n,u}$ denote a random variable following the binomial distribution with parameters $n$ and $u$. Then we have
	\begin{flalign}
	&\hspace*{2.96cm} \lim_{n \to \infty} P\left\{B_{n,u} \geq 1+\left[(n-1)\gamma_n\right]\right\}=\frac{1}{2} 1_{\left\{\gamma\right\}}(u)+1_{(\gamma,1]}(u),\label{2-eq-asymptotic-binomial-sf}\\
	&\noindent \textnormal{and}& \nonumber\\
	&\hspace*{2.96cm} \lim_{n \to \infty} nP\left\{B_{n-1,u} = \left[(n-1)\gamma_n\right]\right\}=\begin{cases*}
	0 & if  $u \neq \gamma$,\\
	\infty & if $u=\gamma$.
	\end{cases*} \tag*{$\blacksquare$}
	\end{flalign}
\end{proposition}\vs

By definition, $\phi_{n,\gamma}(u)=P\left\{B_{n,u} \geq 1+\left[(n-1)\gamma\right]\right\}$. Also, from \eqref{2-eq-n-gamma-phi-n-first-derivative}, we see that $\phi_{n,\gamma}'(u)=nP\left\{B_{n-1,u} = \left[(n-1)\gamma\right]\right\}$. The next proposition, which states how $\phi_{n,\gamma}$ and $\phi_{n,\gamma}'$ behave as $n$ becomes large, follows directly from \cref{2-proposition-asymptotics-binomial}. This result will be required to prove \cref{2-lemma-item-discrete-limit-p-mixture-delta-i}.
\begin{proposition}\label{2-proposition-asymptotic-phi_{n,gamma}-and-phi_{n,gamma}-derivative}
	Let $n \in \mathbb{N}$ and $\gamma \in [0,1]$. Then
	\begin{flalign}
		&\hspace*{4.62cm}\lim_{n \to \infty} \phi_{n,\gamma_n}(u)=\frac{1}{2} 1_{\left\{\gamma\right\}}(u)+1_{(\gamma,1]}(u),\label{2-eq-asymptotic-phi_{n,gamma}}\\
		&\noindent \textnormal{and}& \nonumber\\
		&\hspace*{4.62cm}\lim_{n \to \infty} \phi_{n,\gamma_n}'(u)=\begin{cases*}
			0 & if  $u \neq \gamma$,\\
			\infty & if $u=\gamma$.
		\end{cases*} \label{2-eq-asymptotic-phi_{n,gamma}-derivative}
	\end{flalign}
\end{proposition}\vs

\begin{remark}
Informally speaking, \eqref{2-eq-asymptotic-phi_{n,gamma}-derivative} asserts that $\{\phi_{n,\gamma}'\}$ forms a sequence of pulse functions (as in the context of signal processing) that converges to a dirac-delta function $\delta(x-\gamma)$, as $n \to \infty$, where
\[
\delta(x)=\begin{cases*}
	0 & if  $u \neq 0$,\\
	\infty & if $u=0$.
\end{cases*}
\]
A similar structure comes up again in \cref{2-section-applications}, when we discuss asymptotic stochastic order of mixtures of order statistics, where the ``degeneracy" occurs on a finite set of points. \hfill $\blacksquare$
\end{remark}\vs

Now, we turn to asymptotic stochastic precedence order (see \citet{GN_2021_A} for definition and properties) between $X_{1+[(n-1)\gamma_n]:n}$ and $Y_{1+[(n-1)\gamma_n]:n}$. Let $\textnormal{dist}(x,A)=\inf{\{\left\vert x-y \right\vert : y \in A\}}$, for any $x \in \mathbb{R}$ and $A \subseteq \mathbb{R}$. The following result, unlike \cref{2-theorem-n-gamma-rate}, does not require the baseline distributions $F_X$ and $F_Y$ to be strictly increasing or to possess second order moment.
\begin{theorem}\label{2-theorem-n-gamma-rate-stochastic-precedence}
	Let $X_1,X_2,\ldots,X_n$ and $Y_1,Y_2,\ldots,Y_n$ be random samples from respective continuous distributions $F_X$ and $F_Y$. Suppose that $\left\{\gamma_n : n \in \mathbb{N}\right\}$ is a $[0,1]$-valued sequence that converges to $\gamma \in [0,1]$, at the rate $\lvert \gamma_n-\gamma \rvert=O(1/n)$. Then
	\begin{flalign*}
		&\hspace*{3.4cm} X_{1+[(n-1)\gamma_n]:n} \leq_{\textnormal{asp}} Y_{1+[(n-1)\gamma_n]:n} \text{ if } \textnormal{dist}(\gamma, A_0)>0\\
		&\noindent \textnormal{and}&\\
		&\hspace*{3.4cm} X_{1+[(n-1)\gamma_n]:n} =_{\textnormal{asp}} Y_{1+[(n-1)\gamma_n]:n} \text{ if } \textnormal{dist}(\gamma, A_2)>0.
	\end{flalign*}
\end{theorem}

\begin{proof}
	The probability density function (pdf) of $X_{1+[(n-1)\gamma_n]:n}$ and $Y_{1+[(n-1)\gamma_n]:n}$ are respectively given by
	\begin{flalign*}
		&\hspace*{4.4cm} f_{X_{1+[(n-1)\gamma_n]:n}}(x)=\phi_{n,\gamma_n}'(F_X(x))f_X(x),\\
		&\noindent \textnormal{and}&\\
		&\hspace*{4.4cm} f_{Y_{1+[(n-1)\gamma_n]:n}}(x)=\phi_{n,\gamma_n}'(F_Y(x))f_Y(x),
	\end{flalign*}
	for every $x \in \mathbb{R}$. Then
	\begin{align}
	P(X_{1+[(n-1)\gamma_n]:n} \leq Y_{1+[(n-1)\gamma_n]:n})&=\int_{-\infty}^{\infty} \phi_{n,\gamma_n}(F_X(x)) \phi_{n,\gamma_n}'(F_Y(x)) f_Y(x) dx\label{2-eq-asymptotic-stochastic-precedence-distorted-distribution-proof-order-statistics}\\
	&\geq \int_{\mathbb{R} \setminus B_0} \phi_{n,\gamma_n}(F_Y(x)) \phi_{n,\gamma_n}'(F_Y(x)) f_Y(x) dx\nonumber\\
	&=\int_{\phi_{n,\gamma_n}(F_Y(\mathbb{R} \setminus B_0))} u\,du=\frac{1}{2}-\int_{\phi_{n,\gamma_n}(A_0)} u\,du,\nonumber
	\end{align}
	where the inequality follows due to the fact that $F_X(x) \geq F_Y(x)$ for every $x \in \mathbb{R} \setminus B_0$ and $\phi_{n,\gamma_n}$ is nondecreasing, the second equality is obtained by taking the transformation $u=\phi_{n,\gamma_n}(F_Y(x))$ and observing that both $\phi_{n,\gamma_n}$ and $F_Y$ are nondecreasing functions. The last equality is obtained by noting that $\phi_{n,\gamma_n}(F_Y(\mathbb{R} \setminus B_0))=\phi_{n,\gamma_n}([0,1] \setminus F_Y(B_0))=[0,1] \setminus \phi_{n,\gamma_n}(F_Y(B_0))=[0,1] \setminus \phi_{n,\gamma_n}(A_0)$, where the first equality is due to injectivity of $F_Y$ and the third equality is due to \cref{2-lemma-A_0-B_0}. To prove $X_{1+[(n-1)\gamma_n]:n} \leq_{\textnormal{asp}} Y_{1+[(n-1)\gamma_n]:n}$, it is enough to show that $\int_{\phi_{n,\gamma_n}(A_0)} u\,du$ shrinks to $0$, as $n \to \infty$. Let $\epsilon=\textnormal{dist}(\gamma,A_0)>0$. Then, $A_0 \subseteq (0,\gamma-\epsilon) \cup (\gamma+\epsilon,1)$. Thus,
	\begin{align*}
	\int_{\phi_{n,\gamma_n}(A_0)} u\,du &\leq \int_0^{\phi_{n,\gamma_n}(\gamma-\epsilon)} u\,du+\int_{\phi_{n,\gamma_n}(\gamma+\epsilon)}^1 u\,du\\
	&=\frac{\left\{\phi_{n,\gamma_n}(\gamma-\epsilon)\right\}^2}{2}+\frac{1-\left\{\phi_{n,\gamma_n}(\gamma+\epsilon)\right\}^2}{2},
	\end{align*}
	which goes to $0$, as $n \to \infty$, by \cref{2-proposition-asymptotic-phi_{n,gamma}-and-phi_{n,gamma}-derivative}. This completes the first part of the proof. Let $\delta=\textnormal{dist}(\gamma,A_2)>0$. Then, $A_2 \subseteq (0,\gamma-\delta) \cup (\gamma+\delta,1)$. We see that $P\left\{X_{1+[(n-1)\gamma_n]:n} \leq Y_{1+[(n-1)\gamma_n]:n}\right\}$ is equal to
	\begin{align*}
		&\phantom{\,\,\,\,\,\,}\int_{B_2} \phi_{n,\gamma_n}(F_X(x)) \phi_{n,\gamma_n}'(F_Y(x)) f_Y(x)\,dx+\int_{\mathbb{R} \setminus B_2} \phi_{n,\gamma_n}(F_X(x)) \phi_{n,\gamma_n}'(F_Y(x)) f_Y(x)\,dx\\
		&\leq \int_{B_2} \phi_{n,\gamma_n}'(F_Y(x)) f_Y(x)\,dx+\int_{\mathbb{R} \setminus B_2} \phi_{n,\gamma_n}(F_Y(x)) \phi_{n,\gamma_n}'(F_Y(x)) f_Y(x)\,dx\\
		&=\int_{\phi_{n,\gamma_n}(A_2)} du+\int_{-\infty}^{\infty} \phi_{n,\gamma_n}(F_Y(x)) \phi_{n,\gamma_n}'(F_Y(x)) f_Y(x)\,dx\\
		&\leq \left(\int_0^{\phi_{n,\gamma_n}(\gamma-\delta)} du+\int_{\phi_{n,\gamma_n}(\gamma+\delta)}^1 du\right)+\int_0^1 u\,du\\
		&=\phi_{n,\gamma_n}(\gamma-\delta)+1-\phi_{n,\gamma_n}(\gamma+\delta)+\frac{1}{2},
	\end{align*}
	which converges to $1/2$, as $n \to \infty$, by \cref{2-proposition-asymptotic-phi_{n,gamma}-and-phi_{n,gamma}-derivative}. Hence
	\begin{equation}\label{2-eq-n-gamma-rate-stochastic-precedence}
		\lim_{n \to \infty} P\left\{X_{1+[(n-1)\gamma_n]:n} \leq Y_{1+[(n-1)\gamma_n]:n}\right\} \leq 1/2.
	\end{equation}
	Since, $A_0 \subseteq A_2$, $\textnormal{dist}(\gamma,A_2)>0 \Rightarrow \textnormal{dist}(\gamma,A_0)>0$. By the first part of the theorem, it follows that $\lim_{n \to \infty} P\left\{X_{1+[(n-1)\gamma_n]:n} \leq Y_{1+[(n-1)\gamma_n]:n}\right\} \geq 1/2$. Combining this with \eqref{2-eq-n-gamma-rate-stochastic-precedence}, the proof follows.
\end{proof}\vs

\section{Departure-based asymptotic stochastic order of certain stochastic processes}\label{2-section-distorted-distributions}

For any distortion function $\phi$ and a cdf $F$, $\phi(F)$ is again a cdf, often referred to as distorted cdf. Let $T \subseteq \mathbb{R}$ be unbounded above and let $\left\{X_t : t \in T\right\}$ and $\left\{Y_t : t \in T\right\}$ be two stochastic processes. We assume all the random variables involved to possess finite second order moments. Let us denote the respective cdfs of $X_t$ and $Y_t$ by $F_{X_t}$ and $F_{Y_t}$, for $t \in T$. Assume that there exists a class of distortion functions $\left\{\phi_t : t \in T\right\}$ such that $F_{X_t}(x)=\phi_t(F_X(x))$ and $F_{Y_t}(x)=\phi_t(F_Y(x))$ for every $x \in \mathbb{R}$, where $F_X$ and $F_Y$ are two baseline distributions. In this work, it is often beneficial to see the class $\left\{\phi_t : t \in T\right\}$ as a map $t \mapsto \phi_t$ which assigns a distortion function $\phi_t$ to each element of the index set $T$. If $B_0=\emptyset$ (i.e. $F_X \leq_{\textnormal{st}} F_Y$), then we immediately have $F_{X_t} \leq_{\textnormal{st}} F_{Y_t}$, by nondecreasingness of $\phi_t$, for every $t \in T$. On the other hand if $B_0=B_2$, we similarly have $F_{Y_t} \leq_{\textnormal{st}} F_{X_t}$ for every $t \in T$. An interesting situation arises when $B_0$ is neither empty nor is equal to $B_2$, for in this case we cannot conclude anything on the stochastic order between $F_{X_t}$ and $F_{Y_t}$, based on the behaviour of $\phi_t$. The goal of this section is to analyze departure-based asymptotic stochastic order between $F_{X_t}$ and $F_{Y_t}$, as $t \to \infty$.\vs

Let $T=\mathbb{N}$. Suppose that $X_n \overset{P}{\to} b$ and $Y_n \overset{P}{\to} a$, as $n \to \infty$ with $a>b$. Also let $\max{\{EX_n^2,EY_n^2\}}<\infty$, for every $n \in \mathbb{N}$. It can be checked that $P\left\{X_n \leq Y_n\right\}$ can go arbitrarily close to $1$, for $n$ large enough. Counterintuitively, the next example shows that, under the given conditions, it is not necessarily true that $X_n \leq_{\textnormal{d-ast}} Y_n$, as $n \to \infty$.

\begin{example}\label{2-counterexample}
	Let $a>b>0$. Also, for every $n \in \mathbb{N}$, let the respective cdfs $F_{X_n}$ and $F_{Y_n}$, of $X_n$ and $Y_n$, be given by
	\begin{flalign*}
		&\hspace*{2.15cm} F_{X_n}(x)=
		\begin{cases*}
			\frac{1}{2n}e^{x+\frac{\sqrt{n}}{2}} & if  $-\infty < x \leq -\frac{\sqrt{n}}{2}$,\\
			x+\frac{1}{2n}+\frac{\sqrt{n}}{2} & if $-\frac{\sqrt{n}}{2} < x \leq -\frac{\sqrt{n}}{2}+\frac{1}{2n}$,\\
			\left(\frac{\frac{1}{n}}{b+\frac{\sqrt{n}}{2}-\frac{1}{2n}}\right)(x-b)+\frac{2}{n} & if  $-\frac{\sqrt{n}}{2}+\frac{1}{2n} < x \leq b$,\\
			(n-3)(x-b)+\frac{2}{n} & if  $b < x \leq b+\frac{1}{n}$,\\
			1-\frac{1}{n}e^{-(x-b-\frac{1}{n})} & if  $b+\frac{1}{n} < x < \infty$.
		\end{cases*}\\
		&\noindent \textnormal{and}&\\
		&\hspace*{2.2cm} F_{Y_n}(x)=
		\begin{cases*}
			\frac{1}{2n}e^{x+\sqrt{n}} & if  $-\infty < x \leq -\sqrt{n}$,\\
			x+\frac{1}{2n}+\sqrt{n} & if $-\sqrt{n} < x \leq -\sqrt{n}+\frac{1}{2n}$,\\
			\left(\frac{\frac{1}{n}}{a+\sqrt{n}-\frac{1}{2n}}\right)(x-a)+\frac{2}{n} & if  $-\sqrt{n}+\frac{1}{2n} < x \leq a$,\\
			(n-3)(x-a)+\frac{2}{n} & if  $a < x \leq a+\frac{1}{n}$,\\
			1-\frac{1}{n}e^{-(x-a-\frac{1}{n})} & if  $a+\frac{1}{n} < x < \infty$.
		\end{cases*}
	\end{flalign*}
	
	Let us write $A_{0,n}=\{u \in (0,1) : F_{Y_n}^{-1}(u)>F_{X_n}^{-1}(u)\}$. Now, it can be checked that
	\begin{align*}
		&EX_n^2<\infty,\, EY_n^2<\infty,\, X_n \overset{P}{\to} a \text{ and } Y_n \overset{P}{\to} b,\\
		&\int_{A_{0,n}} (F_{X_n}^{-1}(u)-F_{Y_n}^{-1}(u))^2 du>\frac{3}{32}, \text{ for every } n \geq 4,\\
		&\int_0^1 (F_{X_n}^{-1}(u)-F_{Y_n}^{-1}(u))^2du \leq \frac{1}{3}+(b-a)^2+\frac{b-a}{6}, \text{ for every } n \in \mathbb{N}.
	\end{align*}
	Setting $c=\frac{1}{3}+(b-a)^2+\frac{b-a}{6}$, we have $\varepsilon_{\mathcal{W}_2}(F_{X_n},F_{Y_n})>\frac{3}{32c}$, whenever $n \geq 4$. Thus, $\lim_{n \to \infty} \varepsilon_{\mathcal{W}_2}(F_{X_n},F_{Y_n}) \geq \frac{3}{32c}>0$. Hence, we cannot have $X_n \leq_{\textnormal{d-ast}} Y_n$, as $n \to \infty$. \hfill $\blacksquare$
\end{example}\vs

The main result of this section gives sufficient conditions on the asymptotic behaviour of the map $t \mapsto \phi_t$ and two baseline distributions $F_X$ and $F_Y$ to ensure asymptotic stochastic dominance of $\phi_t(F_X)$ over $\phi_t(F_Y)$ as $t \to \infty$. We assume that $F_X$ and $F_Y$ are continuous, strictly increasing and have finite second order moments. Let $A_0$, $B_0$, $A_2$ and $B_2$ be as defined in \cref{2-lemma-A_0-B_0}. Also, for $\gamma \in [0,1]$, let $c_{\gamma}$, $a_{\gamma}$, $d_{\gamma}$ and $b_{\gamma}$ be as defined in \cref{2-theorem-n-gamma-rate}. Further, we assume the following conditions:
\begin{enumerate}[label=\textnormal{(\arabic*)}]
	\item[(C1)]\label{2-item-discrete-phi-n-p} $\phi_t$ is strictly increasing in $[0,1]$, for every $t \geq t_0$, for some $t_0 \in T$.
	\item[(C2)]\label{2-item-discrete-phi-n-p-2} $\phi_t$ is continuously differentiable in $(0,1)$ for every $t \geq t_0$, for some $t_0 \in T$.
	\item[(C3)]\label{2-item-discrete-limit-p} For $i=1,2,\ldots,p$, there exists $\delta_i>0$ such that if $\gamma_i-\delta_i \leq a<b<\gamma_i$ or $\gamma_i<b<a \leq \gamma_i+\delta_i$, then $\phi_t'(a)/\phi_t'(b) \to 0$ as $t \to \infty$.
	\item[(C4)]\label{2-n-item-discrete-limit-p-convex-outside-epsilon} For any $\epsilon>0$, there exists $t_0 \in T$ such that, for every $t \geq t_0$, $\phi_t'$ is convex in $\left[0,\gamma_1-\epsilon\right], \cup \left[\gamma_1+\epsilon,\gamma_2-\epsilon\right] \cup \left[\gamma_2+\epsilon,\gamma_3-\epsilon\right] \cup \cdots \cup \left[\gamma_{p-1}+\epsilon,\gamma_p-\epsilon\right] \cup \left[\gamma_p+\epsilon,1\right]$.
\end{enumerate}\vs

Here \textnormal{(C1)} and \textnormal{(C2)} respectively provide us with overall nature and smoothness of $\phi_t'$, whereas \textnormal{(C3)} and \textnormal{(C4)} describe asymptotic behaviour of $\phi_t'$ near and away from each $\gamma_i$, respectively. As a consequence of these conditions, any random variable $U_t$ with support $[0,1]$ and cdf $\phi_t$, has the probability mass concentrated more and more around the points $\gamma_1,\gamma_2,\ldots,\gamma_p$, as $t \to \infty$. In this work, without loss of generality, we shall consider $\delta_1<(\gamma_2-\gamma_1)/2$, $\delta_i<\min{\left\{(\gamma_i-\gamma_{i-1})/2,(\gamma_{i+1}-\gamma_i)/2\right\}}$ for every $i \in \{2,3,\ldots,p-1\}$ and $\delta_p<(\gamma_p-\gamma_{p-1})/2$, so that for every $i \in \{1,2,\ldots,p-1\}$, $(\gamma_i-\delta_i,\gamma_i+\delta_i) \medcap [0,1]$ and $(\gamma_{i+1}-\delta_{i+1},\gamma_{i+1}+\delta_{i+1}) \medcap [0,1]$ do not intersect each other.\vs

Now, we state the main result of this section, under the setup and the notational scheme described above.
\begin{theorem}\label{2-theorem-general}
	Let $\left\{X_t : t \in T\right\}$ and $\left\{Y_t : t \in T\right\}$ be two stochastic processes. For every $t \in T$, let $X_t$ and $Y_t$ have the respective cdfs $F_{X_t}=\phi_t(F_X)$ and $F_{Y_t}=\phi_t(F_Y)$, where $\left\{\phi_t : t \in T\right\}$ is a sequence of distortion functions satisfying \textnormal{(C1)--(C4)}, and let $F_X$ and $F_Y$ be two baseline distributions, which are continuous, strictly increasing and have finite second order moments satisfying the following conditions for every $i \in \left\{1,2,\ldots,p\right\}$.
	\begin{enumerate}
		\item[\textnormal{(A1)}]\label{2-item-c-a} $\max{\left\{F_X(c_{\gamma_i}),\gamma_i-\delta_i\right\}}<F_X(a_{\gamma_i})$, whenever $A_2 \medcap (0,\gamma_i) \neq \phi$.
		\item[\textnormal{(A2)}]\label{2-item-b-d} $F_X(b_{\gamma_i})<\min{\left\{F_X(d_{\gamma_i}),\gamma_i+\delta_i\right\}}$, whenever $A_2 \medcap (\gamma_i,1) \neq \phi$.
	\end{enumerate}
	Then, $X_t \leq_{\textnormal{d-ast}} Y_t$, as $t \to \infty$. \hfill $\blacksquare$
\end{theorem}\vs

To prove \cref{2-theorem-general}, we need the following lemmas, which are given under the setup of the theorem itself.

\begin{lemma}\label{2-lemma-cauchy-normal-2.r-general}
	Let $t \in T$. Then, $F_{X_t}^{-1}(v)=F_X^{-1}(\phi_t^{-1}(v))$, for every $v \in [0,1]$.
\end{lemma}\vs

\begin{lemma}\label{2-lemma-A_0-A_0-n-distorted-distribution}
	For any $t \in T$, let us define $A_{0,t}=\{v \in (0,1) : F_{X_t}^{-1}(v)>F_{Y_t}^{-1}(v)\}$. Then $A_{0,t}=\phi_t(A_0)$, where $\phi_t(A_0)=\{\phi_t(u) : u \in A_0\}$.
\end{lemma}\vs

The proofs of \cref{2-lemma-cauchy-normal-2.r-general} and \cref{2-lemma-A_0-A_0-n-distorted-distribution} are similar to the respective proofs of \cref{2-lemma-cauchy-normal-2.r} and \cref{2-lemma-A_0-A_0-n}, and hence omitted. The next two lemmas follow from the defining properties of convex function.
\begin{lemma}\label{2-lemma-convex-max}
	Let $-\infty<a<b<\infty$ and let $f:\left[a,b\right] \to \mathbb{R}$ be a convex function. Then, $\max{\left\{f(x) : a \leq x \leq b\right\}}=\max{\left\{f(a),f(b)\right\}}$.
\end{lemma}\vs

\begin{lemma}\label{2-lemma-convex-min}
	Let $-\infty<a<c<b<\infty$ and let $f:\left[a,b\right] \to \mathbb{R}$ be a convex function such that $f(a)>f(c)>f(b)$. Then $\min{\left\{f(x) : a \leq x \leq c\right\}}=f(c)$. Again, if $f(a)<f(c)<f(b)$, then $\min{\left\{f(x) : c \leq x \leq b\right\}}=f(c)$.
\end{lemma}\vs

\begin{proof}[Proof of \cref{2-theorem-general}]
	If $F_X=F_Y$ throughout, then so is $F_{X_t}=F_{Y_t}$. For the sake of nontriviality, let us assume that $F_X$ is not identical to $F_Y$. Now, there can be four possible cases:
	\begin{enumerate}[label=(\roman*)]
		\item\label{2-item-theorem-general-case-i} $0<\gamma_1<\gamma_2<\cdots<\gamma_{p-1}<\gamma_p<1$,
		\item\label{2-item-theorem-general-case-ii} $0<\gamma_1<\gamma_2<\cdots<\gamma_{p-1}<\gamma_p=1$,
		\item\label{2-item-theorem-general-case-iii} $0=\gamma_1<\gamma_2<\cdots<\gamma_{p-1}<\gamma_p<1$,
		\item\label{2-item-theorem-general-case-iv} $0=\gamma_1<\gamma_2<\cdots<\gamma_{p-1}<\gamma_p=1$.
	\end{enumerate}\vs
	
	We shall prove the result for case \ref{2-item-theorem-general-case-i}. The proofs for the other cases are similar. By \textnormal{(A1)} and \textnormal{(A2)}, we have $\sup{\left\{A_0 \cap (0,\gamma_i)\right\}}<\gamma_i<\inf{\left\{A_0 \cap (\gamma_i,1)\right\}}$, and hence, for every $i=1,2,\ldots,p$, we have $\gamma_i \notin A_0$. Let $\gamma_0=0$ and $\gamma_{p+1}=1$. Then
	\begin{align}
	\label{2-eq-partition-of-A_0}
	A_0=A_0 \medcap \left\{\cup_{i=0}^p (\gamma_i,\gamma_{i+1})\right\}=\cup_{i=0}^p \left\{A_0 \medcap (\gamma_i,\gamma_{i+1})\right\}.
	\end{align}
	
	Let us fix $i \in \left\{0,1,\ldots,p\right\}$. In the case where $A_0 \medcap (\gamma_i,\gamma_{i+1}) \neq \emptyset$, we have
	\begin{align}
	\label{2-eq-A_0-gamma_i-gamma_{i+1}}
	A_0 \medcap (\gamma_i,\gamma_{i+1})&=F_X(B_0) \medcap F_X(F_X^{-1}(\gamma_i),F_X^{-1}(\gamma_{i+1}))\nonumber\\
	&=F_X(B_0 \medcap (F_X^{-1}(\gamma_i),F_X^{-1}(\gamma_{i+1}))),
	\end{align}
	where the first equality follows from \cref{2-lemma-A_0-B_0} and injectivity of $F_X$, and the second equality follows from the fact that $F_X$ is nondecreasing and continuous. Now,
	\begin{align*}
	\sup{\left\{A_0 \medcap (\gamma_i,\gamma_{i+1})\right\}}&=\sup{\left\{F_X(B_0 \medcap (F_X^{-1}(\gamma_i),F_X^{-1}(\gamma_{i+1})))\right\}}\nonumber\\
	&\leq \sup{\left\{F_X(B_0 \medcap (-\infty,F_X^{-1}(\gamma_{i+1})))\right\}}\nonumber\\
	&=F_X(\sup{\left\{B_0 \medcap (-\infty,F_X^{-1}(\gamma_{i+1}))\right\}})\nonumber\\
	&=F_X(c_{\gamma_{i+1}}),
	\end{align*}
	where the first equality follows from \eqref{2-eq-A_0-gamma_i-gamma_{i+1}} and the second equality follows from nondecreasingness of $F_X$. Similarly, we have $\inf{\left\{A_0 \medcap (\gamma_i,\gamma_{i+1})\right\}}=F_X(d_{\gamma_i})$. Combining these observations, we have
	\begin{equation}
	\label{2-eq-A_0-gamma_i-gamma_{i+1}-subset}
	A_0 \medcap (\gamma_i,\gamma_{i+1}) \subseteq (F_X(d_{\gamma_i}),F_X(c_{\gamma_{i+1}})), \text{ for } i=0,1,\ldots,p.
	\end{equation}
	Now, we have
	\begin{align}
	\label{2-eq-varepsilon-partition}
	\varepsilon_{\mathcal{W}_2}(F_{X_t},F_{Y_t})&=\frac{\int_{A_{0,t}} (F_{X_t}^{-1}(v)-F_{Y_t}^{-1}(v))^2 dv}{\int_0^1 (F_{X_t}^{-1}(v)-F_{Y_t}^{-1}(v))^2 dv}\nonumber\\
	&=\frac{\int_{A_{0,t}} (F_X^{-1}(\phi_t^{-1}(v))-F_Y^{-1}(\phi_t^{-1}(v)))^2 dv}{\int_0^1 (F_X^{-1}(\phi_t^{-1}(v))-F_Y^{-1}(\phi_t^{-1}(v)))^2 dv}\nonumber\\
	&=\frac{\int_{A_{0}} (F_X^{-1}(u)-F_Y^{-1}(u))^2 \phi_t'(u)du}{\int_0^1 (F_X^{-1}(u)-F_Y^{-1}(u))^2 \phi_t'(u) du}\nonumber\\
	&=\sum_{i=0}^p \frac{\int_{A_0 \medcap (\gamma_i,\gamma_{i+1})} (F_X^{-1}(u)-F_Y^{-1}(u))^2 \phi_t'(u)du}{\int_0^1 (F_X^{-1}(u)-F_Y^{-1}(u))^2 \phi_t'(u) du},
	\end{align}
	where $A_{0,t}=\{v \in (0,1) : F_{X_t}^{-1}(v)>F_{Y_t}^{-1}(v)\}$, the second equality is due to \cref{2-lemma-cauchy-normal-2.r-general}, the third equality follows by taking the transformation $\phi_t^{-1}(v)=u$ and using \cref{2-lemma-A_0-A_0-n-distorted-distribution} and the fourth equality follows from \eqref{2-eq-partition-of-A_0}. If $A_0 \medcap (\gamma_i,\gamma_{i+1})=\phi$, then the $(i+1)$th term does not contribute to $\varepsilon_{\mathcal{W}_2}(F_{Y_t},F_{X_t})$. On the other hand, if $A_0 \medcap (\gamma_i,\gamma_{i+1}) \neq \phi$, then by \eqref{2-eq-A_0-gamma_i-gamma_{i+1}-subset},
	\begin{equation*}
	\frac{\int_{A_0 \medcap (\gamma_i,\gamma_{i+1})} (F_X^{-1}(u)-F_Y^{-1}(u))^2 \phi_t'(u)du}{\int_0^1 (F_X^{-1}(u)-F_Y^{-1}(u))^2 \phi_t'(u) du} \leq \frac{\int_{F_X(d_{\gamma_i})}^{F_X(c_{\gamma_{i+1}})} (F_X^{-1}(u)-F_Y^{-1}(u))^2 \phi_t'(u)du}{\int_0^1 (F_X^{-1}(u)-F_Y^{-1}(u))^2 \phi_t'(u) du}.
	\end{equation*}
	Now, there can be two cases:\vs
	
	\ni{\bf Case $\mathbf{(i)}$:} $\phi_t'(F_X(d_{\gamma_i})) \leq \phi_t'(F_X(c_{\gamma_{i+1}}))$. In this case, for large $t$, we have
	\begin{align*}
	&\phantom{\,\,\,\,\,\,\,\,}\frac{\int_{F_X(d_{\gamma_i})}^{F_X(c_{\gamma_{i+1}})} (F_X^{-1}(u)-F_Y^{-1}(u))^2 \phi_t'(u)du}{\int_0^1 (F_X^{-1}(u)-F_Y^{-1}(u))^2 \phi_t'(u) du}\\
	&\leq \frac{\int_{F_X(d_{\gamma_i})}^{F_X(c_{\gamma_{i+1}})} (F_X^{-1}(u)-F_Y^{-1}(u))^2 \phi_t'(u)du}{\int_{F_X(a_{\gamma_{i+1}})-\epsilon}^{F_X(a_{\gamma_{i+1}})-\frac{\epsilon}{2}} (F_X^{-1}(u)-F_Y^{-1}(u))^2 \phi_t'(u) du},\\
	&\leq \frac{\phi_t'(F_X(c_{\gamma_{i+1}})) \int_{F_X(d_{\gamma_i})}^{F_X(c_{\gamma_{i+1}})} (F_X^{-1}(u)-F_Y^{-1}(u))^2 du}{\phi_t'(F_X(a_{\gamma_{i+1}})-\epsilon) \int_{F_X(a_{\gamma_{i+1}})-\epsilon}^{F_X(a_{\gamma_{i+1}})-\frac{\epsilon}{2}} (F_X^{-1}(u)-F_Y^{-1}(u))^2 du}\\
	&\leq \frac{\phi_t'(\max{\left\{F_X(c_{\gamma_{i+1}}),\gamma_{i+1}-\delta_{i+1}\right\}}) \int_{F_X(d_{\gamma_i})}^{F_X(c_{\gamma_{i+1}})} (F_X^{-1}(u)-F_Y^{-1}(u))^2 du}{\phi_t'(F_X(a_{\gamma_{i+1}})-\epsilon) \int_{F_X(a_{\gamma_{i+1}})-\epsilon}^{F_X(a_{\gamma_{i+1}})-\frac{\epsilon}{2}} (F_X^{-1}(u)-F_Y^{-1}(u))^2 du},
	\end{align*}
	where $\epsilon<F_X(a_{\gamma_{i+1}})-(\gamma_{i+1}-\delta_{i+1})$ is an arbitrarily small positive real number, the second inequality follows from \textnormal{(C4)}, \cref{2-lemma-convex-max} and \cref{2-lemma-convex-min}. Note that there exists $t_1,t_2 \in T$ such that $t \geq t_1 \Rightarrow \phi_t'(F_X(a_{\gamma_{i+1}})-\epsilon)/\phi_t'(F_X(a_{\gamma_{i+1}})-\epsilon/2)<1$ and $t \geq t_2 \Rightarrow \phi_t'(\gamma_{i+1}-\delta_{i+1})/\phi_t'(F_X(a_{\gamma_{i+1}})-\epsilon)<1$. Hence, for every $t \geq t_0=\max{\left\{t_1,t_2\right\}}$, we have $\phi_t'(\gamma_{i+1}-\delta_{i+1})<\phi_t'(F_X(a_{\gamma_{i+1}})-\epsilon)<\phi_t'(F_X(a_{\gamma_{i+1}})-\epsilon/2)$. The last inequality follows from \cref{2-lemma-convex-max}. On the other hand, if $F_X(c_{\gamma_{i+1}}) \geq \gamma_{i+1}-\delta_{i+1}$, then the inequality becomes an equality.\vs
	
	\ni{\bf Case $\mathbf{(ii)}$:} $\phi_t'(F_X(d_{\gamma_i}))>\phi_t'(F_X(c_{\gamma_{i+1}}))$. Proceeding similarly as in Case $(i)$, we have
	\begin{align*}
	&\phantom{\,\,\,\,\,\,\,\,}\frac{\int_{F_X(d_{\gamma_i})}^{F_X(c_{\gamma_{i+1}})} (F_X^{-1}(u)-F_Y^{-1}(u))^2 \phi_t'(u)du}{\int_0^1 (F_X^{-1}(u)-F_Y^{-1}(u))^2 \phi_t'(u) du}\\
	&\leq \frac{\phi_t'(\min{\left\{F_X(d_{\gamma_i}),\gamma_i+\delta_i\right\}})\int_{F_X(d_{\gamma_i})}^{F_X(c_{\gamma_{i+1}})} (F_X^{-1}(u)-F_Y^{-1}(u))^2 du}{\phi_t'(F_X(b_{\gamma_i})+\epsilon)\int_{F_X(b_{\gamma_i})+\frac{\epsilon}{2}}^{F_X(b_{\gamma_i})+\epsilon} (F_X^{-1}(u)-F_Y^{-1}(u))^2 du}.
	\end{align*}
	
	Substituting each of the ratios of integrals in the right hand side of \eqref{2-eq-varepsilon-partition} and combining the two cases, we obtain
	\begin{equation}\label{2-eq-varepsilon-partition-upper-bound}
	\varepsilon_{\mathcal{W}_2}(F_{X_t},F_{Y_t}) \leq \sum_{i=0}^p k_{\epsilon,\gamma_i,\gamma_{i+1}} Z_{t,\epsilon,\gamma_i,\gamma_{i+1}},
	\end{equation}
	where, for $i=0,1,\ldots,p$,
	\[
	k_{\epsilon,\gamma_i,\gamma_{i+1}}=\max{\left\{\frac{\int_{F_X(d_{\gamma_i})}^{F_X(c_{\gamma_{i+1}})} (F_X^{-1}(u)-F_Y^{-1}(u))^2 du}{\int_{\gamma_i+\frac{\epsilon}{2}}^{\gamma_i+\epsilon} (F_X^{-1}(u)-F_Y^{-1}(u))^2 du},\frac{\int_{F_X(d_{\gamma_i})}^{F_X(c_{\gamma_{i+1}})} (F_X^{-1}(u)-F_Y^{-1}(u))^2 du}{\int_{\gamma_{i+1}-\epsilon}^{\gamma_{i+1}-\frac{\epsilon}{2}} (F_X^{-1}(u)-F_Y^{-1}(u))^2 du}\right\}},
	\]
	if $A_0 \medcap (\gamma_i,\gamma_{i+1}) \neq \phi$ and equal to $0$ otherwise, and
	\[
	Z_{t,\epsilon,\gamma_i,\gamma_{i+1}}=\max{\left\{\frac{\phi_t'(\min{\left\{F_X(d_{\gamma_i}),\,\gamma_i+\delta_i\right\}})}{\phi_t'(\gamma_i+\epsilon)},\frac{\phi_t'(\max{\left\{F_X(c_{\gamma_{i+1}}),\,\gamma_{i+1}-\delta_{i+1}\right\}})}{\phi_t'(\gamma_{i+1}-\epsilon)}\right\}},
	\]
	if $A_0 \medcap (\gamma_i,\gamma_{i+1}) \neq \phi$ and equal to $0$ otherwise. It follows that $0 \leq k_{\epsilon,\gamma_i,\gamma_{i+1}}<\infty$. If $A_0 \medcap (\gamma_i,\gamma_{i+1})=\phi$, then $Z_{t,\epsilon,\gamma_i,\gamma_{i+1}}=0$, for every $t \in T$. Let us now assume that $A_0 \medcap (\gamma_i,\gamma_{i+1}) \neq \phi$. Note that $\gamma_i<\gamma_i+\epsilon<\min{\left\{F_X(d_{\gamma_i}),\,\gamma_i+\delta_i\right\}} \leq \gamma_i+\delta_i$ and $\gamma_{i+1}-\delta_{i+1}<\max{\left\{\gamma_{i+1}-\delta_{i+1},\,F_X(c_{\gamma_{i+1}})\right\}}<\gamma_{i+1}-\epsilon<\gamma_{i+1}$, for small enough $\epsilon>0$. Thus, by \textnormal{(C3)}, $Z_{t,\epsilon,\gamma_i,\gamma_{i+1}} \to 0$, as $t \to \infty$. The proof now follows by taking limit, as $t \to \infty$ on both sides of \eqref{2-eq-varepsilon-partition-upper-bound}.
\end{proof}\vs

The following proposition is not only useful in the rest of the paper but also provides intuition on asymptotic behaviour of $\phi_t'$, as $t \to \infty$.
\begin{proposition}\label{2-lemma-phi-t-derivative-vanishes-away-from-gamma-i}
	Let $\left\{\phi_t : t \in T\right\}$ be a class of distortion functions, satisfying \textnormal{(C1)--(C4)}. Then $\phi_t'(u) \to 0$, as $t \to \infty$, whenever $u \in [0,1] \setminus \left\{\gamma_1,\gamma_2,\ldots,\gamma_p\right\}$.
\end{proposition}

\begin{proof}
	Let us fix $i \in \left\{1,2,\ldots,p\right\}$ and assume that $\gamma_i<u<\gamma_{i+1}$. The case when $0 \leq u<\gamma_1$ or $\gamma_p<u \leq 1$ will be considered afterwards. For every $i \in \left\{1,2,\ldots,p\right\}$, let $\delta_i$ be as in \textnormal{(C3)}. Now, we have the following cases.\vs
	
	{\bf Case 1.} $\gamma_i<u \leq \gamma_i+\delta_i$. Suppose, for the sake of contradiction, that $\phi_t'(u) \nrightarrow 0$, as $t \to \infty$. Hence, there exists $\epsilon>0$ and a sequence $\left\{t_n : n \in \mathbb{N}\right\}$ such that $t_n \to \infty$, as $n \to \infty$ and $\phi_{t_n}'(u)>\epsilon$ for every $n \in \mathbb{N}$. Let us fix $v \in (\gamma_i,u)$ and $w \in (\gamma_i,v)$. By \textnormal{(C3)}, $\lim_{n \to \infty} \phi_{t_n}'(u)/\phi_{t_n}'(v)=0$ and $\lim_{n \to \infty} \phi_{t_n}'(v)/\phi_{t_n}'(w)=0$. Thus, there exists $N_1 \in \mathbb{N}$ and $N_2 \in \mathbb{N}$ such that $\phi_{t_n}'(u)/\phi_{t_n}'(v)<1$, whenever $n \geq N_1$ and $\phi_{t_n}'(v)/\phi_{t_n}'(w)<1$, whenever $n \geq N_2$. Again, for any $\alpha>0$, there exists $N_3 \in \mathbb{N}$ such that $n \geq N_3$ implies $\phi_{t_n}'(u)/\phi_{t_n}'(v)<\alpha$, i.e., $\phi_{t_n}'(v)>\phi_{t_n}'(u)/\alpha>\epsilon/\alpha$. Since $w<v<u$ by choice and $\phi_{t_n}'(w)>\phi_{t_n}'(v)>\phi_{t_n}'(u)$, for every $n \geq \max{\left\{N_1,N_2\right\}}$, it follows from \cref{2-lemma-convex-min} that $\min{\left\{\phi_{t_n}'(z) : w \leq z \leq v\right\}}=\phi_{t_n}'(v)$, for every $n \geq \max{\left\{N_1,N_2\right\}}$. Then, for every $n \geq \max{\left\{N_1,N_2,N_3\right\}}$, we have
	\[
	\int_w^v \phi_{t_n}'(z) dz \geq (v-w) \phi_{t_n}'(v)>\frac{(v-w)\epsilon}{\alpha}.
	\]
	This is a contradiction since the integral in the left hand side is equal to $\phi_{t_n}(v)-\phi_{t_n}(w)$, which can at most be $1$ and the ratio in the right hand side can be made arbitrarily large by choosing $\alpha>0$ to be small enough. Hence we must have $\lim_{t \to \infty} \phi_t'(u)=0$.\vs
	
	{\bf Case 2.} $\gamma_{i+1}-\delta_{i+1} \leq u<\gamma_{i+1}$. By a similar line of argument, based on contradiction, as in the previous case, one can show that $\lim_{t \to \infty} \phi_t'(u)=0$.\vs
	
	{\bf Case 3.} $\gamma_i+\delta_i<u<\gamma_{i+1}-\delta_{i+1}$. It follows from \cref{2-lemma-convex-max} that, for every $t \geq t_0$, we have $\phi_t'(u) \leq \max{\left\{\phi_t'(\gamma_i+\delta_i),\phi_t'(\gamma_{i+1}-\delta_{i+1})\right\}}$. From the previous two cases, we have $\phi_t'(\gamma_i+\delta_i) \to 0$ and $\phi_t'(\gamma_{i+1}-\delta_{i+1}) \to 0$, as $t \to \infty$. Thus, $\lim_{t \to \infty} \max{\left\{\phi_t'(\gamma_i+\delta_i),\phi_t'(\gamma_{i+1}-\delta_{i+1})\right\}}=0$. Consequently $\lim_{t \to \infty} \phi_t'(u)=0$.\vs
	
	{\bf Case 4.} $0 \leq u<\gamma_1$. This case can be treated similarly using the same arguments as in Case 2 (for $0 \leq u<\gamma_1-\delta_1$) and Case 3 (for $\gamma_1-\delta_1 \leq u<\gamma_1$).\vs
	
	{\bf Case 5.} $\gamma_p<u \leq 1$. This case can be dealt with similarly using the arguments given in Case 1 (for $\gamma_p<u \leq \gamma_p+\delta_p$) and Case 2 (for $\gamma_p+\delta_p<u \leq 1$).
\end{proof}\vs

\section{Applications}\label{2-section-applications}

\subsection*{Mixtures of order statistics}\label{2-subsection-mixture}

Let $F_1,F_2,\ldots,F_p$ be $p$ cdfs. Consider the finite mixture of these cdfs $F(x)=\sum_{i=1}^p \alpha_i F_i(x)$, where $\alpha_i \geq 0$ for every $i \in \{1,2,\ldots,p\}$ and $\sum_{i=1}^p \alpha_i=1$. Mixture distributions have significant applications in statistics, text classification, speech recognition, disease mapping, meta analysis among other fields of study. In particular, mixtures of order statistics have attracted special interests in various contexts. Let us consider a coherent system with $n$ components, with respective lifetimes $X_1,X_2,\ldots,X_n$. Then the cdf of system lifetime, $T$, have the representation, given by
\[
F_T(x)=\sum_{i=1}^n s_i F_{X_{i:n}}(x),
\]
for every $x \in \mathbb{R}$, where $(s_1,s_2,\ldots,s_n)$ is the signature vector, which provides us with distribution-free characterization of the design of the system (see \citet{S_2007}). In the present section, we explore asymptotic stochastic comparison of finite mixtures of order statistics, generated from two homogeneous samples.\vs

Let $F_X$ and $F_Y$ be two continuous, strictly increasing cdfs and let $p \in \mathbb{N}$. Assume that $\boldsymbol{\gamma}=(\gamma_1,\gamma_2,\ldots,\gamma_p) \in [0,1]^p$ and $\boldsymbol{\gamma}_n=(\gamma_{1,n},\gamma_{2,n},\ldots,\gamma_{p,n}) \in [0,1]^p$, where $\gamma_{i,n} \to \gamma_i$ with the rate of convergence $\left\vert \gamma_{i,n}-\gamma_i \right\vert=O(1/n)$, for $i=1,2,\ldots,p$. Also, let $\boldsymbol{\alpha}=(\alpha_1,\alpha_2,\ldots,\alpha_p) \in (0,1)^p$ with $\sum_{i=1}^p \alpha_i=1$. Let us define the mixtures of order statistics (characterized by $\boldsymbol{\alpha}$ and $\boldsymbol{\gamma}_n$) based on the respective random samples (of size $n$) from $F_X$ and $F_Y$ as
\[
F_{X,n,\boldsymbol{\alpha},\boldsymbol{\gamma}_n}(t)=\sum_{i=1}^p \alpha_i F_{X_{(1+[(n-1)\gamma_{i,n}])}}(t),\,\, F_{Y,n,\boldsymbol{\alpha},\boldsymbol{\gamma}_n}(t)=\sum_{i=1}^p \alpha_i F_{Y_{(1+[(n-1)\gamma_{i,n}])}}(t),
\]
for every $t \in \mathbb{R}$. We wish to examine the asymptotic stochastic order between the random variables with respective cdfs $F_{X,n,\boldsymbol{\alpha},\boldsymbol{\gamma}_n}$ and $F_{Y,n,\boldsymbol{\alpha},\boldsymbol{\gamma}_n}$. Now, we have
\[
F_{X,n,\boldsymbol{\alpha},\boldsymbol{\gamma}_n}(t)=\sum_{i=1}^p \alpha_i \left[\sum_{j=1+[(n-1)\gamma_{i,n}]}^n \binom{n}{j} (F_X(t))^j (1-F_X(t))^{n-j}\right],
\]
for every $t \in \mathbb{R}$. We can rewrite it as $F_{X,n,\boldsymbol{\alpha},\boldsymbol{\gamma}_n}(t)=\phi_{n,\boldsymbol{\alpha},\boldsymbol{\gamma}_n}(F_X(t))$, where
\begin{equation}
\label{2-eq-mixture-phi-n-alpha-gamma-n}
\phi_{n,\boldsymbol{\alpha},\boldsymbol{\gamma}_n}(u)=\sum_{i=1}^p \alpha_i\,\phi_{n,\gamma_{i,n}}(u), \text{ for every } u \in [0,1].
\end{equation}
Under the above setup, we have the following results which are needed to prove the main result, \cref{2-theorem-mixture}. The next lemma is required to prove \cref{2-proposition-item-discrete-limit-p-mixture-delta-i-ratio}.
\begin{lemma}
	\label{2-lemma-item-discrete-limit-p-mixture-delta-i}
	For every $i \in \left\{1,2,\ldots,p\right\}$, there exists $\delta_i>0$ and $N \in \mathbb{N}$ such that $\phi_{n,\gamma_{i,n}}'(u) \geq \phi_{n,\gamma_{j,n}}'(u)$, for every $u \in (\gamma_i-\delta_i,\gamma_i+\delta_i) \medcap [0,1]$, $j \in \left\{1,2,\ldots,p\right\}$ and $n \geq N$.
\end{lemma}

\begin{proof}
	Let us fix $i \in \left\{1,2,\ldots,p\right\}$. By \cref{2-proposition-asymptotic-phi_{n,gamma}-and-phi_{n,gamma}-derivative}, we have $\phi_{n,\gamma_{i,n}}'(u) \to \infty$, if $u=\gamma_i$ and $\phi_{n,\gamma_{i,n}}'(u) \to 0$, otherwise. If $j=i$, the proof is trivial. Thus, we consider $j \in \left\{1,2,\ldots,p\right\} \setminus \left\{i\right\}$. Then, $\phi_{n,\gamma_{j,n}}'(u) \to \infty$ if $u=\gamma_j$ and $\phi_{n,\gamma_{j,n}}'(u) \to 0$ otherwise. Thus, for large $n$, $\phi_{n,\gamma_{i,n}}'(\gamma_i)>\phi_{n,\gamma_{j,n}}'(\gamma_i)$ and $\phi_{n,\gamma_{i,n}}'(\gamma_j)<\phi_{n,\gamma_{j,n}}'(\gamma_j)$. By \textnormal{(C1)}, both $\phi_{n,\gamma_{i,n}}'$ and $\phi_{n,\gamma_{j,n}}'$ are continuous. Hence they must cross each other between $\gamma_i$ and $\gamma_j$. Let this point be $u_{i,j,n}$. We shall show that $\left\{u_{i,j,n} : n \in \mathbb{N}\right\}$ forms a convergent sequence and its limit $u_{i,j}$ is bounded away from both $\gamma_i$ and $\gamma_j$, i.e.,
	\begin{equation}
		\label{2-eq-lemma-mixture-delta-i-j}
		\delta_{i,j}:=\min{\left\{\lvert \gamma_i-u_{i,j} \rvert, \lvert \gamma_j-u_{i,j} \rvert\right\}}>0.
	\end{equation}
	Note that $\phi_{n,\gamma_{i,n}}'(u_{i,j,n})=\phi_{n,\gamma_{j,n}}'(u_{i,j,n})$. From \eqref{2-eq-n-gamma-phi-n-first-derivative}, this is equivalent to
	\begin{align}
		\label{2-eq-lemma-mixture-delta-i}
		&\binom{n-1}{\left[(n-1)\gamma_{i,n}\right]} u_{i,j,n}^{\left[(n-1)\gamma_{i,n}\right]} (1-u_{i,j,n})^{n-1-\left[(n-1)\gamma_{i,n}\right]}\nonumber\\
		&=\binom{n-1}{\left[(n-1)\gamma_{j,n}\right]} u_{i,j,n}^{\left[(n-1)\gamma_{j,n}\right]} (1-u_{i,j,n})^{n-1-\left[(n-1)\gamma_{j,n}\right]}.
	\end{align}
	We shall need the following identity.
	\begin{equation}\label{2-eq-mixture-combination-identity}
		\binom{n-1}{\left[(n-1)x\right]}^{-1} \binom{n-1}{\left[(n-1)y\right]}=\prod_{k=\left[(n-1)x\right]+1}^{\left[(n-1)y\right]}\frac{n-k}{k},
	\end{equation}
	for every $0 \leq x<y \leq 1$. Now, we consider the cases when $j>i$ and $j<i$ separately.\vs
	
	\ni{\bf Case (i):} $j>i$. In this case, $\gamma_j>\gamma_i$ and hence $\gamma_{j,n}>\gamma_{i,n}$ for large $n$. Then, it follows from \eqref{2-eq-lemma-mixture-delta-i} and \eqref{2-eq-mixture-combination-identity} that
	\begin{equation*}
		\left(\frac{1-u_{i,j,n}}{u_{i,j,n}}\right)^{\left[(n-1)\gamma_{j,n}\right]-\left[(n-1)\gamma_{i,n}\right]}=\prod_{k=\left[(n-1)\gamma_{i,n}\right]+1}^{\left[(n-1)\gamma_{j,n}\right]}\frac{n-k}{k}.
	\end{equation*}
	Taking logarithm, it follows that
	\begin{equation}\label{2-eq-lemma-mixture-delta-i-log}
		\ln{\left(\frac{1-u_{i,j,n}}{u_{i,j,n}}\right)}=\frac{n}{\left[(n-1)\gamma_{j,n}\right]-\left[(n-1)\gamma_{i,n}\right]} \sum_{k=\left[(n-1)\gamma_{i,n}\right]+1}^{\left[(n-1)\gamma_{j,n}\right]} \frac{1}{n} f\left(\frac{k}{n}\right),
	\end{equation}
	where $f(x)=\ln{(\frac{1-x}{x})}$. As $n \to \infty$, the right hand side of \eqref{2-eq-lemma-mixture-delta-i-log} converges to
	\[
	\frac{1}{\gamma_j-\gamma_i} \int_{\gamma_i}^{\gamma_j} f(x) dx.
	\]
	Thus, the left hand side of \eqref{2-eq-lemma-mixture-delta-i-log} must converge and hence $u_{i,j,n}$ converges to some limit $u_{i,j}$, as $n \to \infty$. Then
	\[
	\ln{\left(\frac{1-u_{i,j}}{u_{i,j}}\right)}=\frac{1}{\gamma_j-\gamma_i} \int_{\gamma_i}^{\gamma_j} \ln{\left(\frac{1-x}{x}\right)} dx.
	\]
	Since, $\ln{(\frac{1-x}{x})}$ is strictly decreasing in $x$, we have
	\[
	\ln{\left(\frac{1-\gamma_j}{\gamma_j}\right)}<\ln{\left(\frac{1-u_{i,j}}{u_{i,j}}\right)}<\ln{\left(\frac{1-\gamma_i}{\gamma_i}\right)}.
	\]
	It follows that $\gamma_i<u_{i,j}<\gamma_j$, which implies \eqref{2-eq-lemma-mixture-delta-i-j}.\vs
	
	\ni{\bf Case (ii):} $j<i$. In this case, $\gamma_{j,n}<\gamma_{i,n}$. Proceeding similarly as in Case (i), we have $\gamma_j<u_{i,j}<\gamma_i$, which implies \eqref{2-eq-lemma-mixture-delta-i-j}.\vs
	
	The proof now follows by choosing $\delta_i=\min{\left\{\delta_{i,j} : j \in \left\{1,2,\ldots,p\right\} \setminus \left\{i\right\}\right\}}$.
\end{proof}\vs

\begin{remark}
	Although it is not required for the proof of \cref{2-lemma-item-discrete-limit-p-mixture-delta-i} itself, it is interesting to note that
	\[
	\min{\left\{\delta_{i,j} : j \in \left\{1,2,\ldots,p\right\} \setminus \left\{i\right\}\right\}}=\begin{cases*}
		\delta_{1,2} & if  $i=1$,\\
		\min{\left\{\delta_{i,i-1},\delta_{i,i+1}\right\}}, & if $i \in \left\{2,\ldots,p-1\right\}$,\\
		\delta_{p-1,p} & if  $i=p$.
	\end{cases*}
	\]	
	Hence, in the choice of $\delta_i$, the minimum is taken over at most $2$ positive real numbers (instead of $p-1$). To see this, it is enough to notice that $(i)$ $u_{i,i-1}<\gamma_i<u_{i,i+1}$ and $(ii)$ $u_{i,j}$ is nondecreasing in $j \in \left\{1,2,\ldots,p\right\} \setminus \left\{i\right\}$, for every $i \in \left\{1,2,\ldots,p\right\}$. To show the latter fact, let us choose $j$ and $k$ from the set $\left\{1,2,\ldots,p\right\} \setminus \left\{i\right\}$ such that $j<k$. If $j<i<k$, then $u_{i,j}<\gamma_i<u_{i,k}$. If $i<j<k$, then for large $n$, $u<u_{i,j} \Rightarrow \phi_{n,\gamma_{i,n}}'(u)>\phi_{n,\gamma_{j,n}}'(u)$. Again, for large $n$, $u<u_{j,k} \Rightarrow \phi_{n,\gamma_{j,n}}'(u)>\phi_{n,\gamma_{k,n}}'(u)$. Note that $u_{i,j}<\gamma_j<u_{j,k}$. Thus, whenever $u<u_{i,j}$, we have $\phi_{n,\gamma_{i,n}}'(u)>\phi_{n,\gamma_{k,n}}'(u)$. Hence $u_{i,j} \leq u_{i,k}$. If $j<k<i$, it can be shown that $u_{i,j} \leq u_{i,k}$ by similar line of arguments.
\end{remark}\vs

\begin{proposition}
	\label{2-proposition-item-discrete-limit-p-mixture-delta-i-ratio}
	For every $i \in \left\{1,2,\ldots,p\right\}$, there exists $\delta_i>0$ such that, if $\gamma_i-\delta_i<a<b<\gamma_i$ or $\gamma_i>b>a>\gamma_i+\delta_i$, then
	\begin{equation}\label{2-eq-item-discrete-limit-p-mixture-delta-i-ratio}
		\frac{\phi_{n,\boldsymbol{\alpha},\boldsymbol{\gamma}_n}'(a)}{\phi_{n,\boldsymbol{\alpha},\boldsymbol{\gamma}_n}'(b)} \to 0 \text{ as } n \to \infty.
	\end{equation}
\end{proposition}

\begin{proof}
	By \cref{2-lemma-item-discrete-limit-p-mixture-delta-i}, there exists positive real numbers $\delta_1,\delta_2,\ldots,\delta_p$ such that, for every $(i,j) \in \left\{1,2,\ldots,p\right\} \times \left\{1,2,\ldots,p\right\}$, $\phi_{n,\gamma_{i,n}}'(u) \geq \phi_{n,\gamma_{j,n}}'(u)$, whenever $u \in (\gamma_i-\delta_i,\gamma_i+\delta_i) \medcap [0,1]$. Now, let us fix $i \in \left\{1,2,\ldots,p\right\}$ and $\gamma_i<b<a<\gamma_i+\delta_i$. Then, by \eqref{2-eq-n-gamma-phi-n-first-derivative}, we have
	\begin{align*}
		\frac{\phi'_{n,\boldsymbol{\alpha},\boldsymbol{\gamma}_n}(a)}{\phi'_{n,\boldsymbol{\alpha},\boldsymbol{\gamma}_n}(b)}&=\frac{\sum_{j=1}^p \alpha_j\,\phi'_{n,\gamma_{j,n}}(a)}{\sum_{j=1}^p \alpha_j\,\phi'_{n,\gamma_{j,n}}(b)}\\
		&=\frac{\sum_{j=1}^i \alpha_j\,\phi'_{n,\gamma_{j,n}}(a)+\sum_{j=i+1}^p \alpha_j\,\phi'_{n,\gamma_{j,n}}(a)}{\sum_{j=1}^i \alpha_j\,\phi'_{n,\gamma_{j,n}}(b)+\sum_{j=i+1}^p \alpha_j\,\phi'_{n,\gamma_{j,n}}(b)}\\
		&\leq \sum_{j=1}^i \frac{\phi_{n,\gamma_{j,n}}'(a)}{\phi_{n,\gamma_{j,n}}'(b)}+\sum_{j=i+1}^p \frac{\alpha_j\phi_{n,\gamma_{j,n}}'(a)}{\alpha_i\phi_{n,\gamma_{i,n}}'(b)}\\
		&\leq \sum_{j=1}^i \frac{\phi_{n,\gamma_{j,n}}'(a)}{\phi_{n,\gamma_{j,n}}'(b)}+\frac{\phi_{n,\gamma_{i,n}}'(a)}{\phi_{n,\gamma_{i,n}}'(b)}\sum_{j=i+1}^p \frac{\alpha_j}{\alpha_i},
	\end{align*}
	where the last inequality follows from \cref{2-lemma-item-discrete-limit-p-mixture-delta-i}. Applying \cref{2-corollary-phi_n-r-2-limit} on $\phi_{n,\gamma_{1,n}}$, $\phi_{n,\gamma_{2,n}}$, $\ldots$ , $\phi_{n,\gamma_{i,n}}$, we have $\phi_{n,\gamma_{j,n}}'(a)/\phi_{n,\gamma_{j,n}}'(b) \to 0$, as $n \to \infty$, for every $j \in \left\{1,2,\ldots,i\right\}$, which implies \eqref{2-eq-item-discrete-limit-p-mixture-delta-i-ratio}. On the other hand, if $\gamma_i-\delta<a<b<\gamma_i$, then proceeding in a similar line, we obtain
	\begin{equation*}
		\frac{\phi_{n,\boldsymbol{\alpha},\boldsymbol{\gamma}_n}(a)}{\phi_{n,\boldsymbol{\alpha},\boldsymbol{\gamma}_n}(b)} \leq \frac{\phi_{n,\gamma_{i,n}}'(a)}{\phi_{n,\gamma_{i,n}}'(b)}\sum_{j=1}^{i-1} \frac{\alpha_j}{\alpha_i}+\sum_{j=i}^p \frac{\phi_{n,\gamma_{j,n}}'(a)}{\phi_{n,\gamma_{j,n}}'(b)}.
	\end{equation*}
	Again, applying \cref{2-corollary-phi_n-r-2-limit} on $\phi_{n,\gamma_{i,n}},\ldots,\phi_{n,\gamma_{p,n}}$, we have $\phi_{n,\gamma_{j,n}}'(a)/\phi_{n,\gamma_{j,n}}'(b) \to 0$, as $n \to \infty$, for every $j \in \left\{i,\ldots,p\right\}$, which implies \eqref{2-eq-item-discrete-limit-p-mixture-delta-i-ratio}. Hence the proof is established.
\end{proof}\vs

\begin{lemma}
	\label{2-lemma-mixture-convex-outside-epsilon}
	For every $\epsilon>0$, there exists $N \in \mathbb{N}$ such that $\phi_{n,\boldsymbol{\alpha},\boldsymbol{\gamma}_n}'$ is convex in $\left[0,\gamma_1-\epsilon\right] \cup \left[\gamma_1+\epsilon,\gamma_2-\epsilon\right] \cup \left[\gamma_2+\epsilon,\gamma_3-\epsilon\right] \cup \cdots \cup \left[\gamma_{p-1}+\epsilon,\gamma_p-\epsilon\right] \cup \left[\gamma_p+\epsilon,1\right]$, for every $n \geq N$.
\end{lemma}

\begin{proof}
	Note that the convexity of a function $f$ on an empty interval is a vacuous truth. It is sufficient to show that $\phi_{n,\boldsymbol{\alpha},\boldsymbol{\gamma}_n}'''(u) \geq 0$ for every $u \in [0,1] \setminus \cup_{i=1}^p (\gamma_i-\epsilon,\gamma_i+\epsilon)$. Since $\phi_{n,\boldsymbol{\alpha},\boldsymbol{\gamma}_n}'''(u)=\sum_{i=1}^p \alpha_i\,\phi_{n,\gamma_{i,n}}'''(u)$, for every $u \in [0,1]$, the result follows if we show that $\phi_{n,\gamma_{i,n}}'''$ is nonnegative outside the interval $(\gamma_i-\epsilon,\gamma_i+\epsilon) \medcap [0,1]$, for every $i \in \left\{1,2,\ldots,p\right\}$. Let us fix $i \in \left\{1,2,\ldots,p\right\}$. Then simple algebra shows that
	\begin{equation}\label{2-eq-mixture-phi-n-gamma-i-n-third-derivative}
		\phi_{n,\gamma_{i,n}}'''(u)=n(n-1)(n-2)\binom{n-1}{\left[(n-1)\gamma_{i,n}\right]}u^{\left[(n-1)\gamma_{i,n}\right]-2}(1-u)^{n-\left[(n-1)\gamma_{i,n}\right]-3}(u-\alpha_n)(u-\beta_n),
	\end{equation}
	where $\alpha_n$ and $\beta_n$ are the two roots of the equation
	\[
	u^2-\frac{\left[(n-1)\gamma_{i,n}\right]}{n-2}\left(1+\frac{n-3}{n-1}\right)u+\frac{\left[(n-1)\gamma_{i,n}\right](\left[(n-1)\gamma_{i,n}\right]-1)}{(n-1)(n-2)}=0.
	\]
	Then, we have $\alpha_n=s_n-t_n$ and $\beta_n=s_n+t_n$, where
	$s_n=\frac{1}{2}\left(\frac{\left[(n-1)\gamma_{i,n}\right]}{n-2}\right)\left(1+\frac{n-3}{n-1}\right)$ and
	\[
	t_n=\frac{1}{2}\left\{\left(\frac{\left[(n-1)\gamma_{i,n}\right]}{n-2}\right)^2\left(1+\frac{n-3}{n-1}\right)^2-4\left(\frac{\left[(n-1)\gamma_{i,n}\right]}{n-2}\right)\left(\frac{\left[(n-1)\gamma_{i,n}\right]-1}{n-1}\right)\right\}^{1/2}.
	\]
	It can be checked that $t_n \geq 0$ and $\lim_{n \to \infty} t_n=0$. Also, $\lim_{n \to \infty} s_n=\gamma_i$. Thus, $\alpha_n<\beta_n$ and both the roots converge to $\gamma_i$, as $n \to \infty$. Hence, for every $\epsilon>0$, there exists $N_i \in \mathbb{N}$ such that $(\alpha_n,\beta_n) \subset (\gamma_i-\epsilon,\gamma_i+\epsilon)$ whenever $n \geq N_i$. From \eqref{2-eq-mixture-phi-n-gamma-i-n-third-derivative}, we know that $\phi_{n,\gamma_{i,n}}'''(u)<0$ if and only if $u \in (\alpha_n,\beta_n)$. Thus, $\phi_{n,\gamma_{i,n}}'''$ is nonnegative outside $(\gamma_i-\epsilon,\gamma_i+\epsilon) \medcap [0,1]$. By choosing $N=\max{\left\{N_1,N_2,\ldots,N_p\right\}}$ and noting that the choice of $\epsilon>0$ is arbitrary, the proof follows.
\end{proof}\vs

Now we can state the main result on the departure-based asymptotic stochastic order between mixtures of order statistics, arising from two different homogeneous samples.

\begin{theorem}\label{2-theorem-mixture}
	Let $X_{n,\boldsymbol{\alpha},\boldsymbol{\gamma}_n}$ and $Y_{n,\boldsymbol{\alpha},\boldsymbol{\gamma}_n}$ be two random variables with respective cdfs $F_{X,n,\boldsymbol{\alpha},\boldsymbol{\gamma}_n}$ and $F_{Y,n,\boldsymbol{\alpha},\boldsymbol{\gamma}_n}$, for every $n \in \mathbb{N}$. Assume that the baseline distributions $F_X$ and $F_Y$ satisfy \textnormal{(A1)} and \textnormal{(A2)}, as stated in the setup of \cref{2-theorem-general}. Then $X_{n,\boldsymbol{\alpha},\boldsymbol{\gamma}_n} \leq_{\textnormal{d-ast}} Y_{n,\boldsymbol{\alpha},\boldsymbol{\gamma}_n}$, as $n \to \infty$. \hfill $\blacksquare$
\end{theorem}

Note that the theorem follows if $\left\{\phi_{n,\boldsymbol{\alpha},\boldsymbol{\gamma}_n} : n \in \mathbb{N}\right\}$ satisfies \textnormal{(C1)--(C4)} in the setup of \cref{2-theorem-general}. From \cref{2-proposition-phi_n-r}, we see that $\phi_{n,\gamma_{i,n}}$ is strictly increasing and continuously differentiable. Since both the properties are closed under addition and positive scalar multiplication, it immediately follows that $\phi_{n,\boldsymbol{\alpha},\boldsymbol{\gamma}_n}$ satisfies \textnormal{(C1)} and \textnormal{(C2)}. \cref{2-proposition-item-discrete-limit-p-mixture-delta-i-ratio} and \cref{2-lemma-mixture-convex-outside-epsilon} respectively show that $\left\{\phi_{n,\boldsymbol{\alpha},\boldsymbol{\gamma}_n} : n \in \mathbb{N}\right\}$ satisfies \textnormal{(C3)} and \textnormal{(C4)}.

\subsection*{Record values}\label{2-subsection-record-values}
The concept of record values was introduced by \citet{C_1952}. Let $\left\{X_n : n\in \mathbb{N}\right\}$ be a sequence of iid random variables, each following the distribution $F_X$. The first observation $X_1$ is a record by default. The next upper record occurs when the realized value of an earliest succeeding random variable exceed that of $X_1$. This intuition leads to the sequence $\{T_n : n \in \mathbb{N}\}$ of upper record times, defined by $T_1=1$ and $T_{n+1}=\min{\{i \in \mathbb{N} : X_i>X_{T_n}\}}$, for every $n \in \mathbb{N}$. The corresponding sequence of upper record values $\left\{R_n : n \in \mathbb{N}\right\}$ is defined by $R_n=X_{T_n}$, for every $n \in \mathbb{N}$. For more details, see \citet{AN_2011}. The notion was generalized by \citet{DK_1976} to $k$th record values, which considers records in terms of $k$th highest value in the sequence, where $k$ is a positive integer. The sequence $\{T_n^{(k)} : n \in \mathbb{N}\}$ of upper $k$th record times is defined by
\[
T_1^{(1)}=1 \text{ and } T_{n+1}^{(k)}=\min{\left\{i \in \mathbb{N} : X_{i:i+k-1}>X_{T_n^{(k)}:T_n^{(k)}+k-1}\right\}}, \text{ for every } n \in \mathbb{N}.
\]
The corresponding sequence of upper $k$th record values $\{R_n^{(k)} : n \in \mathbb{N}\}$ is defined by $R_n^{(k)}=X_{T_n^{(k)}}$, for every $n \in \mathbb{N}$. In this work, we shall focus on upper $k$th record value, although analogous result for lower $k$th record value, which is defined similarly, can be easily obtained. We consider the sequence $\left\{\phi_n : n \in \mathbb{N}\right\}$ of distortion functions given by
\begin{equation}\label{2-eq-kth-record-value-phi-n}
\phi_n(u)=\frac{1}{(n-1)!}\int_0^{-k\ln{(1-u)}} y^{n-1} e^{-y} dy,
\end{equation}
when $0 \leq u<1$, and $\phi_n(1)=\lim_{u \to 1-} \phi_n(u)$. Then the cdf of $R_n^{(k)}$ can be written as $\phi_n(F_X(x))$ (see Corollary $2$ of \citet{DK_1976}). Now we state the main result on record values.

\begin{theorem}\label{2-theorem-kth-record-value}
	Let $F_X$ and $F_Y$ be two cdfs satisfying $F_X(x)<F_Y(x)$ for every $x>c$, for some $c \in \mathbb{R}$. Also let $\left\{X_n : n \in \mathbb{N}\right\}$ and $\left\{Y_n : n \in \mathbb{N}\right\}$ be two sequences of iid random variables with respective parent cdfs $F_X$ and $F_Y$, and respective sequences of $k$th record values $\{R_n^{(k)} : n \in \mathbb{N}\}$ and $\{S_n^{(k)} : n \in \mathbb{N}\}$. Then $R_n^{(k)} \leq_{\textnormal{d-ast}} S_n^{(k)}$ for every $k \in \mathbb{N}$. \hfill $\blacksquare$
\end{theorem}\vs

Observe that the sequence $\left\{\phi_n : n \in \mathbb{N}\right\}$ satisfies \textnormal{(C1)--(C4)} (as stated in the setup of \cref{2-theorem-general}). The proof of \cref{2-theorem-kth-record-value} then follows from \cref{2-theorem-general}.\vs

{\bf Acknowledgement.} The first author is supported by Senior Research Fellowship from University Grants Commission, Government of India.

\baselineskip=16pt

\bibliographystyle{newapa}
\bibliography{references}

\begin{thebibliography}{}

\bibitem[\protect\citeauthoryear{Ahsanullah \& Nevzorov}{Ahsanullah \&
  Nevzorov}{2011}]{AN_2011}
Ahsanullah, M. \& Nevzorov, V.~B. (2011).
\newblock Record Statistics. In: {\it International Encyclopedia of Statistical
  Science}, Lovric, M. (Ed.). 1195--1202. Springer Berlin Heidelberg.

\bibitem[\protect\citeauthoryear{Chandler}{Chandler}{1952}]{C_1952}
Chandler, K.~N. (1952).
\newblock The distribution and frequency of record values.
\newblock {\em Journal of the Royal Statistical Society. Series B}, {\em
  14\/}(2), 220--228.

\bibitem[\protect\citeauthoryear{Darwin}{Darwin}{1878}]{D_1878}
Darwin, C. (1878).
\newblock {\em The Effect of Cross- and Self-fertilization in the Vegetable
  Kingdom\/} (2nd ed.).
\newblock John Murray, London.

\bibitem[\protect\citeauthoryear{del Barrio, Cuesta-Albertos \& Matr{\'a}n}{del
  Barrio et~al.}{2018}]{B_2018}
del Barrio, E., Cuesta-Albertos, J.~A., \& Matr{\'a}n, C. (2018).
\newblock An Optimal Transportation Approach for Assessing Almost Stochastic
  Order. In: {\it The Mathematics of the Uncertain: A Tribute to Pedro Gil},
  Gil, E., Gil, E., Gil, J. and Gil, M. (Eds.), 33--44. Springer International
  Publishing.

\bibitem[\protect\citeauthoryear{Denneberg}{Denneberg}{1994}]{D_1994}
Denneberg, D. (1994).
\newblock {\em Non-Additive Measure and Integral}.
\newblock Springer, Netherlands.

\bibitem[\protect\citeauthoryear{Dziubdziela \& Kopociński}{Dziubdziela \&
  Kopociński}{1976}]{DK_1976}
Dziubdziela, W. \& Kopociński, B. (1976).
\newblock Limiting properties of the k-th record values.
\newblock {\em Applicationes Mathematicae}, {\em 15\/}(2), 187--190.

\bibitem[\protect\citeauthoryear{Ghosh \& Nanda}{Ghosh \&
  Nanda}{2021}]{GN_2021_A}
Ghosh, S. \& Nanda, A.~K. (2021).
\newblock {A}symptotic {S}tochastic {C}omparison of {R}andom {P}rocesses.
\newblock \url{arXiv:2103.01720}.

\bibitem[\protect\citeauthoryear{Gross \& Holland}{Gross \&
  Holland}{1968}]{GH_1968}
Gross, S. \& Holland, P. (1968).
\newblock The distribution of galton's statistic.
\newblock {\em The Annals of Mathematical Statistics}, {\em 39\/}(6),
  2114--2117.

\bibitem[\protect\citeauthoryear{Hodges}{Hodges}{1955}]{H_1955}
Hodges, J.~L., J. (1955).
\newblock Galton's rank order test.
\newblock {\em Biometrika}, {\em 42}, 261--262.

\bibitem[\protect\citeauthoryear{Lehmann}{Lehmann}{1955}]{L_1955}
Lehmann, E.~L. (1955).
\newblock Ordered families of distributions.
\newblock {\em Annals of Mathematical Statistics}, {\em 26\/}(3), 399--419.

\bibitem[\protect\citeauthoryear{Leshno \& Levy}{Leshno \&
  Levy}{2002}]{LL_2002}
Leshno, M. \& Levy, H. (2002).
\newblock Preferred by “all” and preferred by “most” decision makers:
  Almost stochastic dominance.
\newblock {\em Management Science}, {\em 48\/}(8), 1074--1085.

\bibitem[\protect\citeauthoryear{Samaniego}{Samaniego}{2007}]{S_2007}
Samaniego, F.~J. (2007).
\newblock {\em {System Signatures and their Applications in Engineering
  Reliability}}.
\newblock Springer, New York.

\bibitem[\protect\citeauthoryear{Shaked \& Shanthikumar}{Shaked \&
  Shanthikumar}{2007}]{SS_2007}
Shaked, M. \& Shanthikumar, J.~G. (2007).
\newblock {\em Stochastic Orders}.
\newblock Springer, New York.

\end{thebibliography}


\section{Appendix}\label{2-section-appendix}

\begin{proof}[Proof of \cref{2-lemma-A_0-B_0}]
	Note that, for any cdf $F$, we have $F(F^{-1}(u)) \geq u$, for every $u \in (0,1)$, where equality holds if $F$ is continuous and $F^{-1}(F(x)) \leq x$, for every $x \in \mathbb{R}$, where equality holds if $F$ is strictly increasing. Let $x \in B_0$. Then $F_X(x)<F_Y(x)$. Suppose, for the sake of contradiction, that $F_X(x) \notin A_0$. Then $F_Y^{-1}(F_Y(x))=x=F_X^{-1}(F_X(x)) \leq F_Y^{-1}(F_X(x))$. By nondecreasingness of $F_Y$, we have $F_X(x) \geq F_Y(x)$, a contradiction. Hence $F_X(x) \in A_0$. Since $x \in B_0$ is arbitrary, we have $F_X(B_0) \subseteq A_0$. Again, starting with $x \in B_0$ and $F_Y(x) \notin A_0$, we similarly have $A_0 \subseteq F_X(B_0)$. Hence $A_0=F_X(B_0)$. Similarly, $A_0=F_Y(B_0)$. Interchanging the role of $F_X$ and $F_Y$, we have $A_1=F_X(B_1)=F_Y(B_1)$. Again, $A_2=A_0 \cup A_1=F_X(B_0) \cup F_X(B_1)=F_X(B_0 \cup B_1)=F_X(B_2)$. In a similar manner, we obtain $A_2=F_Y(B_2)$.
\end{proof}\vs

\begin{proof}[Proof of \cref{2-proposition-phi_n-r-2}]
	Let us denote the fractional part of $(n-1)\gamma_n$ by $\left\{(n-1)\gamma_n\right\}$. Then, we have
	\begin{equation*}
		\frac{\phi_{n,\gamma_n}'(t)}{\phi_{n,\gamma_n}'(s)}=\left(\frac{s(1-t)}{t(1-s)}\right)^{\left\{(n-1)\gamma_n\right\}}\left(\left(\frac{t}{s}\right)^{\gamma_n} \left(\frac{1-t}{1-s}\right)^{1-\gamma_n}\right)^{n-1}.
	\end{equation*}
	Considering the cases $t>s$, $t<s$ and $t=s$ separately, it is easy to see that
	\[
	\left(\frac{s(1-t)}{t(1-s)}\right)^{\left\{(n-1)\gamma_n\right\}} \leq \left(\frac{s(1-t)}{t(1-s)}\right)^{\frac{1}{2}-\frac{1}{2}\,\cdot\,\sign{(t-s)}},
	\]
	with equality if and only if $t=s$. Again,
	\[
	\left(\left(\frac{t}{s}\right)^{\gamma_n} \left(\frac{1-t}{1-s}\right)^{1-\gamma_n}\right)^{n-1}=\left(\left(\frac{t}{s}\right)^{\gamma} \left(\frac{1-t}{1-s}\right)^{1-\gamma}\right)^{n-1} \left(\frac{t(1-s)}{s(1-t)}\right)^{(n-1)(\gamma_n-\gamma)}.
	\]
	Since, $\lvert \gamma_n-\gamma \rvert=O(1/n)$, there exists $K>0$ and $N \in \mathbb{N}$ such that, for every $n \geq N$, $\lvert \gamma_n-\gamma \rvert \leq K/n$, which in turn implies that $\lvert \gamma_n-\gamma \rvert \leq K/(n-1)$, i.e. $-K \leq (n-1)(\gamma_n-\gamma) \leq K$. Considering the cases $t<s$, $t>s$ and $t=s$ separately, we can see that
	\[
	\left(\frac{t(1-s)}{s(1-t)}\right)^{(n-1)(\gamma_n-\gamma)} \leq \left(\frac{t(1-s)}{s(1-t)}\right)^{K\,\cdot\,\sign{(t-s)}}.
	\]
	Hence,
	\begin{align*}
		\frac{\phi_{n,\gamma_n}'(t)}{\phi_{n,\gamma_n}'(s)} &\leq \left(\frac{s(1-t)}{t(1-s)}\right)^{\frac{1}{2}-\frac{1}{2}\,\cdot\,\sign{(t-s)}} \left(\frac{t(1-s)}{s(1-t)}\right)^{K\,\cdot\,\sign{(t-s)}} \left(\left(\frac{t}{s}\right)^{\gamma} \left(\frac{1-t}{1-s}\right)^{1-\gamma}\right)^{n-1}\\
		&=\left(\frac{s(1-t)}{t(1-s)}\right)^{\frac{1}{2}-(K+\frac{1}{2})\,\cdot\,\sign{(t-s)}} \left(\left(\frac{t}{s}\right)^{\gamma} \left(\frac{1-t}{1-s}\right)^{1-\gamma}\right)^{n-1}.
	\end{align*}
	Now, the proof follows by choosing $C(t,s)=\left(\frac{s(1-t)}{t(1-s)}\right)^{\frac{1}{2}-(K+\frac{1}{2})\,\cdot\,\sign{(t-s)}}$.
\end{proof}\vs

\begin{proof}[Proof of \cref{2-proposition-asymptotics-binomial}]
	To prove the first part, note that
	\begin{align}\label{2-eq-lemma-asymptotics-binomial-Z}
		\lim_{n \to \infty} P\left\{B_{n,u} \geq 1+\left[(n-1)\gamma_n\right]\right\}&=\lim_{n \to \infty} P\left\{B_{n,u} \geq \frac{1}{2}+\left[(n-1)\gamma_n\right]\right\}\nonumber\\
		&=\lim_{n \to \infty} P\left\{Z > \frac{\frac{1}{2}+\left[(n-1)\gamma_n\right]-nu}{\sqrt{nu(1-u)}}\right\},
	\end{align}
	where the last equality is due to CLT and $Z$ has the standard normal distribution. Now, using the hypothesis $\lvert \gamma_n-\gamma \rvert=O(1/n)$, we have
	\begin{align*}
		\frac{\frac{1}{2}+\left[(n-1)\gamma_n\right]-nu}{\sqrt{nu(1-u)}}=O(\sqrt{n})(\gamma-u)+O(1/\sqrt{n}),
	\end{align*}
	Denoting the cdf of standard normal distribution by $\Phi$, we have from \eqref{2-eq-lemma-asymptotics-binomial-Z} that
	\[
	\lim_{n \to \infty} P\left\{B_{n,u} \geq 1+\left[(n-1)\gamma_n\right]\right\}=1-\lim_{n \to \infty} \Phi(O(\sqrt{n})(\gamma-u)+O(1/\sqrt{n})),
	\]
	which essentially implies \eqref{2-eq-asymptotic-binomial-sf} by continuity of $\Phi$. To prove the second part, note that
	\begin{align*}
		&\lim_{n \to \infty} nP\left\{B_{n-1,u}=\left[(n-1)\gamma_n\right]\right\}=\lim_{n \to \infty} nP\left\{\left[(n-1)\gamma_n\right]-\frac{1}{2}<B_{n-1,u}<\left[(n-1)\gamma_n\right]+\frac{1}{2}\right\}\\
		&=\lim_{n \to \infty} nP\left\{\frac{\left[(n-1)\gamma_n\right]-\frac{1}{2}-(n-1)u}{\sqrt{(n-1)u(1-u)}}<\frac{B_{n-1,u}-(n-1)u}{\sqrt{(n-1)u(1-u)}}<\frac{\left[(n-1)\gamma_n\right]+\frac{1}{2}-(n-1)u}{\sqrt{(n-1)u(1-u)}}\right\}.
	\end{align*}
	Using CLT, the above expression becomes
	\begin{align*}
		&\lim_{n \to \infty} n\left\{\Phi\left(\frac{\left[(n-1)\gamma_n\right]+\frac{1}{2}-(n-1)u}{\sqrt{(n-1)u(1-u)}}\right)-\Phi\left(\frac{\left[(n-1)\gamma_n\right]-\frac{1}{2}-(n-1)u}{\sqrt{(n-1)u(1-u)}}\right)\right\}\nonumber\\
		&=\lim_{n \to \infty} n\left\{\Phi\left(\frac{(n-1)(\gamma_n-u)+\frac{1}{2}-\left\{(n-1)\gamma_n\right\}}{\sqrt{(n-1)u(1-u)}}\right)-\Phi\left(\frac{(n-1)(\gamma_n-u)-\frac{1}{2}-\left\{(n-1)\gamma_n\right\}}{\sqrt{(n-1)u(1-u)}}\right)\right\}\nonumber\\
		&=\lim_{n \to \infty} n\Bigg\{\Phi\left(\sqrt{\frac{n-1}{u(1-u)}}\left\{\gamma_n-u+\frac{1}{2(n-1)}-\frac{\left\{(n-1)\gamma_n\right\}}{n-1}\right\}\right)\nonumber\\
		&\phantom{=\lim_{n \to \infty} n}-\Phi\left(\sqrt{\frac{n-1}{u(1-u)}}\left\{\gamma_n-u-\frac{1}{2(n-1)}-\frac{\left\{(n-1)\gamma_n\right\}}{n-1}\right\}\right)\Bigg\}\\
		&=\lim_{n \to \infty} n \int_{l_n}^{u_n} \phi(t)dt,
	\end{align*}
	where $l_n=\sqrt{\frac{n-1}{u(1-u)}}\left\{\gamma_n-u-\frac{1}{2(n-1)}-\frac{\left\{(n-1)\gamma_n\right\}}{n-1}\right\}$ and $u_n=\sqrt{\frac{n-1}{u(1-u)}}\left\{\gamma_n-u+\frac{1}{2(n-1)}-\frac{\left\{(n-1)\gamma_n\right\}}{n-1}\right\}$ and $\phi$ denotes the pdf of standard normal distribution. We consider the three cases $u<\gamma$, $u>\gamma$ and $u=\gamma$ separately.\vs
	
	\ni{\bf Case 1:} $u<\gamma$. Using the hypothesis $\lvert \gamma_n-\gamma \rvert=O(1/n)$, we have
	\begin{equation}\label{2-eq-lemma-binomial-approximation-l-n}
		l_n=(\gamma-u)\sqrt{\frac{n-1}{u(1-u)}}+O(1/\sqrt{n}).
	\end{equation}
	Thus, there exists $N_1 \in \mathbb{N}$ such that $0<l_n<u_n$ whenever $n \geq N_1$. Then $n\int_{l_n}^{u_n} \phi(t)dt \leq n(u_n-l_n) \phi(l_n)$, for every $n \geq N_1$. Now, $n(u_n-l_n)=n/\sqrt{(n-1)u(1-u)}$. Also, from \eqref{2-eq-lemma-binomial-approximation-l-n}, we obtain $l_n \geq (\gamma-u)\sqrt{n}/\sqrt{2u(1-u)}+O(1/\sqrt{n})$. Thus, there exists a natural number $N_2$ large enough, such that $l_n \geq (\gamma-u)\sqrt{n}/\sqrt{4u(1-u)}$, whenever $n \geq N_2$. Letting $k_1=(\gamma-u)/\sqrt{4u(1-u)}$, we have $\phi(l_n)=(\sqrt{2\pi})^{-1} \exp{\left(-l_n^2/2\right)} \leq (\sqrt{2\pi})^{-1} \exp{\left(-k_1^2 n/2\right)}$, for every $n \geq N_2$. Thus, for every $n \geq \max{\{N_1,N_2\}}$, we have $n(u_n-l_n) \phi(l_n) \leq n\exp{\left(-k_1^2 n/2\right)}/\sqrt{2\pi(n-1)u(1-u)}$, which goes to $0$, as $n \to \infty$.\vs
	
	\ni{\bf Case 2:} $u>\gamma$. Let $k_2=2(u-\gamma)/u(1-u)$. On using the same chain of arguments as in Case $1$, we obtain $n(u_n-l_n) \phi(u_n) \leq n\exp{\left(-k_2^2 n/2\right)}/\sqrt{2\pi(n-1)u(1-u)}$, which goes to $0$, as $n \to \infty$.\vs
	
	\ni{\bf Case 3:} $u=\gamma$. Since $\phi$ is continuous, by \href{https://en.wikipedia.org/wiki/Mean_value_theorem\#First_mean_value_theorem_for_definite_integrals}{integral mean value theorem}, there exists $m_n \in \left[l_n,u_n\right]$ such that
	\[
	\int_{l_n}^{u_n} \phi(t)dt=(u_n-l_n)\phi(m_n).
	\]
	Clearly, for every $n \in \mathbb{N}$, there exists $\alpha_n \in \left[-1,1\right]$ such that
	\[
	m_n=\sqrt{\frac{n-1}{u(1-u)}}\left\{\gamma_n-u+\frac{\alpha_n}{2(n-1)}-\frac{\left\{(n-1)\gamma_n\right\}}{n-1}\right\}.
	\]
	Observing that $u=\gamma$ and $\lvert \gamma_n-\gamma \rvert=O(1/n)$, we have $m_n=O(1/\sqrt{n})$, which goes to zero, as $n \to \infty$. By continuity of $\phi$, we have $\phi(m_n) \to \phi(0)=1/\sqrt{2\pi}$, as $n \to \infty$. Hence, $\lim_{n \to \infty} n\int_{l_n}^{u_n} \phi(t)dt=\infty$, which completes the proof.
\end{proof}

\end{document}